\documentclass[11pt]{article}
\usepackage[utf8]{inputenc}
\usepackage{fullpage}

\usepackage{amsthm,mathtools,scalerel}
\usepackage{color,soul,latexsym,amsmath,amssymb,amsfonts,amsthm,dsfont,enumitem,xcolor,bbm,threeparttable,graphicx,float,subcaption}
\usepackage{authblk}
\usepackage{natbib}
\PassOptionsToPackage{hyphens}{url}
\RequirePackage[colorlinks=true,allcolors=blue,hypertexnames=false]{hyperref}%
\usepackage[symbol]{footmisc}
\usepackage{graphicx,todonotes}
\usepackage{xcolor,stmaryrd}
\usepackage{enumitem}
\usepackage{booktabs}
\usepackage{wrapfig}
\definecolor{Green}{HTML}{069937}
\newif \ifnotes

\usepackage[most]{tcolorbox}
\colorlet{shadecolor}{gray!15}
\usepackage{outlines}
\usepackage{array}
\usepackage{selectp}
\usepackage{cleveref}
\usepackage[ruled,linesnumbered, vlined, noend]{algorithm2e}
\usepackage{titlesec}
\usepackage{mathrsfs}

\newtheorem{theorem}{Theorem}%[section]
\newcommand{\Ll}{\mathscr{L}}
\newtheorem{lemma}{Lemma}
\newtheorem{corollary}{Corollary}
\newtheorem{proposition}{Proposition}

\newcommand{\E}{\mathbb{E}}

\renewcommand{\hat}{\widehat}
\let\tilde\widetilde

\newcommand{\for}{\overrightarrow{T}}
\newcommand{\back}{\overleftarrow{T}}

\newcommand{\RR}{\mathbb{R}}

\newcommand{\B}{\mathcal{B}}

\newcommand{\KL}{\mathsf{KL}}
\newcommand{\ud}{\mathrm{d}}

\newcommand\Var{\mathrm{Var}}

\newcommand{\Pb}{\mathbb{P}}

\iftrue  % \iftrue or \iffalse
  \newcommand{\gm}[1]{[\textcolor{red}{GM: #1}] }
  \newcommand{\tg}[1]{[\textcolor{blue}{TG: #1}] }
\else
  \newcommand{\gm}[1]{}
  \newcommand{\tg}[1]{}
\fi

\allowdisplaybreaks
\usepackage{tikz}
\usetikzlibrary{positioning,shapes.geometric,calc,shapes,arrows.meta,arrows,decorations.markings, external, trees}
\tikzstyle{Arrow} = [thick,decoration={markings,mark=at position 1 with {\arrow[thick]{latex}}},shorten >= 3pt, preaction = {decorate}]

\title{\bf Statistical Rates for Entropic Optimal Transport in the Discrete to SubGaussian Regime}
\author[1]{Tomas Gonzalez}
\author[2]{Gonzalo Mena}
\affil[1]{Department of Machine Learning, Carnegie Mellon University}
\affil[2]{Department of Statistics \& Data Science, Carnegie Mellon University}
\affil[ ]{\texttt{\{tcgonzal, gmena\}@andrew.cmu.edu}}

\date{}                   
\begin{document}

\maketitle

\begin{abstract}
We study statistical rates in entropic optimal transport in the semi-discrete regime where one measure has finite support and the other is subGaussian. Our main result establishes parametric convergence rates for the empirical dual potentials to their population counterparts, with no dimension dependence in the leading term. Our result relies on tailored strong concavity analysis of the semi-dual objective, coupled with specialized bounds for the semi-discrete potentials. As a consequence, we obtain fast rates for downstream quantities derived from the optimal coupling. Chiefly,  the empirical barycentric projection achieves a squared-error rate $n^{-1}$, matching the fully compact case and improving over the less favorable $n^{-1/2}$ rate known for fully subGaussian settings. Altogether, these results may indicate a lower complexity adaptation phenomenon whereby the statistical complexity of the barycentric projection is governed by the discrete measure. As an application, we analyze Sinkhorn-EM, an EM-type algorithm in which the E-step is replaced by an entropic optimal transport problem. In a well-specified and balanced two-component Gaussian mixture model, we prove $\sqrt{n}$-consistency of the empirical iterates to their population counterparts for any fixed number of iterations, matching classical EM rates up to a $\sqrt{\log n}$ factor. Simulations support the theory.
\end{abstract}

\section{Introduction}

Optimal transport (OT) provides a geometrically meaningful way to compare probability distributions. Given measures $P$ and $Q$ and a cost function, the problem is to find the coupling---a joint distribution with marginals $P$ and $Q$---that minimizes the expected transport cost. The optimal coupling describes how mass is matched, while, under suitable conditions, the Brenier map provides a deterministic way to transport $P$ onto $Q$. The optimal transport cost, coupling, and map have become central tools across statistics and machine learning, with applications ranging from generative modeling and domain adaptation to single-cell genomics, shape analysis, and distributional robustness \citep{peyre2019computational,villani2009optimal}. In most applications, however, at least one of the underlying distributions is typically unknown and must be inferred from finite samples.

Statistical optimal transport studies the accuracy with which transport costs, couplings, and maps can be estimated from such samples. For classical OT, these estimation problems generally suffer from the curse of dimensionality: convergence rates deteriorate rapidly with the ambient dimension $d$, and computationally efficient estimators with sharp guarantees are scarce \citep{hutter2021minimax}. Entropic optimal transport (EOT) is an attractive alternative that preserves the geometric structure of OT while adding a relative-entropy penalty to the transport objective. As the strength of this regularization vanishes, EOT recovers classical OT under suitable conditions. For positive regularization, it produces a smooth optimal coupling, can be computed efficiently using Sinkhorn's algorithm, and admits estimators with substantially more favorable sample complexity \citep{cuturi2013sinkhorn,peyre2019computational}. Its barycentric projections provide entropic analogues of the Brenier map and, together with the optimal coupling and dual potentials, are important statistical objects in their own right.

Existing finite-sample guarantees for estimating these objects depend strongly on the support assumptions imposed on the two measures. When both measures are compactly supported, Rigollet and Stromme \citep{rigollet2025sample} establish parametric squared-error rates of order $n^{-1}$ for both the empirical coupling density and its barycentric projections, where $n$ denotes the sample size. To the best of our knowledge, an $n^{-1}$ expected squared-error bound for the empirical coupling density has not previously been established in the semi-discrete-to-subGaussian setting considered here. The literature on barycentric projections extends further: a parametric rate is known in the compact semi-discrete setting \citep{pooladian2023minimax}, while Masud et al.\ \citep{masud2023multivariate} obtain an $n^{-1/2}$ squared-error rate when the source measure is subGaussian and the target measure is compactly supported. When both measures may be subGaussian and unbounded, Werenski et al.\ \citep{werenski2023estimation} establish slower rates. This leaves open a natural intermediate question:
\begin{center}
\textit{Can parametric rates for the coupling density and barycentric projections be recovered when one measure has finite support while the other is subGaussian and unbounded?}
\end{center}

We answer this question affirmatively. Our central result establishes dimension-independent convergence rates for the empirical dual potentials in the discrete-to-subGaussian regime. This potential-level result yields a parametric squared-error rate for the density of the empirical optimal coupling and, in turn, parametric convergence rates for both barycentric projections. Our analysis exploits the discrete marginal, which allows the corresponding dual potential to be represented by a finite-dimensional vector. Our proof strategy is closest to that of Pooladian et al.\ \citep{pooladian2023minimax}, combining strong concavity of the semi-dual objective with empirical-process arguments. A key ingredient in this approach is uniform control of the dual potentials, for which existing bounds do not extend suitably to an unbounded subGaussian marginal. 

Beyond its fundamental importance in the theory of statistical EOT, the discrete-to-subGaussian regime arises naturally in the analysis of the Sinkhorn expectation-maximization (Sinkhorn-EM) algorithm \citep{mena2020sinkhorn,mena2026eotclustering} and is closely related to the setup studied in the classical work on deterministic annealing for clustering \citep{rose1998deterministic}. For Gaussian mixture models, Sinkhorn-EM replaces the usual likelihood-based update with an EOT problem between the discrete measure supported on the current mixture centers and the continuous data distribution. While previous work has demonstrated its practical relevance and studied its population behavior, its finite-sample properties remain poorly understood \citep{mena2020sinkhorn,mena2026eotclustering}. Our EOT results yield the first guarantee on how accurately sample-based Sinkhorn-EM iterates approximate their population counterparts, providing a first step toward a broader theoretical understanding of the algorithm.

\paragraph{Contributions.}
We make the following contributions. 

\begin{itemize}
    \item In \Cref{teo:potentialbound}, we establish $n^{-1}$ expected squared-error rates for the empirical entropic dual potentials in the discrete-to-subGaussian regime, with dimension-independent leading constants and an exponentially small remainder.

    \item We prove an $n^{-1}$ squared-error rate for the empirical optimal coupling density (\Cref{theo:densitibound}), as well as rates of order $n^{-1}$ and $d/n$ for the backward and forward barycentric projections, respectively (\Cref{theo:barybound}). A matching minimax lower bound shows that the $n^{-1}$ rate for the backward projection is optimal (\Cref{prop:minimax_lower_bound}).

    \item We establish a strong-concavity bound for the semi-dual objective and a dimension-independent uniform bound for the discrete-side dual potential (\Cref{prop:strong,prop:potentialbound}), which are key to obtain the results mentioned above.

    \item As an application, we give the first finite-sample guarantee for Sinkhorn-EM in a balanced, symmetric two-component Gaussian mixture, proving $n^{-1/2}$-consistency of the empirical iterates with their population counterparts for any fixed number of iterations (\Cref{thm:convergence-sample-SEM}).

    \item We corroborate the predicted convergence rates through simulations in \Cref{sec:simulations}.
\end{itemize}

\paragraph{Related work.}
The statistical theory of optimal transport is extensive; we refer to
\cite{chewi2025statistical} for a comprehensive treatment. In the unregularized
setting, dimension-dependent rates are known for both Wasserstein distances and
transport maps. In particular, H\"utter and Rigollet
\citep{hutter2021minimax} study estimation of the unregularized Brenier map
between absolutely continuous measures under regularity assumptions on the
transport map and smoothness assumptions on the densities. Under sufficient
smoothness, their squared-$L^2$ risk can attain an $n^{-1/2}$ rate whose
exponent is independent of the ambient dimension. This is a different
statistical problem from the one considered here: we study fixed entropic
regularization with a finitely supported marginal and an unbounded subGaussian
marginal, and obtain guarantees for the dual potentials, coupling density, and
barycentric projections without smoothness assumptions on the densities.

Statistical analyses of entropic optimal transport, popularized by Cuturi
\citep{cuturi2013sinkhorn}, initially focused primarily on transport costs,
establishing sample-complexity bounds and central limit theorems under
compact-support and subGaussian assumptions
\citep{genevay2019sample,mena2019statistical}. Statistical guarantees for finer
EOT objects are more limited. When both marginals are compactly supported,
Rigollet and Stromme \citep{rigollet2025sample} establish parametric
squared-error rates of order $n^{-1}$ for the empirical dual potentials,
coupling density, and barycentric projections. In the compact semi-discrete
setting, Pooladian et al.\ \citep{pooladian2023minimax} obtain an $n^{-1}$ rate
for the empirical EOT barycentric projection at fixed regularization, as an
intermediate step toward estimating the unregularized Brenier map. When the
source is subGaussian and the target is compactly supported, Masud et
al.\ \citep{masud2023multivariate} obtain an $n^{-1/2}$ squared-error rate.
More general noncompact settings are considered by Werenski et
al.\ \citep{werenski2023estimation}, who obtain slower rates under compactness
or strong log-concavity assumptions on the target.
Table~\ref{tab:eot-map-comparison} summarizes the available finite-sample rates
for estimating the EOT barycentric projection at fixed regularization and
situates our result within this literature. To the best of our knowledge, an $n^{-1}$ expected squared-error bound for the (canonically extended) empirical coupling density has not previously been established in this semi-discrete-to-subGaussian regime.

\begin{table}[t]
    \centering
    \small
    \begin{tabular}{@{}p{0.25\linewidth}p{0.25\linewidth}
                         p{0.30\linewidth}p{0.12\linewidth}@{}}
        \toprule
        Work & Source Measure & Target Measure & Squared-$L^2$ rate \\
        \midrule
        Ours
        & SubGaussian
        & Finite support, with masses bounded away from zero
        & $n^{-1}$ \\

        Ours (lower bound)
        & SubGaussian
        & Two atoms, each with mass at least $1/4$
        & $\Omega(n^{-1})$ \\

        Pooladian et al.\ \citep{pooladian2023minimax}
        & Compact, with density bounded above and below
        & Finite support, with masses bounded away from zero
        & $n^{-1}$ \\

        Masud et al.\ \citep{masud2023multivariate}
        & SubGaussian
        & Compact support
        & $n^{-1/2}$ \\

        Werenski et al.\ \citep{werenski2023estimation}
        & SubGaussian
        & Compact support or strongly log-concave
        & $n^{-1/3}$ \\

        Rigollet and Stromme \citep{rigollet2025sample}
        & Compact support
        & Compact support
        & $n^{-1}$ \\
        \bottomrule
    \end{tabular}
    \caption{Finite-sample rates for estimating the EOT barycentric projection
    at fixed regularization. We report only the dependence on the sample size
    $n$, treating the regularization parameter and other problem parameters as
    fixed. In our notation, the map goes from $Q$ to $P$. The lower bound is a minimax bound; all other rows report upper
    bounds.}
    \label{tab:eot-map-comparison}
\end{table}

Lower-complexity adaptation has been studied for transport costs in both the
unregularized and entropic settings. For classical OT costs, Hundrieser et
al.\ \citep{hundrieser2024empirical} prove that empirical rates can adapt to the
lower-complexity marginal. Groppe and Hundrieser
\citep{groppe2024lower} establish analogous adaptation results for EOT costs.
These results concern transport costs, whereas our semi-discrete analysis
concerns finer EOT objects, namely the dual potentials, coupling density, and
barycentric projections.

Finally, our Sinkhorn-EM application connects the paper to the literature on
expectation-maximization (EM) algorithms for mixture models. Finite-sample
analyses of EM for Gaussian mixtures are well developed but technically
delicate, relying on careful control of sample-based EM updates and their
deviations from population iterates
\citep{xu2016globalEM,wu2021randomly,BalWaiYu17,daskalakis2017ten,
weinberger2022algorithm,Dwivedi2018}. Sinkhorn-EM replaces the likelihood-based
EM update with an EOT-based update and has been studied primarily at the
population level
\citep{nejatbakhsh2020probabilistic,mena2020sinkhorn,
mena2026eotclustering}. We provide its first finite-sample guarantee in a
balanced, symmetric two-component Gaussian mixture. For any fixed number of
iterations, the sample-based iterates approximate their population counterparts
at an $n^{-1/2}$ rate, matching the corresponding sample-based EM guarantee up
to logarithmic factors and problem-dependent constants.

\section{Problem setup and notation}\label{sec:notation}

The entropic optimal transport problem between measures $P$ (target) and $Q$ (source) is written as  \cite{peyre2019computational,chewi2025statistical}
\begin{equation}\label{eqn:entropic_OT_primal}
S(P, Q) := \inf_{\pi \in \Pi(P,Q)}\Big\{ \int\int c(x,y)\ud \pi(x,y) + \sigma^2 \mathrm{KL}(\pi\| P \otimes Q)\Big\}\,,
\end{equation} 
where $c(x,y)=||x-y||^2/2$ is the quadratic cost, and $\sigma^2>0$ is a fixed regularization parameter. The infimum is attained by a unique $\pi^\star \in \Pi(P, Q)$ in the transportation polytope $\Pi(P,Q)$, the set of joint distributions with prescribed marginals $P$ and $Q$), and strong duality holds in the sense that \cite{peyre2019computational,chewi2025statistical}
\begin{equation}\label{eqn:entropic_OT_strong_duality}S(P, Q)
= \sup_{(f, g) \in L^{\infty}(P) \times L^{\infty}(Q)} \Phi(f,g),
\end{equation} where $\Phi(f,g)=\Phi(f,g,P,Q,\sigma)$ is defined as
\begin{equation}\label{eqn:phifull}  \Phi(f,g):=\int f(x)\ud P(x)+ \int g(y)\ud Q(y) - \sigma^2\int\int  
(e^{(f(x)+g(y)-c(x,y))/\sigma^2} - 1)\ud (P \otimes Q)(x,y) .\end{equation} The supremum in \eqref{eqn:entropic_OT_strong_duality} is attained at a pair $(f^\ast, g^\ast) \in L^{\infty}(P) \times L^{\infty}(Q)$
of dual potentials, which are unique up to the translation
$(f^\ast, g^\ast) \mapsto (f^\ast + a, g^\ast - a)$ for $a \in \RR $. To avoid degeneracies, we will always assume the following gauge constraint
\begin{equation}\label{eq:gauge}
\E_P(f^\ast(X))= 0.
\end{equation}
The first-order conditions for $f^\ast$,$g^\ast$ write as \cite{chewi2025statistical,cuturi2018semidual}
\begin{subequations}
    
\begin{align}
    \label{eqn:fstar_from_gstar} f^\ast(x)&=-\sigma^2 \ln \Big(\int e^{(g^\ast(y)-c(x,y))/\sigma^2} \ud Q(y)\Big)\quad P\text{-a.s},\\
    g^\ast(y)&=-\sigma^2 \ln\Big(\int e^{(f^\ast(x)-c(x,y))/\sigma^2} \ud P(x)\Big)\quad Q\text{-a.s}.\label{eqn:gstar_from_fstar}
\end{align}
\end{subequations}
By replacing $g^\ast$ as a function of $f^\ast$, the last term in \eqref{eqn:phifull} cancels out, and we arrive at the semi-dual formulation:
\begin{equation} \label{eqn:semidual} S(P,Q)= \sup_{f \in L^{\infty}(P)} \Phi(f), 
\end{equation}
where $\Phi(f)=\Phi(f,P,Q,\sigma)$ is the semi-dual functional
$$\Phi(f):= \sup_{g\in L^{\infty}(Q)} \Phi(f,g)= -\sigma^2 \int  \ln\Big(\int e^{(f(x)-c(x,y))/\sigma^2} \ud P(x)\Big)\ud Q(y)+\int f(x)\ud P(x).$$

\subsection*{Joint, conditionals and barycentric projections}
The optimal plan admits the following Radon-Nikodym density $p^\ast$ \cite{rigollet2025sample}: 
\begin{equation}\label{eq:density} p^\ast(x,y):=\frac{\ud \pi^\ast }{\ud (P\otimes Q)}(x,y)=\exp\left(\frac{f^\ast(x)+g^\ast(y)-c(x,y)}{\sigma^2}\right),\quad P\otimes Q\text{-a.s.}\end{equation}
We define the forward and backward barycentric projections, $\for^\ast, \back^\ast$, respectively, as follows:
\begin{subequations} \begin{align}
\label{eqn:for}\for^\ast(x):=\E_{\pi^\ast}\left(Y|X=x\right)=\int y  \ud\pi^\ast(y|x)=\int y p^{\ast}(x,y)  \ud Q(y),\quad P\text{-a.s.},\\
\label{eqn:back}\back^\ast(y):=\E_{\pi^\ast}\left(X|Y=y\right)=\int x \ud\pi^\ast(x|y)=\int x p^{\ast}(x,y) \ud P(x),\quad Q\text{-a.s.}
\end{align}
\end{subequations}
where the rightmost sides above follow from the condition that $\pi^\ast \in \Pi(P,Q).$ The distinction between forward and backward projection is relevant in our case because of the inherently asymmetric regime that we consider. In the following, unless there is ambiguity, we will drop the $\ast$ superscripts to denote optimal objects and will refer to them simply as $f,g,\pi, p,\for,\back$.
\subsection*{Semi-discrete to subGaussian setup}
We make the following two assumptions about the distributions. 
\begin{itemize}
\item[\textbf{(A)}] For the discrete measure, we assume that 
$$P = \sum_{k=1}^K \alpha_k \delta_{x_k},  \quad \alpha_k \geq 0, \quad \sum_{k=1}^K \alpha_k = 1, \quad \lVert x_k\rVert \leq R,k\in[K],$$
and $\underline{\alpha}:=\min_{k}\alpha_k>0$.
\item[(\textbf{B})] We assume that $Y-\mathbb E Y$ is $\varepsilon^2$-subGaussian; i.e., if  $Y \sim Q$, then for all $v\in \mathbb{R}^d$,  $$\mathbb{E} e^{\langle v, Y-\E Y\rangle} \le e^{\varepsilon^2 \|v\|^2/2}.$$ 
\end{itemize}
That is, letting $B(x,R)$ denote the ball centered at $x$ of radius $R$, we assume that $P$ is a non-degenerate mixture of $K$ (fixed) atoms in the ball $B(0,R)$, and that, after centering, $Q$ is subGaussian with proxy $\varepsilon^2$. As we discuss in the appendix \ref{app:proofssec3} we can always assume that $Q$ is centered at the cost of expressing all bounds in terms of $\tilde{R}:=
    \max_{k\in[K]}\|x_k-\E Y\|\leq R+\lVert \E Y\rVert$ instead of $R$ if we assume a uniform bound on $\E Y$. In this discrete-to-subGaussian case, we can identify the potential $f$ with a vector $f_k=f(x_k)$,
and so the semidual functional $\Phi$ is simply a multivariate function $\Phi:\mathbb{R}^K\to \mathbb{R}$:
\begin{equation}\label{eq:phidisc}\Phi(f)=  \sum_{k=1}^K f_k \alpha_k-\sigma^2 \int \ln \Big(\sum_{k=1}^K\alpha_ke^{(f_k-c(x_k,y))/\sigma^2}\Big)\ud Q(y).\end{equation}

\subsection*{Empirical setup, canonical extensions}
We will investigate rates for the empirical objects arising when replacing $P$ and $Q$ by the empirical measures %P_n$ and $Q_n$,
$$
P_n=\frac{1}{n}\sum_{i=1}^n  \delta_{X_i},\quad Q_n= \frac{1}{n}\sum_{i=1}^n \delta_{Y_i},$$
where \(X_1,\ldots,X_n \overset{\mathrm{i.i.d.}}{\sim} P\) and
\(Y_1,\ldots,Y_n \overset{\mathrm{i.i.d.}}{\sim} Q\), with the two samples
\((X_i)_{i=1}^n\) and \((Y_i)_{i=1}^n\) independent of each other. The resulting objects $f_n,g_n,p_n,\for_n$ and $\back_n$ are defined in principle only $P_n$, $Q_n$ or $P_n\otimes Q_n$ almost surely. These can be extended to functions of the entire ambient space through the canonical extensions \cite[Proposition 6]{mena2019statistical}. To do so, note that the right-hand sides in \eqref{eqn:fstar_from_gstar} and \eqref{eqn:gstar_from_fstar} define functions on $\mathbb{R}^d$. Consistent with our semi-discrete motivation, we will treat $g_n,\back_n:\mathbb{R}^d\to \mathbb{R}$ as functions, $f_n,\for_n\in\mathbb{R}^K$ as vectors expressing evaluations at fixed support points $x_1,\ldots x_k$, and $p_n(x_k,\cdot):\mathbb{R}^d\to\mathbb{R}$ as a function for each $k\in[K]$.

\section{Estimation rates for EOT-related objects}\label{sec:sample_complexity}

In this section, we present statistical rates for different objects in the setting described above. Later, in \ref{sub:strategy} we discuss the main ingredients used to establish these results, including new bounds that are of independent interest. We start by establishing the convergence of the potentials

\begin{theorem} \label{teo:potentialbound} Let $P,Q$ satisfy conditions $\textbf{(A)}$ and $ \textbf{(B)}$. Let \((f_n,g_n)\) and \((f,g)\) be the optimal entropic potentials for the problems \((P_n,Q_n)\) and \((P,Q)\), respectively, where \(f_n\) and \(f\) satisfy the gauge condition \eqref{eq:gauge}, i.e.,  $\E_P f_n(X)=\E_P f(X)=0$. Then \[ \E \|f_n-f\|_{L^\infty(P)}^2 \le \frac{C}{n}+r_{n,d}, \qquad \E \|g_n-g\|_{L^2(Q)}^2 \le \frac{C}{n}. \] Here \(C=C(K,\underline\alpha,\tilde{R},\varepsilon,\sigma^2)\) is independent of \(d\). The remainder \(r_{n,d}\) satisfies $ r_{n,d} \le C_2\exp(-cn)$, where $C_2=C_2(d,K,\underline\alpha,\tilde{R},\varepsilon,\sigma^2)$ and $c=c(K,\underline\alpha)$ are two positive constants. \end{theorem}

The exponentially small, dimension-dependent remainder \(r_{n,d}\) term in Theorem \ref{teo:potentialbound}
arises only from the event that \(P_n\) fails to charge all atoms of \(P\). On
this event, the empirical potentials are intrinsically defined only on
\(\operatorname{supp}(P_n)\), and the term \(r_{n,d}\) appears from (perhaps sub-optimal) control of the canonical extension on missing atoms. We view this dependence as a
technical artifact of the fixed-support formulation used in the theorem. 

We note in the following corollary, that this term 
disappears in the one-sample setting, where the discrete marginal \(P\) is kept
fixed, and it can also be avoided in the two-sample setting by formulating the
bounds intrinsically on the empirical support, for example, using the
\(L^\infty(P_n)\) norm.
\begin{corollary}
\label{cor:onesample}
Let $f,f_n$ be as in \Cref{teo:potentialbound}. Then
$$
\mathbb E\|f_n-f\|_{L^\infty(P_n)}^2\leq \frac{C}{n}.
$$
In the one-sample setting, let $f_n$ instead denote the optimal
discrete-side potential for $(P,Q_n)$, with $P$ kept fixed and
$\mathbb E_P f_n(X)=0$. Then
$$
\mathbb E\|f_n-f\|_{L^\infty(P)}^2\leq \frac{C}{n}.
$$
\end{corollary}

Additionally, we have the following convergence result for the joint density

\begin{theorem}\label{theo:densitibound}
Denote by $p$ and $p_n$ the densities of the optimal population and empirical couplings, $\pi,\pi_n$, with respect to $P\otimes Q$ and $P_n\otimes Q_n$, respectively, as defined in \eqref{eq:density}. Then, $$\E \lVert p-p_n\rVert_{L^\infty(P; L^2(Q))}^2 \leq \frac{C}{n}.$$
where $\|h\|_{L^\infty(P;L^2(Q))}^2
:=
\max_{k\in[K]}
\int|h(x_k,y)|^2\,\ud Q(y)$.
\end{theorem}

Furthermore, we can state the following convergence bounds for the forward and backward barycentric projections. This result will be used to obtain the first convergence rate for the sample-based Sinkhorn Expectation-Maximization algorithm, presented in Section \ref{sec:SEM}.

\begin{theorem}\label{theo:barybound} Let \(\for,\back\) be the population barycentric projections associated with
\((P,Q)\), and let \(\for_n,\back_n\) be the empirical barycentric
projections associated with \((P_n,Q_n)\). Then,
%$$\E \lVert \pi(x_k|\cdot)-\pi_n(x_k|\cdot)\lVert^2_\infty\lesssim n^{-1}.$$
\begin{equation}\label{eqn:baryconv}
 \E \lVert \back -\back_n\lVert^2_{L^\infty(Q)}\leq \frac{C}{n},\quad \E \lVert \for -\for_n\lVert^2_{L^\infty(P)}\leq \frac{Cd}{n}.\end{equation}
\end{theorem}

The following lower bound is inspired by 
\cite{pooladian2023minimax}, who provide a lower bound for the non-regularized map.

\begin{proposition}[Minimax lower bound]\label{prop:minimax_lower_bound}
Fix $R>0$, $\varepsilon>0$, $\sigma^2>0$, and $M\geq0$.
Let $\mathcal C$ be the class of pairs $(P,Q)$ satisfying
Assumptions~\textbf{(A)}--\textbf{(B)} with these fixed parameters,
$K=2$, minimum atom mass at least $1/4$, and $\|\E_QY\|\leq M$.
Then, for every integer $n\geq1$,
\[
\inf_{\widehat T}
\sup_{(P,Q)\in\mathcal C}
\E_{P^n\otimes Q^n}
\left[
\|\widehat T-\back_{P,Q}\|_{L^2(Q)}^2
\right]
\geq
\frac{R^2}{64n},
\]
where $\back_{P,Q}$ is the backward barycentric projection and the infimum
is over all estimators based on $n$ independent samples from each of $P$
and $Q$, with the two samples independent.
Consequently, the $n^{-1}$ squared-error rate for the backward barycentric
projection in Theorem~\ref{theo:barybound} is minimax optimal over
$\mathcal C$, since $\|h\|_{L^2(Q)}\leq\|h\|_{L^\infty(Q)}$.
\end{proposition}

The proof of the proposition above is given in Appendix~\ref{app:minimax_lower_bound}. We conclude this section with a sketch of the proof of Theorem~\ref{teo:potentialbound}, emphasizing the main technical ingredients and the two intermediate estimates on which the convergence rates rely.

\subsection{High-level strategy}
\label{sub:strategy}

A similar result had previously been established in the bounded-to-bounded \cite{rigollet2025sample} and bounded-to-discrete case \cite{pooladian2023minimax}. Although none of the arguments in these papers directly extend to our setup, we followed the route inspired by the latter \cite{pooladian2023minimax}, which roughly follows two steps: establishing strong concavity of the semi-dual objective $\Phi$, and empirical process arguments that enable control of $\Phi(f)-\Phi(f_n)$ at the rate $n^{-1}$. In order to extend this argument, we require control over the Hessian of $\Phi$ in our setup. Specifically, we show that

\begin{proposition}\label{prop:strong}
Suppose that $f\in\mathbb{R}^K$ satisfies 
$\lVert f\rVert_\infty \le L$ (not necessarily an optimal potential).
Then,
\[
\nabla^2 \Phi(f) \preceq  -\kappa\Big(I_K - \tfrac{1}{K}\mathbf{1}\mathbf{1}^\top\Big),
\]
with
\[
\kappa
=
\frac{1}{\sigma^2} K \underline{\alpha}^2
\exp\!\Big(-\left(
4 L +  \tilde{R}^2  + 8\tilde{R}^2\varepsilon^2/\sigma^2
\right)/\sigma^2\Big).\]
In other words, the smallest non-zero eigenvalue of $-\nabla^2 \Phi(f)$ is bounded below by $\kappa$. 
\end{proposition}

Note that the above bound depends on the norm $\lVert f\rVert_\infty$. As we will apply this bound to empirical quantities, we require estimates on this norm.  This is achieved with the following proposition

 \begin{proposition}\label{prop:potentialbound}
 Let  $f$ be the optimal entropic potential for the problem $(P,Q)$ satisfying \eqref{eq:gauge}. Then, with the shortcut $\lVert f\rVert_\infty=\lVert f\rVert_{L^\infty(P)}$, we have
\[
\lVert f\rVert_{\infty}
\;\le\;
\tilde{R}^2
\;+\;
\sigma^2 \log\!\left(\frac{1}{\underline{\alpha}}\right)
\;+\;
 \frac{2\varepsilon^2}{\sigma^2} \tilde{R}^2.
\]
\end{proposition}
This result extends the bound shown in \cite{mena2019statistical} in the subGaussian-to-subGaussian case. Specifically, from
\cite[Proposition 6]{mena2019statistical} it follows that (see Proposition \ref{prop:potentialboundd})
\begin{equation}\label{eq:potentialcanonicalbound} \lVert f\rVert_\infty
\;\lesssim\;
\tilde{R}^2 + d(\varepsilon^2 + \tilde{R}^2) + d^2(\varepsilon^2 + \tilde{R}^2)^2.\end{equation}
However, this bound still depends on $d$. As shown in the proof sketch below, using such a dimension-dependent bound would not be sufficient to establish the rate $n^{-1}$ in a dimension-free manner. We leave all details to the appendix.
\begin{proof}[Proof sketch of Theorem 1]
The proof combines two ingredients: a strong concavity property of the semi-dual objective and localized empirical process bounds controlling fluctuations of the empirical semi-dual. Specifically, by a standard M-estimator argument \cite{wellner2013weak} and Proposition \ref{prop:strong}
\[
\E \lVert f_n-f\rVert_{\infty}^2
\lesssim \kappa^{-1}(f_n,f)
\left(\Phi(f)-\Phi(f_n)\right),
\]
where $\kappa(f_n,f)$ is the one in Proposition \ref{prop:strong} that can be established uniformly over $f_n,f$. In turn, such uniformity arises from bounds on $\lVert f_n\rVert_\infty,\lVert f\rVert_\infty$. By Proposition \ref{prop:potentialbound}, $\lVert f_n\rVert_\infty,\lVert f\rVert_\infty,$ can be controlled in a dimension-independent way, albeit with dependence on a sample version $\tilde{\varepsilon}$ of the subGaussianity parameter $\varepsilon$, following the argument in \cite{mena2019statistical}. Using novel high-probability and moment bounds on such random $\tilde{\varepsilon}$ (Lemma \ref{lemma:sigmaprob} in the Appendix), we can ignore the low-probability event where $\tilde{\varepsilon}$ is large and therefore treat the quantity $\kappa^{-1}({f_n,f})$ as a deterministic and of the order
$$\kappa^{-1}(f_n,f)\approx
\exp\!\Big(
\tilde{R}^2(1+\varepsilon^2)
\Big),$$
where the above notation emphasizes dependence on $\tilde{R}^2$ and $\varepsilon^2$. Finally, the term $\Phi(f)-\Phi(f_n)$ is controlled at the $n^{-1}$ rate by careful localization arguments similar to the ones in \cite{pooladian2023minimax}.
\end{proof}

\section{Application: convergence of sample-based Sinkhorn-EM}
\label{sec:SEM}

We now apply the results of the previous section to obtain the first
finite-sample convergence guarantee for Sinkhorn-EM.

\subsection{Background on Sinkhorn-EM}

Expectation-maximization (EM) is a standard algorithm for latent-variable
models. For mixture models, its E-step computes the posterior responsibility of
each mixture component for each observation, while its M-step updates the model
parameters using these responsibilities. Sinkhorn expectation-maximization
(Sinkhorn-EM)
\citep{nejatbakhsh2020probabilistic,mena2020sinkhorn,
mena2026eotclustering} replaces the usual likelihood-based responsibilities
with those induced by an entropic optimal transport problem. In particular,
rather than assigning responsibilities to each observation independently,
Sinkhorn-EM enforces the prescribed mixture weights at the level of the
aggregate coupling.

Previous work has found that Sinkhorn-EM can converge faster or more reliably
than classical EM in some empirical settings
\citep{mena2020sinkhorn,mena2026eotclustering}. Its population objective can
also have a more favorable optimization geometry and avoid some undesirable
stationary-point configurations of the likelihood optimized by classical EM
\citep{mena2026eotclustering}. To the best of our knowledge, however, no
finite-sample convergence guarantee is available for the sample-based
Sinkhorn-EM algorithm. Our goal is not to provide a complete theoretical
explanation for its improved empirical behavior, but rather to take a first
step toward understanding it. In the same spirit as classical analyses of EM,
we study Sinkhorn-EM in a canonical Gaussian-mixture test bed for which the
sample-based and population iterations can be compared explicitly
\citep{xu2016globalEM,BalWaiYu17,weinberger2022algorithm}.

Specifically, we consider a two-component Gaussian mixture with known spherical
covariance:
\begin{equation}
\label{eq:mixture_of_gaussians}
Q^\ast
=
\alpha \mathcal N(\theta^\ast,\sigma^2 I_d)
+
(1-\alpha)\mathcal N(-\theta^\ast,\sigma^2 I_d),
\end{equation}
where $\theta^\ast\in\RR^d$ is the unknown signal and
$\alpha\in(0,1)$ and $\sigma^2$ are known. For a candidate parameter $\theta$,
define
\[
P_\theta
=
\alpha\delta_\theta+(1-\alpha)\delta_{-\theta}.
\]
Sinkhorn-EM minimizes the EOT objective
\begin{equation}
\Ll(\theta):=S(P_\theta,Q),
\end{equation}
where $Q=Q^\ast$ in the population problem and $Q=Q_n$ in the sample-based
problem. Since $Q^\ast$ is obtained by convolving $P_{\theta^\ast}$ with a
Gaussian of covariance $\sigma^2I_d$, the population objective
$S(P_\theta,Q^\ast)$ is minimized at $\theta^\ast$. The EOT regularization
parameter $\sigma^2$ therefore coincides with the Gaussian noise variance in
the mixture model.

We first describe classical EM in this model. Starting from
$\theta^0_{\mathrm{EM}}$, the population iterates satisfy
\[
\theta^{t+1}_{\mathrm{EM}}
=
F(\theta^t_{\mathrm{EM}},\alpha),
\]
where
\[
F(\theta,\alpha)
:=
\E_{Y\sim Q^\ast}
\left[
Y\bigl(2\Psi(Y,\theta,\alpha)-1\bigr)
\right]
\]
and
\[
\Psi(y,\theta,\alpha)
:=
\frac{
\alpha \exp\!\left(-\frac{\|y-\theta\|^2}{2\sigma^2}\right)
}{
\alpha \exp\!\left(-\frac{\|y-\theta\|^2}{2\sigma^2}\right)
+
(1-\alpha)\exp\!\left(-\frac{\|y+\theta\|^2}{2\sigma^2}\right)
}.
\]
Here $\Psi(y,\theta,\alpha)$ is the posterior responsibility of the component
centered at $\theta$ for the observation $y$.

Given i.i.d.\ observations
$Y_1,\ldots,Y_n\sim Q^\ast$, let
\[
Q_n=\frac1n\sum_{i=1}^n\delta_{Y_i}.
\]
The corresponding empirical EM map and iterates are
\[
F_n(\theta,\alpha)
:=
\frac1n\sum_{i=1}^n
Y_i\bigl(2\Psi(Y_i,\theta,\alpha)-1\bigr),
\qquad
\hat\theta^{t+1}_{\mathrm{EM}}
=
F_n(\hat\theta^t_{\mathrm{EM}},\alpha).
\]

Sinkhorn-EM uses the same update map but replaces the fixed weight $\alpha$
appearing in the responsibilities with a transport-corrected weight. Let
$f_n(\theta)=(f_{n,1}(\theta),f_{n,2}(\theta))\in\RR^2$ be an optimal
semi-dual potential for $S(P_\theta,Q_n)$, and define
\begin{equation}
\alpha_n(\theta)
:=
\frac{
\alpha\exp(f_{n,1}(\theta)/\sigma^2)
}{
\alpha\exp(f_{n,1}(\theta)/\sigma^2)
+
(1-\alpha)\exp(f_{n,2}(\theta)/\sigma^2)
}.
\end{equation}
The sample-based Sinkhorn-EM iterates are then
\begin{equation}\label{eq:sample-based-sem-iterates}
\hat\theta^{t+1}_{\mathrm{SEM}}
=
F_n\bigl(
\hat\theta^t_{\mathrm{SEM}},
\alpha_n(\hat\theta^t_{\mathrm{SEM}})
\bigr).
\end{equation}

Similarly, let
$f(\theta)=(f_1(\theta),f_2(\theta))\in\RR^2$ be an optimal semi-dual
potential for $S(P_\theta,Q^\ast)$, and define
\begin{equation}
\alpha(\theta)
:=
\frac{
\alpha\exp(f_1(\theta)/\sigma^2)
}{
\alpha\exp(f_1(\theta)/\sigma^2)
+
(1-\alpha)\exp(f_2(\theta)/\sigma^2)
}.
\end{equation}
The population Sinkhorn-EM iterates satisfy
\begin{equation}
\theta^{t+1}_{\mathrm{SEM}}
=
F\bigl(
\theta^t_{\mathrm{SEM}},
\alpha(\theta^t_{\mathrm{SEM}})
\bigr).
\end{equation}
Thus, in this two-component model, the difference between EM and Sinkhorn-EM
is the replacement of the fixed mixture weight $\alpha$ by the
transport-corrected weights $\alpha(\theta)$ and $\alpha_n(\theta)$.

\subsection{Finite-sample guarantee for Sinkhorn-EM}
\label{subsec:finite_sample_SEM}

We now specialize to the balanced symmetric model, corresponding to
\eqref{eq:mixture_of_gaussians} with $\alpha=1/2$. We prove that, for any fixed
number of iterations, sample-based Sinkhorn-EM tracks population Sinkhorn-EM at
the parametric $n^{-1/2}$ scale. Combining this statistical approximation with
local contraction of the population update gives an error bound consisting of
an optimization term and a statistical term.

Let $\rho:=\frac{\|\theta^\ast\|}{\sigma}$ denote the signal-to-noise ratio. There exists a universal constant $\eta_0$
such that, whenever $\rho\geq\eta_0$, the population EM operator is a
contraction on $B(\theta^\ast,\|\theta^\ast\|/4)$ with contraction factor $\kappa\leq e^{-c\rho^2}<1$ for a universal constant $c>0$ \citep{BalWaiYu17}. Define
\[
A_1
:=
\frac{\exp(2\rho^2)}
{\min\{\rho,1\}^2},
\qquad
A_2
:=
\frac{16C_{\theta^\ast,\sigma}^2}
{(1-\kappa)^2\|\theta^\ast\|^2},
\]
where $C_{\theta^\ast,\sigma}$ is the constant appearing in
\Cref{prop:sample-based-sem-iterates}. We have the following theorem. 

\begin{theorem}
\label{thm:convergence-sample-SEM}
Let $Q^\ast$ be as in \eqref{eq:mixture_of_gaussians} with
$\alpha=1/2$ and $\theta^\ast\neq0$, and let
$Y_1,\ldots,Y_n\overset{\mathrm{i.i.d.}}{\sim}Q^\ast$. Suppose that
$\rho\geq\eta_0$ and that the initialization $\hat\theta^0_{\mathrm{SEM}}$ of Sinkhorn-EM satisfies $\left\|
\hat\theta^0_{\mathrm{SEM}}-\theta^\ast
\right\|
\leq
\frac14\|\theta^\ast\|$. There exists a universal constant $C>0$ such that, if
\[
n
\geq
C\max_{j\in\{1,2\}}
\left\{
A_j
\left[
d\log(edA_j)+\log\frac1\delta
\right]
\right\},
\]
then, with probability at least
$1-\delta-n^{-c_1d}-c_2n^{-2}$, all the iterates of Sinkhorn-EM $\hat\theta^0_{\mathrm{SEM}}, ..., \hat\theta^T_{\mathrm{SEM}}$, defined in \eqref{eq:sample-based-sem-iterates}, remain in
$B(\theta^\ast,\|\theta^\ast\|/4)$ and satisfy
\[
\left\|
\hat\theta^t_{\mathrm{SEM}}-\theta^\ast
\right\|
\leq
\kappa^t
\left\|
\hat\theta^0_{\mathrm{SEM}}-\theta^\ast
\right\|
+
\frac{C_{\theta^\ast,\sigma}}{1-\kappa}
\sqrt{
\frac{d\log n+\log(1/\delta)}{n}
}
\]
for every $0\leq t\leq T$, where $c_1,c_2>0$ are constants depending
only on $\theta^\ast$ and $\sigma$.
\end{theorem}

The first term is the optimization error inherited from contraction of the
population operator, while the second is the statistical error incurred by
replacing the population distribution with its empirical measure. In
particular, for any fixed number of iterations, the empirical iterates are
$n^{-1/2}$-consistent with their population counterparts. The dependence on
$n$ and $d$ matches the corresponding sample-based EM guarantee up to the
factor $\sqrt{\log n}$ and problem-dependent constants
\citep{BalWaiYu17}. The additional terms $n^{-c_1d}$ and $c_2n^{-2}$ in the
failure probability arise from the high-probability events used to control the
empirical subGaussian parameters and vanish polynomially with the sample size.

The first step in proving \Cref{thm:convergence-sample-SEM} is to show that
population Sinkhorn-EM coincides with classical EM in the balanced symmetric
model.

\begin{proposition}
[Population Sinkhorn-EM and EM iterates coincide for balanced symmetric
two-component Gaussian mixtures]
\label{prop:SEMcoincideEM}
Suppose that \eqref{eq:mixture_of_gaussians} holds with $\alpha=1/2$ and
$\theta^\ast\neq0$. Let $\theta^{t+1}_{\mathrm{SEM}}
=
F\bigl(
\theta^t_{\mathrm{SEM}},
\alpha(\theta^t_{\mathrm{SEM}})
\bigr)$ and $\theta^{t+1}_{\mathrm{EM}}
=
F(\theta^t_{\mathrm{EM}},1/2)$ denote the population Sinkhorn-EM and EM iterates, respectively. If the two
algorithms start from the same point,
$\theta^0_{\mathrm{SEM}}=\theta^0_{\mathrm{EM}}$, then $\theta^t_{\mathrm{SEM}}=\theta^t_{\mathrm{EM}}$
for every $t\geq0$.
\end{proposition}

The proof uses the fact that, in this balanced symmetric model, the population
EM responsibilities already satisfy the marginal constraints imposed by EOT;
see \Cref{app:sinkhorn}. Consequently,
\Cref{prop:SEMcoincideEM} allows us to invoke existing contraction results for
population EM \citep{BalWaiYu17}. The sample-based algorithms do not generally
coincide, so the main step is to control the uniform statistical error
\[
\eta_n
:=
\sup_{\theta\in B(\theta^\ast,\|\theta^\ast\|/4)}
\left\|
F_n(\theta,\alpha_n(\theta))
-
F(\theta,\alpha(\theta))
\right\|.
\]

\begin{proposition}
[Uniform approximation of the sample-based Sinkhorn-EM update]
\label{prop:sample-based-sem-iterates}
Let $Q^\ast$ be as in \eqref{eq:mixture_of_gaussians} with
$\alpha=1/2$ and $\theta^\ast\neq0$. For each
$\theta\in B(\theta^\ast,\|\theta^\ast\|/4)$, let $\alpha(\theta)$ and
$\alpha_n(\theta)$ denote the population and empirical Sinkhorn weights.
There exist constants $C,c_1,c_2>0$, depending only on
$\sigma$ and $\|\theta^\ast\|$, such that, if
\[
n
\geq
C A_1
\left[
d\log(edA_1)+\log\frac1\delta
\right],
\]
then, with probability at least
$1-\delta-n^{-c_1d}-c_2n^{-2}$,
\[
\sup_{\theta\in B(\theta^\ast,\|\theta^\ast\|/4)}
\left\|
F_n(\theta,\alpha_n(\theta))
-
F(\theta,\alpha(\theta))
\right\|
\leq
C_{\theta^\ast,\sigma}
\sqrt{
\frac{d\log n+\log(1/\delta)}{n}
}.
\]
\end{proposition}

The proof follows the general sample-to-population strategy used in
finite-sample analyses of EM
\citep{BalWaiYu17,Dwivedi2018,weinberger2022algorithm}, with an additional term
arising from the estimation of the transport-corrected weight. Let
$\mathcal D_n=(Y_1,\ldots,Y_n)$ denote the observed sample. We decompose
\begin{align*}
F_n(\theta,\alpha_n(\theta))-F(\theta,\alpha(\theta))
&=
F_n(\theta,\alpha_n(\theta))
-
F(\theta,\alpha_n(\theta))
\\
&\quad+
F(\theta,\alpha_n(\theta))
-
\E_{\mathcal D_n}
\left[
F(\theta,\alpha_n(\theta))
\right]
\\
&\quad+
\E_{\mathcal D_n}
\left[
F(\theta,\alpha_n(\theta))
\right]
-
F(\theta,\alpha(\theta)).
\end{align*}

Write $\mathcal B:=B(\theta^\ast,\|\theta^\ast\|/4)$. For the first term,
\Cref{cor:data_dependent_alpha,lem:alpha_n_bound} imply that, with probability
at least $1-n^{-c_1d}-C_1'n^{-2}$,
\begin{equation}\label{eq:sem-first-term-bound}
\sup_{\theta\in\mathcal B}
\|F_n(\theta,\alpha_n(\theta))-F(\theta,\alpha_n(\theta))\|
\le C_1\sqrt{\frac{d\log n}{n}}.
\end{equation}
Here the concentration bound holds uniformly in both the parameter and the
weight, so it remains valid at the data-dependent weight $\alpha_n(\theta)$.
For the second term, \Cref{prop:uniform-centered-sinkhorn-update} gives, with
probability at least $1-\delta-C_2'n^{-2}$,
\begin{equation}\label{eq:sem-second-term-bound}
\sup_{\theta\in\mathcal B}
\left\|F(\theta,\alpha_n(\theta))
-\E_{\mathcal D_n}F(\theta,\alpha_n(\theta))\right\|
\le C_2\sqrt{\frac{d\log(en)+\log(1/\delta)}{n}}.
\end{equation}
The constants depend only on $\sigma$ and $\|\theta^\ast\|$; in particular,
one may take
$C_2=C\max\{\sigma,\|\theta^\ast\|\}\exp(C\rho^2)$,
where $C$ depends on the weight bound in \Cref{lem:alpha_n_bound}.
The third term is the bias of the population update evaluated at the empirical
weight. It is controlled by the one-sample backward-barycentric-projection
bound from \Cref{theo:barybound,cor:onesample}. To see this, let $Y\sim Q^\ast$ be independent of $\mathcal D_n$. Then
\begin{align*}
&\E_{\mathcal D_n}
\left[
F(\theta,\alpha_n(\theta))
\right]
-
F(\theta,\alpha(\theta))
\\
&\quad=
\E_{\mathcal D_n,Y}
\left[
Y
\left\{
\bigl(2\Psi(Y,\theta,\alpha_n(\theta))-1\bigr)
-
\bigl(2\Psi(Y,\theta,\alpha(\theta))-1\bigr)
\right\}
\right].
\end{align*}
Moreover, if $\back_n^\theta$ and $\back^\theta$ denote the backward
barycentric projections for $(P_\theta,Q_n)$ and
$(P_\theta,Q^\ast)$, respectively, then
\[
\back_n^\theta(Y)-\back^\theta(Y)
=
\theta
\left\{
\bigl(2\Psi(Y,\theta,\alpha_n(\theta))-1\bigr)
-
\bigl(2\Psi(Y,\theta,\alpha(\theta))-1\bigr)
\right\}.
\]
It follows that
\begin{align*}
&
\left\|
\E_{\mathcal D_n}
\left[
F(\theta,\alpha_n(\theta))
\right]
-
F(\theta,\alpha(\theta))
\right\|
\\
&\quad\leq
\frac{1}{\|\theta\|}
\left(\E\|Y\|^2\right)^{1/2}
\left(
\E
\left\|
\back_n^\theta(Y)-\back^\theta(Y)
\right\|^2
\right)^{1/2}.
\end{align*}
For $\theta\in B(\theta^\ast,\|\theta^\ast\|/4)$,
\[
\frac{1}{\|\theta\|}
\leq
\frac{4}{3\|\theta^\ast\|},
\qquad
\E\|Y\|^2
=
\|\theta^\ast\|^2+d\sigma^2.
\]
Furthermore, the one-sample version of \Cref{theo:barybound}, with
$P_\theta$ fixed, gives
\[
\E
\left\|
\back_n^\theta(Y)-\back^\theta(Y)
\right\|^2
\leq
\frac{C_{\mathrm{back}}}{n},
\]
where $C_{\mathrm{back}}$ can be chosen uniformly over the local basin and is
independent of $d$. Consequently,
\[
\left\|
\E_{\mathcal D_n}
\left[
F(\theta,\alpha_n(\theta))
\right]
-
F(\theta,\alpha(\theta))
\right\|
\leq
\frac{4\sqrt{C_{\mathrm{back}}}}{3}
\sqrt{
\frac{\|\theta^\ast\|^2+d\sigma^2}
{n\|\theta^\ast\|^2}
}
=
O\left(
\sqrt{\frac{1+d/\rho^2}{n}}
\right).
\]
The last display is uniform over $\mathcal B$ and is at most
$C_3\sqrt{d/n}$, since $d\ge1$ and $\rho>0$ is fixed. Combining it with
\eqref{eq:sem-first-term-bound} and \eqref{eq:sem-second-term-bound} by the
triangle inequality and a union bound gives, with probability at least
$1-\delta-n^{-c_1d}-c_2n^{-2}$,
\begin{align*}
\eta_n
&\le C_1\sqrt{\frac{d\log n}{n}}
   +C_2\sqrt{\frac{d\log(en)+\log(1/\delta)}{n}}
   +C_3\sqrt{\frac{d}{n}}\\
&\le C_{\theta^\ast,\sigma}
\sqrt{\frac{d\log n+\log(1/\delta)}{n}},
\end{align*}
where the last inequality uses $n\ge2$, and
$c_2=C_1'+C_2'$. This proves \Cref{prop:sample-based-sem-iterates}.
Combining this uniform update bound with population contraction and the
basin-stability argument in \Cref{app:sinkhorn} proves
\Cref{thm:convergence-sample-SEM}.

\section{Simulations}\label{sec:simulations}

\paragraph{\bf Experimental setup.}
We complement our convergence theorem with empirical simulations of the convergence of the empirical entropic backward barycentric
projection in a setting where the population map is available in closed form. We consider population measures of the form
\[
P=\frac1K\sum_{k=1}^K\delta_{x_k},
\qquad
Q=P\ast \mathcal N(0,\sigma^2 I_d).
\]
In the fixed-\(K=2\) experiments, the centers are chosen deterministically with first coordinates \(R\) and \(-R\) and the rest \(0\). When \(K\) is varied,
the centers are sampled uniformly from the Euclidean ball \(B(0,R)\), including in the \(K=2\) case. For this
pair \((P,Q)\), the population entropic optimizer has dual potential
\(f=0\), up to additive constants. Therefore, the population conditional probabilities of the discrete atoms given \(Y=y\) are given by
\[
\rho_k(y)
:=
\pi(X=x_k\mid Y=y)
=
\frac{\exp(-c(x_k,y)/\sigma^2)}
{\sum_{\ell=1}^K \exp(-c(x_\ell,y)/\sigma^2)},
\]
and the population backward barycentric projection is
\[
\back(y)=\sum_{k=1}^K x_k\rho_k(y).
\]

For each sample size \(n\), we draw independent samples
\(X_1,\dots,X_n\sim P\) and \(Y_1,\dots,Y_n\sim Q\),
and form the empirical measures \(P_n, Q_n\). The samples from \(P\) and \(Q\) are not coupled. We compute the empirical entropic backward barycentric projection \(\back_n\) by solving the empirical semi-dual
between \(P_n\) and \(Q_n\) and finding the optimal empirical potential \(f_n\). Given \(f_n=(f_{n,1},\ldots,f_{n,n})\), we define the empirical conditional probabilities, through the canonical extension, by
\[
\rho_{n,i}(y)
:=
\frac{\exp((f_{n,i}-c(X_i,y))/\sigma^2)}
{\sum_{j=1}^n \exp((f_{n,j}-c(X_j,y))/\sigma^2)},
\]
and
\[
\back_n(y)=\sum_{i=1}^n X_i\rho_{n,i}(y).
\]
We estimate
\[
\mathbb E\|\back-\back_n\|_{L^2(Q)}^2
=
\mathbb E\left[
\int \|\back(y)-\back_n(y)\|^2\,\ud Q(y)
\right]
\]
by Monte Carlo. For each repetition, we draw fresh training samples from
\(P^n\) and \(Q^n\), solve the empirical semi-dual, and approximate the
conditional \(L^2(Q)\) error using an independent evaluation sample
\(\widetilde Y_1,\dots,\widetilde Y_M\sim Q\):
\[
\widehat{\mathcal E}_n
=
\frac1M
\sum_{m=1}^M
\left\|
\back(\widetilde Y_m)-\back_n(\widetilde Y_m)
\right\|^2.
\]
We repeat this procedure \(20\) times and report the average of
\(\widehat{\mathcal E}_n\). Error bars correspond to standard errors across
repetitions. We use \(M=5000\) evaluation samples.

In all plots, \(d\in\{1,2,5,10,20\}\) and
\(n\in\{10,20,50,100,200,500,1000\}\). We perform three parameter sweeps: varying \(R\in\{0.01,0.1,1,10\}\) with
\(K=2\) and \(\sigma=0.5\), using deterministic antipodal atoms; varying
\(\sigma\in\{0.05,0.5,1,2\}\) with \(K=2\) and \(R=1\), using deterministic antipodal atoms; and varying \(K\in\{2,20,200,2000\}\) with
\(R=1\) and \(\sigma=0.5\), using atoms sampled uniformly from \(B(0,R)\). All results are reported on log-log plots, with \(n\) on the \(x\)-axis and the estimated squared error \(\mathbb E\|\back-\back_n\|_{L^2(Q)}^2\) on the \(y\)-axis.

\paragraph{\bf Results.}
Since the plots are in log-log scale, a decrease of one order of magnitude in error when \(n\) increases by one order of magnitude is visually consistent with an \(n^{-1}\) rate, while a decrease of roughly half an order of magnitude is consistent with an \(n^{-1/2}\) rate. Figure~\ref{fig:eot-combined} summarizes the three parameter sweeps. Overall, the \(R\)- and \(\sigma\)-sweeps indicate that the fast regime is most visible when the problem is well conditioned: for large \(R\) or small \(\sigma\), the observed decay is closer to \(n^{-1/2}\) over the tested range, whereas for smaller \(R\) or larger \(\sigma\), the curves become closer to \(n^{-1}\).

The top row of Figure~\ref{fig:eot-combined} varies \(R\). Larger \(R\) leads to larger errors, which is consistent with the geometric scaling of the two-atom problem. Indeed, since
\[
\back(y)=x_2+\rho_1(y)(x_1-x_2),
\]
an error in estimating \(\rho_1(y)\) produces squared barycentric error proportional to \(\|x_1-x_2\|^2=(2R)^2\).

The middle row varies \(\sigma\). Larger \(\sigma\) gives cleaner convergence because the conditional probabilities depend on \(\exp(-c/\sigma^2)\): small \(\sigma\) makes the map closer to a hard assignment rule, while larger \(\sigma\) smooths the weights and reduces sensitivity to empirical fluctuations.

The bottom row varies \(K\). Increasing \(K\) mainly increases the constants while preserving a similar qualitative decay in \(n\), suggesting that the lower-bound assumption on \(\underline\alpha=1/K\) may be conservative in these benign uniform examples, although the experiments do not rule out a genuine worst-case dependence.

\paragraph{\bf Main takeaway.}
Across the \(K\)-sweep, and in the better-conditioned regimes of the \(R\)- and \(\sigma\)-sweeps, the empirical decay is visually consistent with the theoretically predicted \(n^{-1}\)-like behavior.

\begin{figure}[t]
    \centering
    \includegraphics[width=\textwidth]{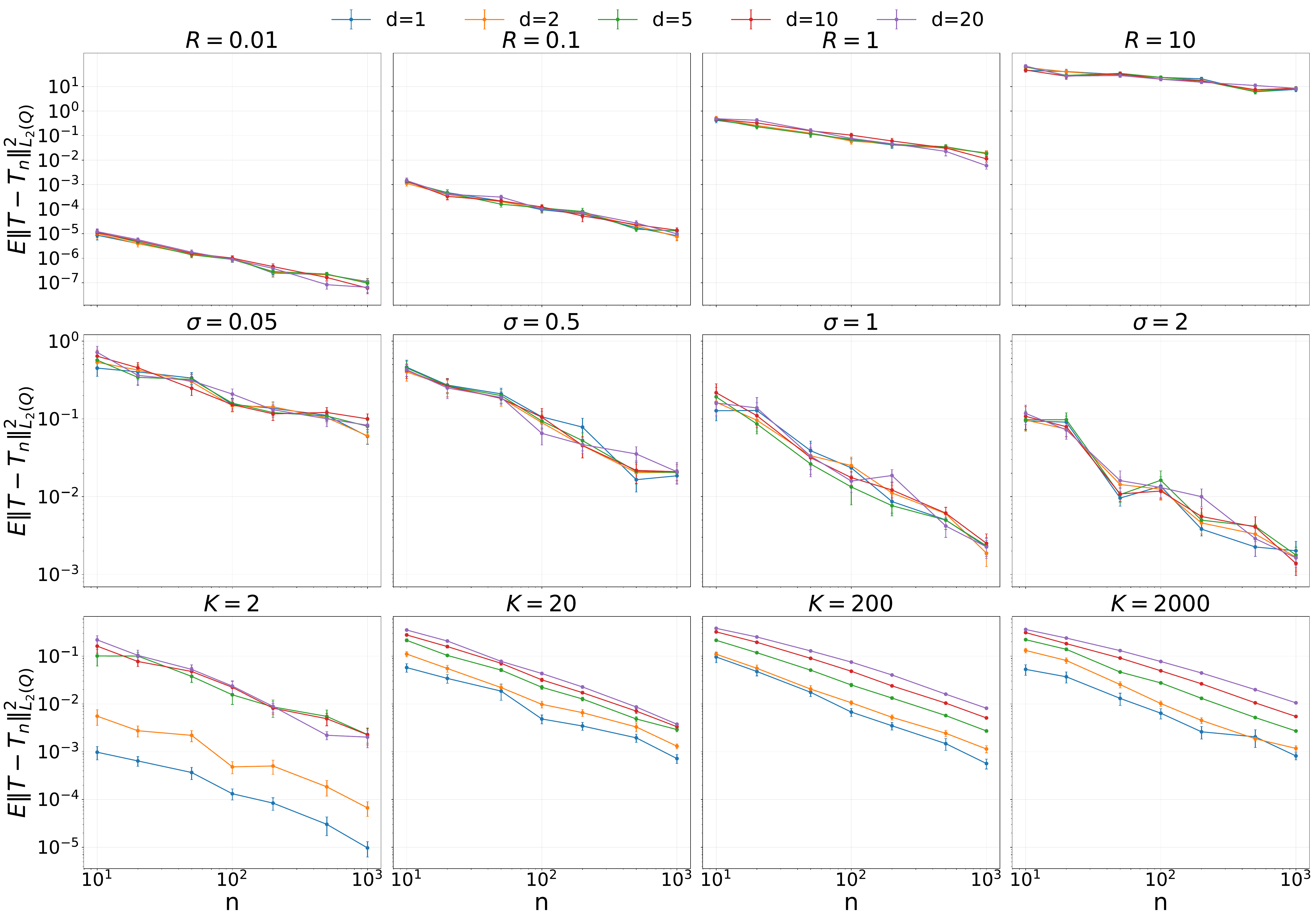}
    \caption{
    Estimated squared error \(\mathbb E\|\back-\back_n\|_{L^2(Q)}^2\) as a function of
    \(n\). Top row: varying \(R\in\{0.01,0.1,1,10\}\) with \(K=2\),
    \(\sigma=0.5\), and deterministic antipodal atoms. Middle row: varying
    \(\sigma\in\{0.05,0.5,1,2\}\) with \(K=2\), \(R=1\), and deterministic
    antipodal atoms. Bottom row: varying \(K\in\{2,20,200,2000\}\) with
    \(R=1\), \(\sigma=0.5\), and atoms sampled uniformly from \(B(0,R)\).
    Each curve corresponds to a different dimension \(d\).
    }
    \label{fig:eot-combined}
\end{figure}

\section{Conclusion and future work}
Our results provide further evidence of lower-complexity adaptation (LCA) in entropic optimal transport. We show that, when one marginal is discrete and the other is subGaussian, the empirical dual potentials and joint coupling densities achieve parametric squared-error rates of order $(n^{-1})$. The same rate holds for both barycentric projections, but with an intrinsic asymmetry: the projection onto the discrete marginal has a dimension-free leading constant, whereas the projection onto the subGaussian marginal carries a linear dependence on the ambient dimension through the moment of $Y$. Thus, noncompactness of one marginal alone does not force the slower rates known in fully subGaussian settings. Rather, in the semi-discrete regime, the statistical complexity of the coupling and of the projection onto the simpler marginal is governed primarily by the discrete measure.

Our Hessian lower bound in Proposition \ref{prop:strong} is inspired by \cite[Lemma 6]{sadhu2024stability}, which establishes a lower bound on the Hessian of the semi-dual functional in the unregularized case under a Poincaré assumption on the densities. However, our proof approach is substantially different; we don't rely on a Cheeger inequality-based bound but instead use a perhaps cruder one. As a result, downstream bounds have an exponential dependence on $R,\varepsilon$ and $\sigma^{-1}$. An open question is whether using more sophisticated spectral analysis machinery may yield an improved rate $\kappa$ and, therefore, improved constants downstream. Additionally, it is unclear whether the \(n^{-1/2}\) rates known for fully subGaussian settings are intrinsic or reflect limitations of current analyses. Finally, our Sinkhorn-EM analysis is limited to a well-specified, balanced, symmetric two-component Gaussian mixture under suitable initialization; extending it to broader mixture models and misspecified settings is an important direction.

\section{Acknowledgements}
GM is supported by NSF-DMS 2412895. We thank Tao Wang for help on a previous version of this work. We also thank Shayan Hundrieser for helpful discussions. Part of this work was done while TG was a visiting graduate student in the Federated and Collaborative Learning program at Simons Institute.

\bibliographystyle{abbrv}%{abbrvnat}
\bibliography{references}

@article{masud2023multivariate,
  title     = {Multivariate soft rank via entropy-regularized optimal transport: {S}ample efficiency and generative modeling},
  author    = {Masud, Shoaib Bin and Werenski, Matthew and Murphy, James M. and Aeron, Shuchin},
  journal   = {Journal of Machine Learning Research},
  volume    = {24},
  number    = {160},
  pages     = {1--65},
  year      = {2023}
}

@book{van2000asymptotic,
  title={Asymptotic statistics},
  author={Van der Vaart, Aad W},
  volume={3},
  year={2000},
  publisher={Cambridge university press}
}

@incollection{boucheron2003concentration,
  title={Concentration inequalities},
  author={Boucheron, St{\'e}phane and Lugosi, G{\'a}bor and Bousquet, Olivier},
  booktitle={Summer school on machine learning},
  pages={208--240},
  year={2003},
  publisher={Springer}
}

@inproceedings{pooladian2023minimax,
  title     = {Minimax estimation of discontinuous optimal transport maps: {T}he semi-discrete case},
  author    = {Pooladian, Aram-Alexandre and Divol, Vincent and Niles-Weed, Jonathan},
  booktitle = {International Conference on Machine Learning},
  pages     = {28128--28150},
  year      = {2023},
  organization = {PMLR}
}

@article{rigollet2025sample,
  title     = {On the sample complexity of entropic optimal transport},
  author    = {Rigollet, Philippe and Stromme, Austin J.},
  journal   = {The Annals of Statistics},
  volume    = {53},
  number    = {1},
  pages     = {61--90},
  year      = {2025}
}

@book{peyre2019computational,
  title     = {Computational Optimal Transport},
  author    = {Peyr{\'e}, Gabriel and Cuturi, Marco},
  publisher = {Foundations and Trends in Machine Learning},
  volume    = {11},
  number    = {5-6},
  pages     = {355--607},
  year      = {2019}
}

@book{villani2009optimal,
  title     = {Optimal Transport: Old and New},
  author    = {Villani, C{\'e}dric},
  publisher = {Springer},
  year      = {2009}
}

@article{hutter2021minimax,
  title     = {Minimax estimation of smooth optimal transport maps},
  author    = {H{\"u}tter, Jan-Christian and Rigollet, Philippe},
  journal   = {The Annals of Statistics},
  volume    = {49},
  number    = {2},
  pages     = {1166--1194},
  year      = {2021}
}

@inproceedings{cuturi2013sinkhorn,
  title     = {{S}inkhorn distances: {L}ightspeed computation of optimal transport},
  author    = {Cuturi, Marco},
  booktitle = {Advances in Neural Information Processing Systems},
  volume    = {26},
  year      = {2013}
}

@article{rosenthal1970subspaces,
  title={On the subspaces of L p (p> 2) spanned by sequences of independent random variables},
  author={Rosenthal, Haskell P},
  journal={Israel Journal of Mathematics},
  volume={8},
  number={3},
  pages={273--303},
  year={1970},
  publisher={Springer}
}

@article{sadhu2024stability,
  title={Stability and statistical inference for semidiscrete optimal transport maps},
  author={Sadhu, Ritwik and Goldfeld, Ziv and Kato, Kengo},
  journal={The Annals of Applied Probability},
  volume={34},
  number={6},
  pages={5694--5736},
  year={2024},
  publisher={Institute of Mathematical Statistics}
}

@article{weinberger2022algorithm,
  title={The EM algorithm is adaptively-optimal for unbalanced symmetric Gaussian mixtures},
  author={Weinberger, Nir and Bresler, Guy},
  journal={Journal of Machine Learning Research},
  volume={23},
  number={103},
  pages={1--79},
  year={2022}
}

@article{groppe2024lower,
  title={Lower complexity adaptation for empirical entropic optimal transport},
  author={Groppe, Michel and Hundrieser, Shayan},
  journal={Journal of Machine Learning Research},
  volume={25},
  number={344},
  pages={1--55},
  year={2024}
}

@article{mena2019statistical,
  title={Statistical bounds for entropic optimal transport: sample complexity and the central limit theorem},
  author={Mena, Gonzalo and Niles-Weed, Jonathan},
  journal={Advances in neural information processing systems},
  volume={32},
  year={2019}
}

@inproceedings{mena2026eotclustering,
  title={On model-based clustering with entropic optimal transport},
  author={Mena, Gonzalo},
year={2026},
      eprint={2605.03240},
      archivePrefix={arXiv},
      primaryClass={stat.ME},
      url={http://arxiv.org/abs/2605.03240}, 
}

@article{BalWaiYu17,
	Author = {Balakrishnan, Sivaraman and Wainwright, Martin J. and Yu, Bin},
	Doi = {10.1214/16-AOS1435},
	Fjournal = {The Annals of Statistics},
	Issn = {0090-5364},
	Journal = {Ann. Statist.},
	Mrclass = {62F10 (90C26 90C90)},
	Mrnumber = {3611487},
	Mrreviewer = {Jean-Fran\c{c}ois Dupuy},
	Number = {1},
	Pages = {77--120},
	Title = {Statistical guarantees for the {EM} algorithm: from population to sample-based analysis},
	Url = {https://doi.org/10.1214/16-AOS1435},
	Volume = {45},
	Year = {2017}}

@book{wellner2013weak,
  title={Weak convergence and empirical processes: with applications to statistics},
  author={Wellner, Jon and others},
  year={2013},
  publisher={Springer Science \& Business Media}
}

@article{xu2016globalEM,
  title={Global analysis of expectation maximization for mixtures of two gaussians},
  author={Xu, Ji and Hsu, Daniel J and Maleki, Arian},
  journal={Advances in Neural Information Processing Systems},
  volume={29},
  year={2016}
}

@inproceedings{hundrieser2024empirical,
  title={Empirical optimal transport between different measures adapts to lower complexity},
  author={Hundrieser, Shayan and Staudt, Thomas and Munk, Axel},
  booktitle={Annales de l'Institut Henri Poincare (B) Probabilites et statistiques},
  volume={60},
  number={2},
  pages={824--846},
  year={2024},
  organization={Institut Henri Poincar{\'e}}
}

@article{Dwivedi2018,
  title={Singularity, misspecification and the convergence rate of EM},
  author={Dwivedi, Raaz and Ho, Nhat and Khamaru, Koulik and Wainwright, Martin J and Jordan, Michael I and Yu, Bin},
 journal={The Annals of Statistics},
  volume={48},
  number={6},
  pages={3161--3182},
  year={2020}
}

@inproceedings{nejatbakhsh2020probabilistic,
  title={Probabilistic joint segmentation and labeling of c. elegans neurons},
  author={Nejatbakhsh, Amin and Varol, Erdem and Yemini, Eviatar and Hobert, Oliver and Paninski, Liam},
  booktitle={Medical Image Computing and Computer Assisted Intervention--MICCAI 2020: 23rd International Conference, Lima, Peru, October 4--8, 2020, Proceedings, Part V 23},
  pages={130--140},
  year={2020},
  organization={Springer}
}

@article{mena2020sinkhorn,
  title={Sinkhorn EM: An Expectation-Maximization algorithm based on entropic optimal transport},
  author={Mena, Gonzalo and Nejatbakhsh, Amin and Varol, Erdem and Niles-Weed, Jonathan},
  journal={arXiv preprint arXiv:2006.16548},
  year={2020}
}

@article{rose1998deterministic,
  title={Deterministic annealing for clustering, compression, classification, regression, and related optimization problems},
  author={Rose, Kenneth},
  journal={Proceedings of the IEEE},
  volume={86},
  number={11},
  pages={2210--2239},
  year={1998},
  publisher={IEEE}
}

@book{chewi2025statistical,
  title={Statistical optimal transport},
  author={Chewi, Sinho and Niles-Weed, Jonathan and Rigollet, Philippe},
  year={2025},
  publisher={Springer}
}

@inproceedings{daskalakis2017ten,
  title={Ten steps of EM suffice for mixtures of two Gaussians},
  author={Daskalakis, Constantinos and Tzamos, Christos and Zampetakis, Manolis},
  booktitle={Conference on Learning Theory},
  pages={704--710},
  year={2017},
  organization={PMLR}
}

@article{wu2021randomly,
  title={Randomly initialized EM algorithm for two-component Gaussian mixture achieves near optimality in O ($\sqrt{n}$) iterations.},
  author={Wu, Yihong and Zhou, Harrison H},
  journal={Mathematical Statistics \& Learning},
  volume={4},
  year={2021}
}

@inproceedings{genevay2019sample,
  title={Sample complexity of sinkhorn divergences},
  author={Genevay, Aude and Chizat, L{\'e}naic and Bach, Francis and Cuturi, Marco and Peyr{\'e}, Gabriel},
  booktitle={The 22nd international conference on artificial intelligence and statistics},
  pages={1574--1583},
  year={2019},
  organization={PMLR}
}

@article{cuturi2018semidual,
  title={Semidual regularized optimal transport},
  author={Cuturi, Marco and Peyr{\'e}, Gabriel},
  journal={SIAM Review},
  volume={60},
  number={4},
  pages={941--965},
  year={2018},
  publisher={SIAM}
}

@article{werenski2023estimation,
  title={Estimation of entropy-regularized optimal transport maps between non-compactly supported measures},
  author={Werenski, Matthew and Murphy, James M and Aeron, Shuchin},
  journal={arXiv preprint arXiv:2311.11934},
  year={2023}
}

%%%%%%%%%%%%%%%%%%%%%%%%%%%%%%%%%%%%%%%%%%%%%%%%%%%%%%%%%%%%

\appendix

\section{Proofs of main results in Section \ref{sec:sample_complexity}}\label{app:proofssec3}
\subsection{Preliminaries}
In this section we will repeatedly work with the following seminorms. For $u\in\mathbb{R}^K$ we define
  \begin{equation}\label{eq:varinf} \Var_\infty (u):=\inf_{c\in\mathbb{R}}\max_{k\in[K]}\lvert u_k-c\rvert^2=\frac{1}{4}\left(\max_{k\in[K]} u_k -\min_{k\in[K]} u_k\right)^2,\end{equation}
 
 \begin{equation} \label{eq:varalpha}\Var_\alpha (u):=\inf_{c\in\mathbb{R}}\sum_{k=1}^K \alpha_k\lvert u_k-c\rvert^2= \sum_{k=1}^K \alpha_k \left(u_k - \sum_{j=1}^K \alpha_j u_j \right)^2,\end{equation}
 and $\Var(u)=\Var_{\tilde{\alpha}}(u)$ with $\tilde{\alpha}_k=1/K$.

By working with $\Var_\infty(u)$ instead of $\lVert u\rVert_\infty^2$ we are able to handle with degeneracies arising from the fact that optimal potentials are only defined up to constant shifts. We will repeatedly use the following elementary relations among the above defined seminorms, as well as their relation with the usual $\lVert u\rVert_\infty^2$
\begin{lemma}\label{lemma:varinf}
Denote $\overline{\alpha}=\max_{k\in[K]}\alpha_k$. We have that
 \begin{equation} \label{eq:varbound1}K\underline{\alpha} \Var(u)\leq \Var_\alpha(u)\leq K\overline{\alpha}\Var(u)\end{equation} and
 \begin{equation}\label{eq:varbound2}\underline{\alpha} \Var_\infty(u)\leq \Var_\alpha (u)\leq \Var_\infty(u). \end{equation}
Further, 
\begin{equation}\label{eq:infityvar0} \Var_\infty(u)\leq \lVert u\rVert^2_\infty ,\end{equation}
and moreover, if
$$\min_{k\in[K]} u_k\leq 0 \leq \max_{k\in[K]} u_k,$$ then
 \begin{equation}\label{eq:infityvar} \lVert u\rVert^2_\infty\leq 4 \Var_\infty(u).\end{equation}
\end{lemma}
Throughout the proofs, we will work in the centered coordinates. More precisely,
let \(\mu:=\mathbb E Y\). Since the quadratic cost is invariant under common
translations,
\[
    c(x,y)=c(x-\mu,y-\mu),
\]
the entropic optimal transport problem between \((P,Q)\) is equivalent to the
one between the translated measures
\[
    P^\mu:=\sum_{k=1}^K \alpha_k \delta_{x_k-\mu},
    \qquad
    Q^\mu:=\operatorname{Law}(Y-\mu).
\]
The measure \(Q^\mu\) is centered \(\varepsilon^2\)-subGaussian by assumption,
while \(P^\mu\) is supported in the ball of radius
\[
    \tilde{R}:=\max_{k\in[K]}\|x_k-\E(Y)\|
    \le R+\|\E Y\|.
\]
Thus, without loss of generality, the proofs are carried out under the
normalization \(\mathbb E Y=0\). Therefore, in what follows, we will assume that $Y$ is centered. To account for this, we write all bounds using $\tilde{R}$ instead of $R$ in the main text.

\subsection{Proof of Proposition \ref{prop:strong}}

\begin{proof}

Define $\Tilde{\Phi}:=-\Phi$. We will bound the derivatives of this function instead. By differentiating \eqref{eq:phidisc} with respect to coordinates $f_k$ we have
$$\nabla \tilde{\Phi}(f)=\int \omega(y) \ud Q(y)-\alpha$$
where
\[
\omega_k(y) := \frac{s_k(y)}{\sum_{k'=1}^K s_{k'}(y)},\qquad
s_k(y) := \alpha_k e^{(f_k-c(x_k,y))/\sigma^2}.
\]
Differentiating once more yields
\[
\nabla^2 \tilde{\Phi}(f)
=
\frac{1}{\sigma^2} \int \big(\mathrm{Diag}(\omega(y)) - \omega(y)\omega(y)^\top \big)\, \ud Q(y).
\]

Note that for each $y$ the following identity holds for $\omega=\omega(y)$
\[
\mathrm{Diag}(\omega) - \omega\omega^\top
=
\sum_{1 \le k < k' \le K} \omega_k \omega_{k'} (e_k - e_{k'})(e_k - e_{k'})^\top,
\]
where $e_k$ is the $k$-th unit vector. Therefore,
\[
\nabla^2 \tilde{\Phi}(f)
=
\frac{1}{\sigma^2} \sum_{k<k'} a_{kk'} (e_k - e_{k'})(e_k - e_{k'})^\top,
\quad \text{where} \quad 
a_{kk'} := \int \omega_k(y)\omega_{k'}(y)\, \ud Q(y).
\]

Consequently, since for any $u\in\mathbb{R}^K$ we have
$$u^\top (e_k-e_{k'})(e_k-e_{k'})^\top u= (u_k-u_{k'})^2,$$ we get the following lower bound 
\begin{equation*}
u^\top \nabla^2 \tilde{\Phi}(f) u
=
\frac{1}{\sigma^2} \sum_{k<k'} a_{kk'} (u_k - u_{k'})^2
\ge
\frac{1}{\sigma^2} a_* \sum_{k<k'} (u_k - u_{k'})^2,
\end{equation*}
where $a_* := \min_{k<k'} a_{kk'}$. Now, take $v \in \mathbf{1}^\perp\subseteq\mathbb{R}^K$, i.e.\ $\sum_k v_k = 0$. Then, we have that
\begin{eqnarray*}
\sum_{k<k'} (v_k - v_{k'})^2 &=& \frac{1}{2}\left[\sum_{k',k} v_k^2 -2\sum_{k'} v_{k'}\sum_{k}  v_k +\sum_{k',k} v_{k'}^2\right]  \\&=& 
K \|v\|^2.
\end{eqnarray*}
Therefore, for such $v \in \mathbf{1}^\perp$
\begin{equation}\label{eq:lowerlambda}
v^\top \nabla^2 \tilde{\Phi}(f) v \ge \frac{1}{\sigma^2} K a_* \|v\|^2.
\end{equation}
It remains to bound $a_{kk'}$. Define
\[
Z_{kk'}(y) := \frac{\big(\sum_{m} s_m(y)\big)^2}{s_k(y)s_{k'}(y)}.
\]
Then
\[
\omega_k(y)\omega_{k'}(y) = \frac{1}{Z_{kk'}(y)}.
\]
By Jensen's inequality,
\[
a_{kk'}
=
\mathbb{E}\!\left[\frac{1}{Z_{kk'}(Y)}\right]
\ge
\frac{1}{\mathbb{E}[Z_{kk'}(Y)]}.
\]

Now
\[
Z_{kk'}(Y)
=
\sum_{m,\ell} \frac{s_m(Y)s_\ell(Y)}{s_k(Y)s_{k'}(Y)}.
\]
For the cost $c(x,y)=\lVert x-y\rVert^2/2$, since the quadratic terms in $Y$ cancel out and we can express
\begin{eqnarray*}
\frac{s_m(Y)s_\ell(Y)}{s_k(Y)s_{k'}(Y)}
=
\frac{\alpha_m \alpha_\ell}{\alpha_k \alpha_{k'}}
\exp\!\Big(\left[
A_{kk'}^{m\ell}
- \tfrac{1}{2}B_{kk'}^{m\ell}
+  \langle Y, C_{kk'}^{m\ell} \rangle\right]/\sigma^2
\Big),
\end{eqnarray*}
where \begin{eqnarray*} A_{kk'}^{m\ell} &=& f_m + f_\ell - f_k - f_{k'},\\
B_{kk'}^{m\ell} &=& \|x_m\|^2 + \|x_\ell\|^2 - \|x_k\|^2 - \|x_{k'}\|^2,\\
C_{kk'}^{m\ell} &=& x_m + x_\ell - x_k - x_{k'}.
\end{eqnarray*}

Using subGaussianity, for any $w\in\mathbb{R}^d$
$$
\mathbb{E} e^{\langle Y, w \rangle/\sigma^2}
\le
\exp\!\Big(
  \tfrac{\varepsilon^2}{2\sigma^4} \|v\|^2
\Big).
$$

Using $|f_k| \le L$ and $\|x_k\| \le R$, we bound
\[
A_{kk'}^{m\ell}  \le 4L,
\quad 
-B_{kk'}^{m\ell} \le 2R^2,
\quad 
\|C_{kk'}^{m\ell}\| \le 4R.
\]
Hence,
\[
\mathbb{E}\!\left[\frac{s_m(Y)s_\ell(Y)}{s_k(Y)s_{k'}(Y)}\right]
\le
\frac{\alpha_m \alpha_\ell}{\alpha_k \alpha_{k'}}
\exp(M/\sigma^2),
\]
with
\[
M:= 4 L +  R^2  +  \frac{8\varepsilon^2}{\sigma^2} R^2.
\]

Summing over $m,\ell$ gives
\[
\mathbb{E}[Z_{kk'}(Y)]
\le
\frac{e^{M/\sigma^2}}{\alpha_k \alpha_{k'}}.
\]
Therefore
\[
a_{kk'} \ge \alpha_k \alpha_{k'} e^{-M/\sigma^2} \ge \underline{\alpha}^2 e^{-M/\sigma^2}.
\]

Combining with the bound in \eqref{eq:lowerlambda} yields
\begin{equation}\label{eq:boundortho}
v^\top \nabla^2 \tilde{\Phi}(f) v
\ge
\frac{K \underline{\alpha}^2 e^{-M/\sigma^2}}{\sigma^2} \|v\|^2,
\end{equation}
if $v \in \mathbf{1}^\perp$. Additionally, note that if $v \in \mathbf{1}$, $v=c\mathbf{1}$ for $c\in \mathbb{R}$ and so 
\begin{equation*}
v^\top \nabla^2 \tilde{\Phi}(f) v
=
\frac{1}{\sigma^2} \sum_{k<k'} a_{kk'} (c - c)^2 = 0
\end{equation*}
Finally, take an arbitrary $u\in\mathbb{R}^K$. Then, $u=u-S u +S u$ where \begin{equation} \label{eq:pperp} S:=\Big(I_K - \tfrac{1}{K}\mathbf{1}\mathbf{1}^\top\Big)\end{equation}
is the projection onto the orthogonal complement of the space spanned by $\mathbf{1}$, and satisfies $S=S^2$. Then, taking $v=Su$ and using that by the above observation $ \nabla^2 \tilde{\Phi}(f)(u-Su)=0$, we obtain
\begin{eqnarray*} u^\top \nabla^2 \tilde{\Phi}(f) u
&\geq&
\frac{K \underline{\alpha}^2 e^{-M/\sigma^2}}{\sigma^2} \|Su\|^2\\
&\geq &  u^\top \left[\frac{K \underline{\alpha}^2 e^{-M/\sigma^2}}{\sigma^2}\Big(I_K - \tfrac{1}{K}\mathbf{1}\mathbf{1}^\top\Big) \right] u,
\end{eqnarray*}
which yields the desired conclusion.
\end{proof}

The following is a direct consequence of Proposition \ref{prop:strong}, and will lead to useful upper bounds for the  $f_n-f$ in the proof of Theorem \ref{teo:potentialbound}
\begin{corollary}\label{cor:strong} Let $f\in\mathbb{R}^K$ be arbitrary and $f^\ast\in\mathbb{R}^K$ be an optimal potential for $(P,Q)$. Suppose that $\lVert f\rVert_\infty \leq L$ and $\lVert f^\ast\rVert_\infty \leq L$.  Then,
$$\Var_\infty(f-f^\ast) \leq \frac{2}{\kappa}\left(\Phi(f^\ast)-\Phi(f)\right),$$
where $\kappa$ is as in $\Cref{prop:strong}$
\end{corollary}
\begin{proof}
Define $f_t= f^\ast +t(f-f^\ast)$. Note that $\lVert f_t\rVert_\infty \leq L$. By a second-order Taylor expansion with integral remainder, and using that $f^\ast$ is a stationary point of $\Phi$:
\begin{eqnarray*}  
\Phi(f)-\Phi(f^\ast)&=&\nabla \Phi(f^\ast)^\top (f-f^\ast)+\int_0^1(1-t)(f-f^\ast)^\top \nabla^2 \Phi(f_t)(f-f^\ast)dt\\ 
&=&\int_0^1(1-t)(f-f^\ast)^\top \nabla^2 \Phi(f_t)(f-f^\ast)dt\\
&\leq &-\kappa \int_0^1(1-t)(f-f^\ast)^\top S^\top S (f-f^\ast)dt\\
&=&-\frac{\kappa}{2}(f-f^\ast)^\top S  (f-f^\ast),\end{eqnarray*}
where $S$ is as in \eqref{eq:pperp}.
Now, note that for $u\in\mathbb{R}^K$,
\begin{eqnarray}
u^\top Su =\sum_{k=1}^K u^2_k-\frac{1}{K} \left(\sum_{k=1}^K u_k\right)^2=K\Var (u).
\end{eqnarray}
Therefore, by the above, and \eqref{eq:varbound2},
$$\Var_\infty(f-f^\ast)\leq K\Var(f-f^\ast)\leq \frac{2}{\kappa}\left(\Phi(f^\ast)-\Phi(f)\right).$$
\end{proof}
\subsection{Proof of Proposition \ref{prop:potentialbound}}

 \begin{proof}
Define
\[
\tilde f(x) := \frac{1}{2}\|x\|^2 - f(x).
\]

From the fact that
\[
-\frac{1}{2}\|x-y\|^2
=
-\frac{1}{2}\|x\|^2 + \langle x,y\rangle - \frac{1}{2}\|y\|^2,
\]
we obtain
\[
\tilde f(x)
=
\sigma^2 \log \int
\exp\!\left(
\frac{\langle x,y\rangle + g(y) - \frac{1}{2}\|y\|^2}{\sigma^2}
\right)
\ud Q(y).
\]

Additionally, since
\[
g(y)
=
-\sigma^2 \log \sum_{k=1}^K \alpha_k
\exp\!\left(
\frac{f(x_k) - \frac{1}{2}\|x_k-y\|^2}{\sigma^2}
\right),
\]
we have for any fixed $k$,
\[
g(y)
\le
-\left(f(x_k) - \frac{1}{2}\|x_k-y\|^2\right)
- \sigma^2 \log \alpha_k.
\]

Consequently, plugging this into the definition of $\tilde{f}$ we get that for each $x\in\mathbb{R}^d$
\begin{eqnarray*}
\tilde f(x)
&\le&
\sigma^2 \log \int
\exp\!\left(
\frac{\langle x,y\rangle + \frac{1}{2}\|x_k-y\|^2 - f(x_k) - \sigma^2 \log \alpha_k - \frac{1}{2}\|y\|^2}{\sigma^2}
\right)
\ud Q(y)\\
&\le & \sigma^2 \log \int
\exp\!\left(
\frac{\langle x,y\rangle + \frac{1}{2}\|x_k\|^2 +\frac{1}{2}\|y\|^2 -\langle x_k,y\rangle- f(x_k) - \sigma^2 \log \alpha_k - \frac{1}{2}\|y\|^2}{\sigma^2}
\right)
\ud Q(y) \\
&\le& \sigma^2 \log \int
\exp\!\left(
\frac{\langle x-x_k,y\rangle + \frac{1}{2}\|x_k\|^2 - f(x_k) - \sigma^2 \log \alpha_k}{\sigma^2}
\right)
\ud Q(y)\\ 
&\le& \tilde{f}(x_k)- \sigma^2 \log \alpha_k+\sigma^2 \log \int
\exp\!\left(
\frac{\langle x-x_k,y\rangle}{\sigma^2}
\right)
\ud Q(y)\\
&\le & \tilde f(x_k)
- \sigma^2 \log \alpha_k
+
\frac{\varepsilon^2}{2\sigma^2}\|x-x_k\|^2,
\end{eqnarray*}
where, in the last inequality, we used the subGaussianity of $Q$. Taking $x=x_{k'}$ and symmetrizing yields
\[
|\tilde f(x_k)-\tilde f(x_{k'})|
\le
\sigma^2 \log\!\left(\frac{1}{\underline{\alpha}}\right)
+
\frac{\varepsilon^2}{2\sigma^2}\|x_k-x_{k'}\|^2.
\]

Since $\|x_k-x_{k'}\|\le 2R$,
\[
|\tilde f(x_k)-\tilde f(x_{k'})|
\le
\underbrace{\sigma^2 \log\!\left(\frac{1}{\underline{\alpha}}\right)
+
2\frac{\varepsilon^2}{\sigma^2}R^2}_{B}.
\]
Now, let
\[
m := \sum_{k=1}^K \alpha_k \tilde f(x_k).
\]
From the gauge condition $\sum_{k=1}^K \alpha_k f(x_k)=0$ it follows that
\begin{equation}\label{eq:mbound}
m=\sum_{k=1}^K \alpha_k \tilde f(x_k)
=
\frac{1}{2}\sum_{k=1}^K \alpha_k \|x_k\|^2\le \frac{1}{2}R^2,
\end{equation}

Now, define $\hat f(x_k) := \tilde f(x_k) - m$. Note that \begin{equation}\label{eq:amplitude}\max_{k} \hat{f}(x_k)-\min_{k}  \hat{f}(x_k)\leq \sup_{k,k'}|\hat f(x_k)-\hat f(x_{k'})|=\sup_{k,k'}|\tilde f(x_k)-\tilde f(x_{k'})|\leq B.\end{equation}
Additionally, since $\sum_k \alpha_k \hat f(x_k)=0$, we have that $\max_{k} \hat{f}(x_k)\geq 0$ and $\min_{k} \hat{f}(x_k)\leq 0$, so by \eqref{eq:amplitude}
\begin{equation}\label{eq:fhat}
|\hat f(x_k)|\le 
\sigma^2 \log\!\left(\frac{1}{\underline{\alpha}}\right)
+
2\frac{\varepsilon^2}{\sigma^2}R^2.
\end{equation}

Now, from the fact that
\[
f(x_k)
=
\frac{1}{2}\|x_k\|^2 - \tilde f(x_k)
=
\frac{1}{2}\|x_k\|^2 - \hat f(x_k) - m,
\]
combined with \eqref{eq:mbound} \eqref{eq:fhat} we obtain
\begin{equation*}
|f(x_k)|
\le
\frac{1}{2}R^2 + |\hat f(x_k)| + |m|
\le
R^2
+
\sigma^2 \log\!\left(\frac{1}{\underline{\alpha}}\right)
+
2\frac{\varepsilon^2}{\sigma^2}R^2.
\end{equation*}
\end{proof}

\subsection{Proof of Theorem \ref{teo:potentialbound}}
 \begin{proof}
 We will first show the bound holds for $\lVert f_n-f\rVert_\infty$. We will reduce the two-sample statement to simpler ones by pivoting on $(P,Q_n)$.  Specifically, denote $\breve{f}_n$  the optimal semidual potential for $(P,Q_n)$ satisfying \eqref{eq:gauge}. We write
 \begin{eqnarray*} \E\left[\Var_{\infty}\left(f_n-f\right)\right]&=&\E\left[\Var_{\infty}\left(f_n-\breve{f}_n+\breve{f}_n-f\right)\right]\\
 &\leq& 2\E\left[\Var_{\infty}\left(f_n-\breve{f}_n\right)\right]+2\E\left[\Var_{\infty}\left(\breve{f}_n-f\right)\right]\\
 &\lesssim & \frac{1}{n} +\frac{1}{n}+r_{n,d}\lesssim \frac{1}{n}+r_{n,d}.
 \end{eqnarray*}
 Where in the last line we used Lemma  \ref{lemma:varbound1} to bound $\E\left[\Var_{\infty}\left(\breve{f}_n-f\right)\right]$ and Lemma \ref{lemma:varbound2} to bound $\E\left[\Var_{\infty}\left(\breve{f}_n-f_n\right)\right]$. Since, by definition, $\E_P(f_n-f)=0$ we have $$\min_{k\in[K]} [f_n(x_k)-f(x_k)]\leq 0 \leq\max_{k\in[K]} [f_n(x_k)-f(x_k)],$$ 
 and so by Lemma \ref{lemma:varinf} this implies that
 $$\E\left(\lVert f_n-f\rVert_\infty^2\right)\leq 4\E\Var_\infty(f_n-f)\lesssim \frac{1}{n}+r_{n,d}.$$
Let's now establish bounds for $g_n-g$. Recall that, using the canonical extensions to \eqref{eqn:gstar_from_fstar} we have that for each $y\in\mathbb{R}^d$,
$$g_n(y)=-\sigma^2\log \left(\sum_{k=1}^K\alpha^n_k\exp\left(f_n(x_k)/\sigma^2-c(x_k,y)/\sigma^2\right)\right)$$ 
and $$g(y)=-\sigma^2\log\left(\sum_{k=1}^K\alpha_k\exp\left(f(x_k)/\sigma^2-c(x_k,y)/\sigma^2\right)\right).$$
We define the intermediate function $\breve{g}_n:\mathbb{R}^d\to \mathbb{R}$.
$$\breve{g}_n(y):=-\sigma^2\log\left(\sum_{k=1}^K\alpha^n_k\exp\left(f(x_k)/\sigma^2-c(x_k,y)/\sigma^2\right)\right).$$
Then, writing $z_1=f_n(x)-c(x,y),z_0=f(x)-c(x,y)\in\mathbb{R}^K$ (the notation $c(x,y)$ denotes the vector $(c(x_1,y),\ldots,c(x_K,y)$) and $z_t=tz_1+(1-t)z_0$, and denoting $$\phi_0(z):=-\sigma^2\log \left(\sum_{k=1}^K \alpha^n_k\exp(z_k/\sigma^2)\right),\quad \phi(t):=\phi_0(z_t),$$ we have
\begin{eqnarray*} \lvert g_n(y)-\breve{g}_n(y)\rvert &=&\Big \lvert \phi_0(z_1)-\phi_0(z_0)\Big \rvert =\Big \lvert \int_0^1 \phi'(s) \ud s\Big \rvert \\ &=&\Big \lvert \int_0^1 \langle \nabla \phi_0(z_s),z_1-z_0\rangle \ud s  \Big \rvert \\ &\leq & \int_0^1 \sum_{k=1}^K \lvert z_1(k)-z_0(k)\rvert \ud s\\  
&\leq & \sum_{k:\alpha_k^n>0} \lvert f_n(x_k)-f(x_k)\rvert \\
&\leq & K \lVert f_n-f\rVert_{L^\infty(P_n)}\\ & \lesssim & \lVert f_n-f\rVert_{L^\infty(P_n)}.
\end{eqnarray*}
Above, we used that for $k$ such that $\alpha_k^n>0$, $$\Big \lvert \frac{\partial \phi_0}{\partial z_k}(z)\Big \rvert =\frac{\alpha^n_k\exp(z_k/\sigma^2)}{\sum_{k'=1}^K \alpha^n_{k'}\exp(z_{k'}/\sigma^2)} \leq 1,$$
and otherwise, this derivative equals 0. Note that the above bound is independent of $y$. %Therefore, squaring and taking supremum, we obtain
%$$\E\lVert g_n-g\rVert_\infty^2\lesssim \E\lVert f_n-f\rVert_\infty^2\lesssim \frac{1}{n}. $$
Since the preceding bound is uniform in $y$ and $Q$ is a probability measure, it implies
$$
\|g_n-\breve g_n\|_{L^2(Q)}^2
\lesssim
\|f_n-f\|_{L^\infty(P_n)}^2.
$$
Taking expectations and using the already established bound for $f_n-f$, we obtain
\begin{equation}
\label{eqn::gn}
\mathbb{E}\|g_n-\breve g_n\|_{L^2(Q)}^2
\lesssim
\frac{1}{n}.
\end{equation}

The analysis of $\breve{g}_n(y)-g(y)$ is more delicate. Let $\alpha^n_k$ be the random weights associated to the empirical measure $P_n$, and define the set of active indexes \begin{equation}\label{eqn:An}A_n:=\{k\in[K]:\alpha_k^n>0\}\end{equation}
and $\Lambda_n$ as the event
\begin{equation}\label{eqn:lambda}\Lambda_n:=\{\alpha_k^n\geq \alpha_k/2, k\in[K]\}.
\end{equation}
Define also
\begin{equation}\label{eqn:pi}\Pi_k(f,\alpha,y):=\frac{\alpha_k e^{(f_{k} -c(x_{k},y))/\sigma^2 }}{\sum_{k'=1}^K \alpha_{k'} e^{(f_{k'} -c(x_{k'},y))/\sigma^2 }},\end{equation}
and
$$
R_n(y)
:=
\frac{
\sum_{k=1}^K \alpha_k^n
\exp\left((f(x_k)-c(x_k,y))/\sigma^2\right)
}{
\sum_{k=1}^K \alpha_k
\exp\left((f(x_k)-c(x_k,y))/\sigma^2\right)
} = \sum_{k=1}^K \Pi_k(f,\alpha,y)\frac{\alpha_k^n}{\alpha_k},
$$
so that clearly
$$
\breve g_n(y)-g(y)
=
-\sigma^2\log R_n(y).
$$

First, note that since $\Pi_k(f,\alpha,y)$ is a probability vector and since
$\alpha_k^n/\alpha_k\geq 1/2$ for all $k$, we have $R_n(y)\geq 1/2$. Then, using the fact that $x\to\log x$ is lipschitz on $[1/2,\infty)$ with Lipschitz constant bounded by $2$, we get that on $\Lambda_n$,
$$
|\log R_n(y)|
\leq
2|R_n(y)-1|.
$$
Moreover,
$$
|R_n(y)-1|
=
\left|
\sum_{k=1}^K \Pi_k(f,\alpha,y)
\left(\frac{\alpha_k^n}{\alpha_k}-1\right)
\right|
\leq
\max_{k\in[K]}
\left|
\frac{\alpha_k^n-\alpha_k}{\alpha_k}
\right|
\leq
\frac{1}{\underline{\alpha}}\|\alpha^n-\alpha\|_\infty.
$$
Therefore, uniformly in $y$,
$$
\mathbf{1}_{\Lambda_n}|\breve g_n(y)-g(y)|^2
\lesssim
\|\alpha^n-\alpha\|_\infty^2.
$$
Integrating with respect to $Q$ and taking expectations, Lemma~\ref{lemma:multinomial} gives
$$
\mathbb{E}\left[
\mathbf{1}_{\Lambda_n}
\|\breve g_n-g\|_{L^2(Q)}^2
\right]
\lesssim
\mathbb{E}\|\alpha^n-\alpha\|_\infty^2
\lesssim
\frac{1}{n}.
$$

We now control the complement $\Lambda_n^c$. For $k\in [K]$ and $k'\in A_n$, define
$$
a_k(y)
:=
\log\alpha_{k}
+
\frac{f(x_{k})-c(x_{k},y)}{\sigma^2},\quad
b_{k'}(y)
:=
\log\alpha^n_{k'}
+
\frac{f(x_{k'})-c(x_{k'},y)}{\sigma^2}.
$$
From the following log-sum-exp inequality: if $z\in\mathbb{R}^M$,
$$
\max_{m\in[M]} z_m
\leq
\log\sum_{m=1}^M e^{z_m}
\leq
\max_{m\in[M]} z_m+\log M,
$$

We apply this twice. Let
$$
N_a(y):=\max_{k\in[K]} a_k(y),
\qquad
N_b(y):=\max_{k\in A_n} b_{k}(y).
$$
Then, for some remainders $r_a(y)$ and $r_b(y)$ we have
$$
\log\sum_{k=1}^K e^{a_k(y)}
=
N_a(y)+r_a(y),
\qquad
0\leq r_a(y)\leq \log K,
$$
and
$$
\log\sum_{k\in A_n}e^{b_{k}(y)}
=
N_b(y)+r_b(y),
\qquad
0\leq r_b(y)\leq \log |A_n|\leq \log K.
$$

Consequently, 
\begin{equation}
\label{eqn:difM}
\left|
\log\sum_{k\in A_n}e^{b_k(y)}
-
\log\sum_{k=1}^K e^{a_k(y)}
\right|
\leq
|N_b(y)-N_a(y)|+|r_b(y)-r_a(y)|.
\end{equation}
Since both $r_b(y)$ and $r_a(y)$ lie in $[0,\log K]$, we have
$$
|r_b(y)-r_a(y)|\leq \log K.
$$
It remains to bound $|N_b(y)-N_a(y)|$. Let
$$
k^\star\in\arg\max_{k\in[K]} a_k(y),\quad
k'^\star\in\arg\max_{k\in A_n} b_k(y).
$$
Then
$$
|N_b(y)-N_a(y)|
=
|b_{k'^\star}(y)-a_{k^\star}(y)|
\leq
\max_{k\in[K],k'\in A_n,}|b_{k'}(y)-a_k(y)|,
$$
and so by \eqref{eqn:difM}, and by the definition of $\breve{g}_n$ and $g$,
$$|\breve g_n(y)-g(y)|= \sigma^2
\left|
\log\sum_{k'\in A_n}e^{b_{k'}(y)}
-
\log\sum_{k=1}^K e^{a_k(y)}
\right|
\leq
\sigma^2 \log K+
\sigma^2 \max_{k'\in A_n,\ k\in[K]}|b_{k'}(y)-a_{k}(y)|.
$$
We must now bound the last term above.
Using that $\alpha_{k'}^n\geq 1/n$ for $k'\in A_n$, that $\alpha^n_k,\alpha_k\leq 1$, and that $\alpha_k\geq\underline{\alpha}$, we get
$$
\max_{k'\in A_n,\ k\in[K]}
\left|\log\frac{\alpha_{k'}^n}{\alpha_k}\right|
\lesssim
\log n.
$$
Moreover, by Proposition~\ref{prop:potentialbound}, $\|f\|_\infty$ is bounded by a constant depending only on $R,\sigma^2,\varepsilon,\underline{\alpha}$. Finally, for the quadratic cost,
$$
|c(x_k,y)-c(x_{k'},y)|
=
\left|\frac{1}{2}\big(\|x_k\|^2-\|x_{k'}\|^2\big)
-
\langle x_k-x_{k'},y\rangle\right|\leq R^2
+
|\langle x_k-x_{k'},y\rangle|.
$$
Therefore,
$$
|\breve g_n(y)-g(y)|
\lesssim
\log n
+
L(y),
\qquad \text{where }
L(y):=
\max_{k,k'\in[K]}
|\langle x_k-x_{k'},y\rangle|.
$$
We bound the expectation of the last term using a standard subGaussianity argument: since \(Y\sim Q\) is \(\varepsilon^2\)-subGaussian, for every \(k,k'\in[K]\) the random variable $
Z_{k,k'}:=\langle x_k-x_{k'},Y\rangle$
is subGaussian with proxy at most
$
\varepsilon^2\|x_k-x_{k'}\|^2\leq 4\varepsilon^2R^2,
$ and so \cite[Chapter 2.3]{boucheron2003concentration}
\[
\mathbb P_Q(|Z_{k,k'}|>t)
\leq
2\exp\left(-\frac{t^2}{8\varepsilon^2R^2}\right).
\]
Therefore, by the union bound,
\[
\mathbb P_Q(L(Y)>t)
\leq
2K^2\exp\left(-\frac{t^2}{8\varepsilon^2R^2}\right),
\]
 and so by the layer cake representation,
\[
\mathbb E_Q L(Y)^2
=
\int_0^\infty 2t\,\mathbb P_Q(L(Y)>t)\,dt.
\]
Set
\[
t_0:=4\varepsilon R\sqrt{\log(2K)}.
\]
Then
\[
\int_0^{t_0}2t\,\mathbb P_Q(L(Y)>t)\,dt
\leq
t_0^2
\lesssim
\varepsilon^2R^2\log(2K),
\]
and using that $\int_{t_0}^\infty x\exp(-x^2/a)dx=a/2\exp(-t_0^2/a)$,
\[
\int_{t_0}^{\infty}2t\,\mathbb P_Q(L(Y)>t)\,dt
\leq
4K^2\int_{t_0}^{\infty}
t\exp\left(-\frac{t^2}{8\varepsilon^2R^2}\right)\,dt
= 4K^2 \frac{8\varepsilon^2R^2}{2} \exp\left(-\frac{t_0^2}{8\varepsilon^2R^2}\right) \leq \frac{16\varepsilon^2 R^2K^2}{(2K)^2}\lesssim \varepsilon^2R^2.
\]
Consequently,
\[
\mathbb E_Q L(Y)^2
\lesssim
\varepsilon^2R^2\log(2K).
\]

We have concluded that
$$
\|\breve g_n-g\|_{L^2(Q)}^2
\lesssim
\log^2 n
+
\varepsilon^2R^2\log(2K).
$$
Therefore, by Lemma~\ref{lemma:multinomial},
$$
\mathbb{E}\left[
\mathbf{1}_{\Lambda_n^c}
\|\breve g_n-g\|_{L^2(Q)}^2
\right]
\lesssim
\left(\log^2 n+\varepsilon^2R^2\log(2K)\right)
\mathbb{P}(\Lambda_n^c)
\lesssim
\frac{1}{n}.
$$
Combining the bounds on $\Lambda_n$ and $\Lambda_n^c$, we conclude that
\begin{equation}
\label{eqn::brevegn}
\mathbb{E}\|\breve g_n-g\|_{L^2(Q)}^2
\lesssim
\frac{1}{n}.
\end{equation}

Finally, by the triangle inequality,
$$
\|g_n-g\|_{L^2(Q)}^2
\leq
2\|g_n-\breve g_n\|_{L^2(Q)}^2
+
2\|\breve g_n-g\|_{L^2(Q)}^2.
$$
Taking expectations and using and \eqref{eqn::gn} and \eqref{eqn::brevegn}, we obtain
$$
\mathbb{E}\|g_n-g\|_{L^2(Q)}^2
\lesssim
\frac{1}{n}.
$$
 \end{proof}

 \subsubsection*{Proof of Corollary \ref{cor:onesample}}
 \begin{proof}
 The reader can verify that the term $r_{n,d}$ appears from the need to control $f_n(x_k)$ where $k$ is such that $\alpha_k^n=0$. In the one-sample case, we work with the population measure, which puts mass on all $x_k$ since $\alpha_k\geq \underline{\alpha}>0$. Likewise, if we measure error using $L^\infty(P_n)$ then there is no need to control $f_n(x_k)$ on unobserved atoms
 \end{proof}
The following corollary strengthening to higher order moments of $\|f_n-f\|_{\infty}$
\begin{corollary}
\label{cor:potential-high-moments}
Under the assumptions of Theorem~\ref{teo:potentialbound}, for every fixed integer
\(q\ge 1\),
\[
\mathbb E\|f_n-f\|_{\infty}^{2q}\lesssim n^{-q}+r_{n,d,q}.
\]
where the underlying constants depend on $q$, in addition to the parameters in Theorem \ref{teo:potentialbound}. The term $r_{n,d,q}$ is similar to the one defined in Theorem \ref{teo:potentialbound} but the leading constant is allowed to depend on $q$.
We can remove the term $r_{n,d,q}$ either in the one-sample case or when measuring error using the norm  $E\|f_n-f\|_{L^\infty(P_n)}$.
\end{corollary}

\begin{proof}
The proof is the same as the proof of Theorem~\ref{teo:potentialbound}, replacing the
second-moment bounds by \(2q\)-moment bounds. We briefly indicate the changes. To bound $\Var_\infty(\breve{f}_n-f)$, we first note that the localized empirical-process estimate used in the proof of
Theorem~\ref{teo:potentialbound} is available in an arbitrary fixed moment order: indeed, by Lemma \ref{lemma:empirical} for every
\(p\ge 2\),
\[
\mathbb E\left[
\sup_{\operatorname{Var}_\infty(u-f)\le \tau^2}
\left|
\int(\Gamma_u-\Gamma_f)\,d(Q_n-Q)
\right|^p
\right]
\lesssim
\left(\tau\sqrt{\frac K n}\right)^p .
\]
Choosing \(p>2q\) in the dyadic peeling argument gives
\[
\mathbb E\operatorname{Var}_\infty(\breve f_n-f)^q
\lesssim n^{-q}.
\]
Indeed, on the slice
\[
\operatorname{Var}_\infty(\breve f_n-f)\in [a_j^2,a_{j+1}^2],
\qquad
a_j:=2^j n^{-1/2},
\]
strong concavity implies that the localized empirical process must be at least of order
\(a_j^2\). Markov's inequality with moment \(p\) then gives a summable bound
\[
\mathbb P\left(
\operatorname{Var}_\infty(\breve f_n-f)\in [a_j^2,a_{j+1}^2]
\right)
\lesssim 2^{-pj},
\]
and hence
\[
\mathbb E\operatorname{Var}_\infty(\breve f_n-f)^q
\lesssim
\sum_{j\ge0} a_{j+1}^{2q}2^{-pj}
\lesssim n^{-q},
\]
provided \(p>2q\).

For the two-sample term, we mimic the strategy in the proof of Lemma \ref{lemma:varbound2}. By Lemma \ref{lemma:reversevar}, we have that on
\(\Lambda_n\),
\[
\operatorname{Var}_\infty(f_n-\breve f_n)
\lesssim
\chi^2(P_n\|P).
\]
Then, by Lemma \ref{lemma:multinomial},
$$ \mathbb E\left[1_{\Lambda_n}\operatorname{Var}_\infty(f_n-\breve f_n)^q\right]\lesssim \E\left(\left[\chi^2(P_n\|P)\right]^q\right)\leq \frac{K^q}{\underline{\alpha}^q}\E\left[\lVert \alpha^n-\alpha\rVert_\infty \right]^{2q}  \lesssim n^{-q}.$$
In the complement \(\Lambda_n^c\), we use the crude bound $\Var_\infty(f_n-\breve f_n)^q\lesssim \lVert f_n\rVert^q+\lVert f\rVert^q$ and then control each term using the (dimension-dependent) bound in Proposition \ref{prop:potentialboundd}. This will lead to a polynomial bound in a random subGaussianity parameter $\tilde{\varepsilon}$, whose moments are nonetheless bounded. We conclude the analysis of this event using Cauchy-Schwarz and the fact that $\Pb(\Lambda^c_n)$ decays exponentially fast. Thus,
\[
\mathbb E\operatorname{Var}_\infty(f_n-\breve f_n)^q
\lesssim n^{-q}+r_{n,d}.
\]

Combining the one-sample and two-sample bounds yields
\[
\mathbb E\|f_n-f\|_{\infty}^{2q}\lesssim 
\E\operatorname{Var}_\infty(f_n-f)^q +\E\operatorname{Var}_\infty(f_n-\breve{f}_n)^q \lesssim n^{-q}+r_{n,d,q}.
\]
\end{proof}

\begin{lemma}\label{lemma:varbound1} In the setup of Theorem \ref{teo:potentialbound}, if $f_n,f$ are the optimal potentials for $(P,Q_n)$ and $(P,Q)$, respectively, satisfying \eqref{eq:gauge}. Then
 $$\E\left[\Var_\infty(f_n-f)\right]\lesssim \frac{1}{n},$$
 with constants that depend on $R,\sigma^2$ and $\varepsilon$ but not on $d$.
\end{lemma}
\begin{proof}
%Note first that since $\Var_\infty(f)=\Var_\infty(f+c)$, we can assume w.l.o.g. the gauge conditions $\E_P(f(X))=\E_{P}(f_n(X))=0$. 
Let $\tilde{\varepsilon}$ be the smallest such that $Q_n,Q$ are uniformly subGaussian, which is a finite random variable by Lemma \ref{lemma:MNW}. Applying Proposition \ref{prop:potentialbound} to the pairs $(P,Q)$ and $(P,Q_n)$ we have
\[
\max\{\lVert f\rVert_\infty,\lVert f_n\rVert_\infty\}
\;\lesssim\;
R^2
\;+\;
\sigma^2 \log\!\left(\frac{1}{\underline{\alpha}}\right)
\;+\;
 \frac{\tilde{\varepsilon}^2}{\sigma^2} R^2.
\]
Then, by Lemma \ref{lemma:varinf} we have the bound
\begin{eqnarray*}Z_n:= \Var_\infty(f_n-f)&\leq& \lVert f-f_n\rVert_\infty^2\\
&\leq &2\lVert f\rVert_\infty^2+2\lVert f_n\rVert_\infty^2\\ 
&\lesssim &R^4
\;+\;
\sigma^4 \log^2\!\left(\frac{1}{\underline{\alpha}}\right)
\;+\;
 \frac{\tilde{\varepsilon}^4}{\sigma^4} R^4.
\end{eqnarray*}
 Above, we used that $(a+b)^2\leq 2(a^2+b^2)$. Let's now define the event $E_n=\{\tilde{\varepsilon}^2<12\varepsilon^2\}$. We will rely on the following decomposition
\begin{equation}\label{eq:abc}\E\left(Z_n\right)=\E\left(Z_n1_{E_n}\right)+\E\left(Z_n1_{E_n^c}\right)\lesssim  \underbrace{\E\left(Z_n1_{E_n}\right)}_{A_n}+\underbrace{\E\left(Z_n^2\right)^{1/2}}_{B_n} \underbrace{\Pb\left(E_n^c \right)^{1/2}}_{C_n},
\end{equation}
and bound each of $A_n,B_n,C_n$. First, regarding $C_n$, it follows directly from Lemma \ref{lemma:sigmaprob} (with $m=4$) that
 $C_n=\Pb(E_n^c)^{1/2}\lesssim n^{-1}$. Additionally, by Lemma \ref{lemma:MNW}(d) all moments of $\tilde{\varepsilon}^2$ are finite, implying that $B_n$ is bounded by a polynomial of degree 4 in $R$, and, in particular, is finite. It only remains to bound $A_n$. Note that in $E_n$ we have a deterministic bound for $f_n$:
 \begin{equation*}\lVert f_n\rVert_\infty
\lesssim L=L_{\varepsilon,R,\sigma^2,\underline{\alpha}}:=R^2+\sigma^2 \log\!\left(\frac{1}{\underline{\alpha}}\right)
\;+\;
 \frac{\varepsilon^2}{\sigma^2} R^2.
\end{equation*}
%Likewise, on this event
%\begin{equation*}Z_n
%\lesssim L'=L'_{\varepsilon,R,d}:=R^4 + d^2(\varepsilon^4 + R^4) + d^4(\varepsilon^8 + R^8).
%\end{equation*}

Denote $\Phi_n$ the semidual function for the one-sample empirical problem $(P,Q_n)$. From Corollary \ref{cor:strong} we obtain that for $\kappa=\kappa(\varepsilon,R,\sigma^2,\underline{\alpha})$
 \begin{eqnarray*}
 \Var_{\infty}(f_n-f)&\leq&\frac{2}{\kappa}\left(\Phi(f)-\Phi(f_n)\right)\\
 &\leq &\frac{2}{\kappa}\left(\Phi(f)-\Phi_n(f_n)+\Phi_n(f_n)-\Phi(f_n)\right)\\
 &\leq & \frac{2}{\kappa}\left(\Phi(f)-\Phi_n(f)+\Phi_n(f_n)-\Phi(f_n)\right)\\
 &=&\frac{2}{\kappa} \left(\int \Gamma_f(y) \ud (Q_n-Q)(y)+  \int\Gamma_{f_n}(y) \ud (Q-Q_n)(y) \right)\\
 &\leq & \frac{2}{\kappa} \underbrace{\Bigg \lvert  \int \left(\Gamma_f(y)-\Gamma_{f_n}(y)\right) \ud (Q-Q_n)(y) \Bigg \rvert}_{W_n}.
 \end{eqnarray*}
 Above, we also used that by optimality of $f_n$, $\Phi_n(f_n)\geq \Phi_n(f)$. Note that, from the above, in $E_n$ if $Z_n\geq a$ then $W_n\geq \kappa a/2$. Consider now the dyadic partition $[a_k^2,a_{k+1}^2]$ where $k\geq0$ and $a_k=2^k/\sqrt{n}$. By Markov's inequality and the above observations, we have
 \begin{eqnarray*}
A&=& \E(Z_n1_{E_n})\\
&=& \sum_{k=0}^\infty \E\left(Z_n 1_{Z_n\in[a^2_k,a^2_{k+1}]}\right)\\ &\leq & \sum_{k=0}^\infty a^2_{k+1}  \Pb\left(Z_n\geq a^2_k,Z_n\leq a^2_{k+1}, E_n\right)\\
&\leq & \sum_{k=0}^\infty a^2_{k+1}  \Pb\left(W_n\geq \frac{\kappa}{2}a^2_k,Z_n\leq a^2_{k+1}, E_n \right)\\
&\leq & \sum_{k=0}^\infty a^2_{k+1}  \Pb\left(\sup_{\Var_\infty(f-f')\leq a_{k+1}^2}\Bigg \lvert  \int \left(\Gamma_f(y)-\Gamma_{f'}(y)\right) \ud (Q-Q_n)(y) \Bigg\rvert \geq \frac{\kappa}{2}a_k^2\right)\\
&\leq & \sum_{k=0}^\infty a^2_{k+1}  \Pb\left(\sup_{\Var_\infty(f-f')\leq a_{k+1}^2}\Bigg \lvert  \int \left(\Gamma_f(y)-\Gamma_{f'}(y)\right) \ud (Q-Q_n)(y) \Bigg\rvert \geq \frac{\kappa}{2}a_k^2\right)\\
&\leq &\sum_{k=0}^\infty   \frac{2a^2_{k+1}}{\kappa^p a^{2p}_k}\E\left(\left[\sup_{\Var_\infty(f-f')\leq a_{k+1}^2}\Bigg \lvert  \int \left(\Gamma_f(y)-\Gamma_{f'}(y)\right) \ud (Q-Q_n)(y) \Bigg\rvert\right]^p\right)\\
&\lesssim &\frac{\kappa^{-1}}{n } \sum_{k=0}^\infty  2^{2k-pk}\\
&\lesssim & \frac{\kappa^{-1}}{n },
 \end{eqnarray*}
 if $p>2$. At this point, the conclusion follows from \eqref{eq:abc}.
\end{proof}
\begin{lemma}\label{lemma:varbound2}
In the setup of Theorem \ref{teo:potentialbound}, let $f_n$ and $\breve{f}_n$ be the optimal semidual potentials for $(P_n,Q_n)$ and $(P,Q_n)$, respectively and satisfying \eqref{eq:gauge}. Then,
$$\E\left(\Var_\infty(f_n-\breve{f}_n)\right) \lesssim \frac{1}{n}$$
\end{lemma}
\begin{proof}
 Consider the event $\Lambda_n$ defined in \eqref{eqn:lambda}.
We have
\begin{eqnarray*}
\E\left[\Var_\infty(f_n-\breve{f}_n)_\infty\right] &=& \E\left(1_{\Lambda_n} \Var_\infty(f_n-\breve{f}_n)_\infty\right)+\E\left(1_{\Lambda^c_n}\Var_\infty(f_n-\breve{f}_n)_\infty\right)\\
&\lesssim & \frac{\kappa}{\underline{\alpha}^2}\E\left(\chi^2(P_n\|P)\right)+ \Pb(\Lambda_n^c)^{1/2}\E\left(\lVert f_n\rVert_\infty^4+\lVert \breve{f}_n\rVert_\infty^4\right)^{1/2}\\
&\lesssim & \frac{1}{n} + \exp(-cn/2)\E\left(\left[R^2 + d(\tilde{\varepsilon}^2 + R^2) + d^2(\tilde{\varepsilon}^2 + R^2)^2\right]^4\right)^{1/2}\\
&\lesssim & \frac{1}{n} + \exp(-cn/2) \left(R^8 + d^4(\varepsilon^8 + R^8) + d^8(\varepsilon^{16} + R^{16}) \right)^{1/2}\\
&\lesssim & \frac{1}{n}+r_{n,d}.
\end{eqnarray*}
In the second line, we used Lemma \ref{lemma:reversevar} with $(P,Q_n)$ and $(P_n,Q_n)$ (note that no condition is imposed on the second measure), the crude bound  $\Var_{\infty}(f)\leq \lVert f\rVert_\infty^2$ (Lemma \ref{eq:varinf}), and Cauchy-Schwarz. In the third line, we used the weaker, dimension-dependent bound Proposition \ref{prop:potentialboundd} for the potentials. These bounds depend on the random subGaussianity parameter $\tilde{\varepsilon}$ described in the proof of \ref{lemma:varbound1}. In the fourth line, we control the moments of $\tilde{\varepsilon}$ using the same argument as in the proof of Lemma \ref{lemma:varbound1}.
\end{proof}
\subsection{Proof of Theorem \ref{theo:densitibound}}
\begin{proof}
We have
\begin{eqnarray*} p(x,y) &=& \exp\left(\frac{f(x)+g(y)-c(x,y)}{\sigma^2}\right)= \frac{\exp\left(\frac{f(x)-c(x,y)}{\sigma^2}\right)}{\sum_{k=1}^K\alpha_k\exp\left(\frac{f(x_k)-c(x_k,y)}{\sigma^2}\right)},\\
p_n(x,y)&=&\exp\left(\frac{f_n(x)+g_n(y)-c(x,y)}{\sigma^2}\right)=\frac{\exp\left(\frac{f_n(x)-c(x,y)}{\sigma^2}\right)}{\sum_{k=1}^K\alpha_k^n\exp\left(\frac{f_n(x_k)-c(x_k,y)}{\sigma^2}\right)}.
\end{eqnarray*}
These expressions are well defined for $x=x_k,k\in[K],y\in\mathbb{R}^d$ through the canonical extensions. The righ-hand sides follow by replacing \eqref{eqn:gstar_from_fstar} in the exponents.
We can bound the difference using the bound 
$\lvert e^{a}-e^b\rvert \leq e^{\max\{a,b\}}|a-b|$. This leads to 

\begin{eqnarray*}
    \lvert p(x,y)-p_n(x,y)\rvert&\leq& \frac{\max\{p(x,y),p_n(x,y)\}}{\sigma^2}\left(\lvert f_n(x)-f(x)+g_n(y)-g(y)\lvert \right)
\end{eqnarray*}

We now bound the expectation on the event \(\Lambda_n\) defined in \eqref{eqn:lambda} and its complement. Inside \(\Lambda_n\), for each $k\in[K]$ we have
$$p(x_k,y)=\frac{1}{\alpha_k}\Pi_k(f,\alpha,y)\leq \frac{1}{\underline{\alpha}},\quad p_n(x_k,y)=\frac{1}{\alpha^n_k}\Pi_k(f_n,\alpha^n,y)\leq \frac{2}{\underline{\alpha}}.$$

Therefore, on \(\Lambda_n\),
 $$\lvert p(x,y)-p_n(x,y)\rvert \lesssim  \lvert f_n(x)-f(x)\rvert +\lvert g_n(y)-g(y)\rvert$$
Integrating with respect to \(Q\), and then taking the maximum over \(k\), gives
\[
\mathbf 1_{\Lambda_n}
\|p_n-p\|_{L^\infty(P;L^2(Q))}^2
\lesssim
\|f_n-f\|^2_\infty
+
\|g_n-g\|_{L^2(Q)}^2.
\]
Taking expectations and and using the estimates from Theorem~\ref{teo:potentialbound} on $\Lambda_n$,
\[
\mathbb E\left[
\mathbf 1_{\Lambda_n}
\|p_n-p\|_{L^\infty(P;L^2(Q))}^2
\right]
\lesssim
\frac1n.
\]
It remains to control the complement \(\Lambda_n^c\). Since
\[
\|p\|_{L^\infty(P;L^2(Q))}^2
=
\max_{k\in[K]}\int p_k(y)^2\,\ud Q(y)\le \max_{k\in[K]}\int \left[\frac{1}{\alpha_k}\right]^2\,\ud Q(y)
\le
\underline\alpha^{-2}.
\]
we get 
\[
\begin{aligned}
\mathbb E\left[
\mathbf 1_{\Lambda_n^c}
\|p_n-p\|_{L^\infty(P;L^2(Q))}^2
\right]
&\le
2\mathbb E\left[
\mathbf 1_{\Lambda_n^c}
\|p_n\|_{L^\infty(P;L^2(Q))}^2
\right]
+
2\underline\alpha^{-2}\mathbb P(\Lambda_n^c)\\
& \le 2\Pb (\Lambda_n^c)^{1/2} \E\left[ \|p_n\|_{L^\infty(P;L^2(Q))}^4\right]^{1/2} +2\underline\alpha^{-2}\mathbb P(\Lambda_n^c)
\end{aligned}
\]
Since, by Lemma \ref{lemma:multinomial}, $P(\Lambda_n^c)$ decays exponentially fast, it only suffices to show that the growth of $\E\left[ \|p_n\|_{L^\infty(P;L^2(Q))}^4\right]$ is bounded by a polynomial. We devote the rest of this proof to show that for every fixed integer \(m\ge 1\),
\begin{equation}\label{eqn:exppm}
\mathbb E\left[
\|p_n\|_{L^\infty(P;L^2(Q))}^{2m}
\right]\lesssim 
 n^{2m+1}.
\end{equation}

Indeed, let $A_n$ be the set of active indexes defined in \eqref{eqn:An}. 
Take an arbitrary \(k_n\in A_n\) so that \(\alpha_{k_n}^n\ge 1/n\) and therefore for each $y\in\mathbb{R}^d$,
$p_n(x_{k_n},y)\le 1/\alpha_{k_n}^n \le n$. Note that, by rearranging terms, we can write for an arbitrary $k\in[K]$,
\begin{eqnarray}
\nonumber p_{n}(x_k,y)
&=&
p_n(x_{k_n},y)
\exp\left(
\frac{
f_{n}(x_k)-f_{n}(x_{k_n})
+
c(x_{k_n},y)-c(x_k,y)
}{\sigma^2}
\right)\\
\nonumber &\leq& n \exp\left(
\frac{
f_{n}(x_k)-f_{n}(x_{k_n})
}{\sigma^2}\right)\exp\left(\frac{
c(x_{k_n},y)-c(x_k,y)
}{\sigma^2}
\right).
\end{eqnarray}
And so, for each $m$ 
\begin{equation}
    \label{eqn::prodexp}
\E\left(\mathbf
\|p_n\|_{L^\infty(P;L^2(Q))}^{2m}\right)\leq n^{2m}\E\left[\exp\left(\frac{2m}{\sigma^2} \max_{k,k'\in[K]}\lvert f_{n}(x_k)-f_{n}(x_{k'})\rvert\right)\right] \int \exp\left(\frac{2m}{\sigma^2}\tilde{L}(y)\right)\ud Q(y),
\end{equation}

\[
\tilde{L}(y):=
\max_{k,k'\in[K]}|c(x_k,y)-c(x_{k'},y)|.
\]
we used Jensen's inequality to the function $x\to x^m$ to write $m$ inside integration with respect to $Q$. We now bound the two exponential terms. We bound the second exponential as follows
\begin{eqnarray*} \int \exp\left(\frac{2m}{\sigma^2}\tilde{L}(y)\right)\ud Q(y)
&\leq& \exp\left(\frac{2m}{\sigma^2}R^2\right)\sum_{k,k'=1}^K\int \exp\left(\frac{2m}{\sigma^2}|\langle x_k-x_{k'},y\rangle|\right)\ud Q(y)\\
&\leq &  2\exp\left(\frac{2m}{\sigma^2}R^2\right)\sum_{k,k'=1}^K \int \exp\left(\frac{2m^2\lVert x_k-x_{k'}\rVert^2\varepsilon^2}{\sigma^4}\right)\\
&\leq & 2K^2\exp\left(\frac{2m}{\sigma^2}R^2\right)\exp\left(\frac{4m^2R^2\varepsilon^2}{\sigma^4}\right)\label{eqn::expQ}
\end{eqnarray*}

We first expectation is a random quantity and we will bound its expectation. First note that
\begin{equation}\label{eqn::fbound}
\max_{k,k'\in[K]}\lvert f_{n}(x_k)-f_{n}(x_{k'})\rvert\le M_n:=\max_{i\leq n} \tilde{L}(Y_i).
\end{equation}

Indeed, by \eqref{eqn:fstar_from_gstar}
\[
f_n(x_k)
=
-\sigma^2
\log
\frac{1}{n}\sum_{i=1}^n
\exp\left(
\frac{g_n(Y_i)-c(x_k,Y_i)}{\sigma^2}
\right).
\]
Therefore,  for every sample point $Y_i$
\[
-M_n
\le
c(x_k,Y_i)-c(x_{k'},Y_i)
\le
M_n,
\]
and so
\[
e^{-M_n/\sigma^2}
e^{(g_n(Y_i)-c(x_{k'},Y_i))/\sigma^2}
\le
e^{(g_n(Y_i)-c(x_k,Y_i))/\sigma^2}
\le
e^{M_n/\sigma^2}
e^{(g_n(Y_i)-c(x_{k'},Y_i))/\sigma^2}.
\]
Averaging over samples and taking logarithms yields \eqref{eqn::fbound}. We can also bound $M_n$ as follows
\[
M_n
\le
R^2+
\max_{1\le i\le n}\max_{k,k'\in[K]}
|\langle x_k-x_{k'},Y_i\rangle|.
\]
Therefore, for every fixed \(m\ge1\),
\begin{eqnarray}
\nonumber \mathbb E
\exp\left(\frac{2mM_n}{\sigma^2}\right)
&\le&
e^{2mR^2/\sigma^2}
\mathbb E
\exp\left(
\frac{2m}{\sigma^2}
\max_{i,k,k'}
|\langle x_k-x_{k'},Y_i\rangle|
\right) \\
\nonumber &\le&
e^{2mR^2/\sigma^2}
\sum_{i=1}^n\sum_{k,k'=1}^K
\mathbb E
\exp\left(
\frac{2m}{\sigma^2}
|\langle x_k-x_{k'},Y_i\rangle|
\right) \\
\label{eqn::m2} &\leq &n 2K^2e^{2mR^2/\sigma^2}\exp\left(\frac{8m^2R^2\varepsilon^2}{\sigma^4}\right).
\end{eqnarray}
Combining \eqref{eqn::prodexp}, \eqref{eqn::fbound} and \eqref{eqn::m2} we get \eqref{eqn:exppm}.
\end{proof}
\subsection{Proof of Theorem \ref{theo:barybound}}
\begin{proof}
Recall that
\[
\back(y)=\int x\,p(x,y)\,\ud P(x)
=
\sum_{k=1}^K \alpha_k x_k p(x_k,y),
\]
and
\[
\back_n(y)=\int x\,p_n(x,y)\,\ud P_n(x)
=
\sum_{k=1}^K \alpha_k^n x_k p_{n}(x_k,y),
\]
Similarly,
\[
\for(x_k)=\int y\,p(x_k,y)\,\ud Q(y),
\qquad
\for_n(x_k)=\int y\,p_{n}(x_k,y)\,\ud Q_n(y).
\]
From the facts that
\[
\int p_{k}(x_k,y)\,\ud Q(y)=1,
\int p_{n}(x_k,y)\,\ud Q_n(y)=1\text{ for } k\in [K], \quad \text{ and }
\sum_{k=1}^K \alpha^n_k p_{n}(x_k,Y_i)=1, \text{ for }1\leq i\leq n,
\]
It follows that for $k\in[K]$, \(\for_n(x_k)\) is a convex combination of the sample points
\(Y_1,\ldots,Y_n\) and that for every \(y\),
\(\back_n(y)\) and $\back(y)$ are convex combinations of the atoms $x_k,k\in[K]$.
We will use the following pointwise bound from the proof of Theorem
\(\ref{theo:densitibound}\):
\[
|p_{n}(x_k,y)-p(x_k,y)|
\le
\max\{p_{n}(x_k,y),p(x_k,y)\}
\frac{|f_n(x_k)-f(x_k)|+|g_n(y)-g(y)|}{\sigma^2}.
\]

Let $\Lambda_n$ be as in \eqref{eqn:lambda}. It follows from the bounds on $g_n(y)-g(y)$ in the proof of Theorem \ref{teo:potentialbound} that on $\Lambda_n$,
$$|g_n(y)-g(y)|\lesssim \lVert f_n-f\rVert_\infty +\lVert \alpha^n-\alpha\rVert_\infty.$$
Therefore, on \(\Lambda_n\), 
\begin{equation}\label{eq:pointwise-density-good-event}
\sup_{k\in[K]}\sup_{y\in\mathbb R^d}
|p_{n}(x_k,y)-p(x_k,y)|
\lesssim \lVert f_n-f\rVert_\infty+ \lVert \alpha^n-\alpha\rVert_\infty.
\end{equation}
We first prove the bound for \(\back_n-\back\). On \(\Lambda_n\), using
\eqref{eq:pointwise-density-good-event} and \(p(x_k,y)\le \underline\alpha^{-1}\),
\[
\begin{aligned}
\|\back_n(y)-\back(y)\|
&\le
\sum_{k=1}^K
\|x_k\|
\left|
\alpha_k^n p_{n}(x_k,y)-\alpha_k p(x_k,y)
\right|  \\
&\le
R\sum_{k=1}^K
\left[
\alpha_k^n |p_{n}(x_k,y)-p(x_k,y)|
+
|\alpha_k^n-\alpha_k|p(x_k,y)
\right]  \\
&\lesssim
 \lVert f_n-f\rVert_\infty+ \lVert \alpha^n-\alpha\rVert_\infty .
\end{aligned}
\]
Taking the supremum over \(y\),
\[
\mathbf 1_{\Lambda_n}
\|\back_n-\back\|_{L^\infty(Q)}^2
\lesssim 1_{\Lambda_n} \lVert f_n-f\rVert^2_\infty+ 1_{\Lambda_n}\lVert \alpha^n-\alpha\rVert^2_\infty.
\]
Hence, by the bounds on $\Lambda_n$ in the proof of Theorem \(\ref{teo:potentialbound}\) and the multinomial bound in
Lemma \(\ref{lemma:multinomial}\),
\[
\mathbb E\left[
\mathbf 1_{\Lambda_n}
\|\back_n-\back\|_{L^\infty(Q)}^2
\right]
\lesssim
\frac1n.
\]
On \(\Lambda_n^c\), both \(\back_n(y)\) and \(\back(y)\) are convex combinations of $x_k$'s. Therefore $\|\back_n-\back\|_{L^\infty(Q)}
\le 2R$ on $\Lambda_n^c$, and so, again by Lemma \ref{lemma:multinomial}
\[
\mathbb E\left[
\mathbf 1_{\Lambda_n^c}
\|\back_n-\back\|_{L^\infty(Q)}^2
\right]
\le
4R^2\mathbb P(\Lambda_n^c)
\lesssim e^{-c\underline\alpha n}.
\]
Combining the bounds on $\Lambda_n$ and $\Lambda^c_n$, we conclude the bound for the backward barycentric projection.
We now establish the bound for the forward difference \(\for_n-\for\). For each \(k\in[K]\), decompose
\[
\for_n(x_k)-\for(x_k)
=
\underbrace{\int (y\,p(x_k,y))\ud (Q_n-Q)(y)}_{A_{n,k}}
+
\underbrace{\int \left[y(p_{n}(x_k,y)-p(x_k,y))\right]\ud Q_n(y)}_{B_{n,k}}.
\]
To bound $A_{n,k}$, define the variables \[
Z_{i,k}:=Y_i p(x_k,Y_i),
\qquad
Z_k:=Yp(x_k,Y),
\]
where \(Y\sim Q\). Then
\[
A_{n,k}
=
\frac1n\sum_{i=1}^n
\left[Z_{i,k}-\mathbb E Z_k\right],
\]
and so, by independence
\[
\begin{aligned}
\mathbb E\|A_{n,k}\|^2
=
\mathbb E\left\|
\frac1n\sum_{i=1}^n
\left[Z_{i,k}-\mathbb E Z_k\right]
\right\|^2  
=
\frac1n
\mathbb E\left\|Z_k-\mathbb E Z_k\right\|^2  
\le
\frac1n
\mathbb E\|Z_k\|^2  
=
\frac1n
\mathbb E\|Yp_k(Y)\|^2 .
\end{aligned}
\]
Therefore, since \(p(x_k,y)\le \underline\alpha^{-1}\), and by subGaussianity,
\[
\begin{aligned}
\mathbb E\left[\max_{k\in[K]}\|A_{n,k}\|^2\right]
\le
\sum_{k=1}^K
\mathbb E\|A_{n,k}\|^2  
\le
\frac{K}{\underline{\alpha}^2 n} E\|Y\|^2
\leq
\frac{Kd\varepsilon^2}{\underline{\alpha}^2 n}.
\end{aligned}
\]
For the second term, define $\Delta_{n,k}(y):=p_{n}(x_k,y)-p(x_k,y)$ so that  \[ B_{n,k} = \frac1n\sum_{i=1}^n Y_i\Delta_{n,k}(Y_i). \] 
Note that, by Cauchy--Schwarz: \[ \begin{aligned} \|B_{n,k}\|^2 &= \left\| \frac1n\sum_{i=1}^n Y_i\Delta_{n,k}(Y_i) \right\|^2 \\ &\le  \left[\frac1n \sum_{i=1}^n \left\|Y_i\right\| \lvert \Delta_{n,k}(Y_i) \rvert \right]^2 \\ &\le \left[ \frac1n\sum_{i=1}^n \|Y_i\|^2 \right] \left[ \frac1n\sum_{i=1}^n \Delta_{n,k}(Y_i)^2 \right]. \end{aligned} \] Taking the maximum over \(k\in[K]\) and $y\in\mathbb{R}^d$ we obtain \[ \max_{k\in[K]}\|B_{n,k}\|^2 \le \frac1n\sum_{i=1}^n \|Y_i\|^2  \sup_{k\in[K],\,y\in\mathbb R^d} |p_{n}(x_k,y)-p(x_k,y)|^2. \] 

Additionally, on \(\Lambda_n\), by \eqref{eq:pointwise-density-good-event} we can bound further, getting  $$ \mathbf 1_{\Lambda_n} \max_{k\in[K]}\|B_{n,k}\|^2 \lesssim \left[\frac1n\sum_{i=1}^n \|Y_i\|^2\right]\left[\|f_n-f\|^2_{\infty} + \|\alpha^n-\alpha\|^2_\infty \right].$$

Taking expectations and applying Cauchy--Schwarz again, \[ \begin{aligned} \mathbb E\left[ \mathbf 1_{\Lambda_n} \max_{k\in[K]}\|B_{n,k}\|^2 \right]  &\le \E\left[\left[\frac1n\sum_{i=1}^n \|Y_i\|^2\right]^2\right]^{1/2}\E  \left[21_{\Lambda_n}\|f_n-f\|^4_{\infty} + 21_{\Lambda_n}\|\alpha^n-\alpha\|^4_\infty \right]^{1/2}. \end{aligned} \] 
Let's bound the two expectations on the right-hand side. For the first one, note that
\[
\begin{aligned}
\E\left[\left[\frac1n\sum_{i=1}^n \|Y_i\|^2\right]^2\right]
&=
\mathbb E\left(\frac1n\sum_{i=1}^n \|Y_i\|^2\right)^2  \\
&=
\frac1{n^2}
\left[
\sum_{i=1}^n \mathbb E \|Y_i\|^4
+
\sum_{i\neq j}\mathbb E(\|Y_i\|^2\|Y_j\|^2)
\right]  \\
&=
\frac1n \mathbb E\|Y\|^4
+
\frac{n-1}{n}
\left(\mathbb E\|Y\|^2\right)^2  \\
&\le
\mathbb E\|Y\|^4\\ &\le
d^2\varepsilon^4,
\end{aligned}
\]
where in the second-to-last inequality follows from Cauchy--Schwarz and in the last one an elementary moment bound for a  \(\varepsilon^2\)-subGaussian vector \cite{boucheron2003concentration}. Using Corollary \ref{cor:potential-high-moments} and the multinomial moment bound in Lemma \ref{lemma:multinomial} with $q=2$, we get
\[
\mathbb E\|f_n-f\|_{\infty}^4\lesssim n^{-2},\quad \text{and} \quad \mathbb E\|\alpha^n-\alpha\|_\infty^4\lesssim n^{-2}.
\]

Putting everything together,  \[ \mathbb E\left[ \mathbf 1_{\Lambda_n} \max_{k\in[K]}\|B_{n,k}\|^2 \right] \lesssim \frac{d}{n}. \]

It remains to control \(B_{n,k}\) on \(\Lambda_n^c\). Since $\for_n(x_k)$ is a convex combination of the sample points, and since $p(x_k,y)\leq \underline{\alpha}^{-1}$,
\[
\max_{k\in[K]}\left\|\int y p_n(x_k,y)\ud Q_n(y)\right\|
\le
\max_{1\le i\le n}\|Y_i\|,\quad\text{and}\quad \max_{k\in[K]}\left\|\int y p(x_k,y)\ud Q_n(y)\right\|
\le
\frac{1}{\underline{\alpha}} \frac{1}{n}\sum_{i=1}^n\|Y_i\|\le \underline{\alpha}^{-1} \max_{1\le i\le n}\|Y_i\|.
\]
%and that, by Cauchy--Schwarz and subGaussianity,
%\[
%\left\|\int y p(x_k,y)\ud Q(y)\right\|
%\le
%\left(\int \|y\|^2 \ud Q(y)\right)^{1/2}
%\left(\int p(x_k,y)^2\ud Q(y)\right)^{1/2}
%\le
%\underline\alpha^{-1}(\mathbb E\|Y\|^2)^{1/2}\le \underline\alpha^{-1}\sqrt{d}\varepsilon.
%\]
This implies that
\[
\mathbf 1_{\Lambda_n^c}
\max_{k\in[K]}\lVert B_{n,k}\rVert^2
\lesssim
\mathbf 1_{\Lambda_n^c}
\left[
\max_{1\le i\le n}\|Y_i\|^2
\right].
\]
Now, by Cauchy--Schwarz, the exponential bound in Lemma \ref{lemma:multinomial}, and a subGaussianity moment bound
\[
\begin{aligned}
\mathbb E\left[
\mathbf 1_{\Lambda_n^c}
\max_{k\in[K]}\|B_{n,k}\|^2
\right] & \le \mathbb E\left[
\mathbf 1_{\Lambda_n^c}
\max_{1\le i\le n}\|Y_i\|^2
\right]\\ 
&\le
\mathbb P(\Lambda_n^c)^{1/2}
\left(\mathbb E\max_{1\le i\le n}\|Y_i\|^4\right)^{1/2}  \\
&\le
\mathbb P(\Lambda_n^c)^{1/2}
\left(\sum_{i=1}^n\mathbb E\|Y_i\|^4\right)^{1/2}\\& 
\lesssim
e^{-c\underline\alpha n/2}\sqrt n\,d\varepsilon^2\lesssim \frac{d}{n},
\end{aligned}
\]

Combining the bounds for \(A_{n,k}\) and \(B_{n,k}\), we obtain
\[
\mathbb E
\|\for_n-\for\|_{\infty}^2
=
\mathbb E\max_{k\in[K]}\|\for_n(x_k)-\for(x_k)\|^2\lesssim  \E\max_{k\in[K]}\left[\lVert A_{n,k}\rVert^2 +\lVert B_{n,k}\rVert^2\right]
\lesssim
\frac{d}{n}.
\]
The proof is complete.
\end{proof}

\subsection{Proof of Proposition~\ref{prop:minimax_lower_bound}}
\label{app:minimax_lower_bound}

\begin{proof}
Fix $a\in\RR^d$ with $\|a\|=R$, and let
\[
r=\frac{1}{4\sqrt n}.
\]
Consider the two hypotheses
\[
P_0
=
\frac12\delta_{-a}+\frac12\delta_a,
\qquad
P_1
=
\left(\frac12-r\right)\delta_{-a}
+
\left(\frac12+r\right)\delta_a,
\]
together with the common source measure
\[
Q=\delta_0.
\]
Both pairs $(P_j,Q)$ belong to $\mathcal C$, since $r\leq1/4$,
$\delta_0$ is $\varepsilon^2$-subGaussian, and $\E_QY=0$.

Because $Q$ is supported on a single point, there is only one coupling
between $P_j$ and $Q$, namely
\[
\pi_j=P_j\otimes\delta_0.
\]
It is therefore the entropic optimal coupling for every $\sigma^2>0$.
Writing $T_j=\back_{P_j,Q}$, its backward barycentric projection satisfies
\[
T_j(0)=\E_{P_j}[X],
\]
and hence
\[
T_0(0)=0,
\qquad
T_1(0)=2ra.
\]
Therefore,
\[
\|T_1-T_0\|_{L^2(Q)}^2
=4r^2R^2.
\]

The $Q$-samples are identical under the two hypotheses, so the total
variation distance between the two experiments equals
\[
d_{\mathrm{TV}}
\left(P_0^n\otimes Q^n,P_1^n\otimes Q^n\right)
=d_{\mathrm{TV}}(P_0^n,P_1^n).
\]
Moreover,
\[
\KL(P_0\|P_1)
=-\frac12\log(1-4r^2)
\leq\frac83r^2,
\]
where the inequality uses $r\leq1/4$.
By tensorization and Pinsker's inequality,
\[
d_{\mathrm{TV}}(P_0^n,P_1^n)
\leq\sqrt{\frac n2\KL(P_0\|P_1)}
\leq2r\sqrt{\frac n3}
\leq\frac12.
\]

Let $\E_j$ denote expectation under $P_j^n\otimes Q^n$.
Le Cam's two-point lemma now gives
\[
\begin{aligned}
\inf_{\widehat T}
\sup_{j\in\{0,1\}}
\E_j\left[\|\widehat T-T_j\|_{L^2(Q)}^2\right]
&\geq
\frac{\|T_1-T_0\|_{L^2(Q)}^2}{8}
\left(1-d_{\mathrm{TV}}(P_0^n,P_1^n)\right)\\
&\geq\frac{4r^2R^2}{16}
=\frac{R^2}{64n}.
\end{aligned}
\]
Since both hypotheses belong to $\mathcal C$, this proves the claim.
\end{proof}

\section{Other technical lemmata for Section \ref{sec:sample_complexity}}
 
The following upper bound addresses a one-sample case, and it is helpful in the proof of Lemma \ref{lemma:varbound2}. 

\begin{lemma}\label{lemma:reversevar}
If $(f,g)$ and $(f',g')$ are the optimal potentials for $(P,Q)$ and $(P',Q)$ where $Q$ is $\varepsilon^2$ subGaussian and $P,P'$ are both semidiscrete supported on the same $x_k$'s and such that $P'\ll P$. Then,
$$\Var_{\infty}\left(f'-f\right)\leq\frac{4}{\kappa^2}\chi^2\left(P'\|P\right).$$
\begin{proof}
Using Corollary \ref{cor:strong} and similarly as in the proof of Lemma \ref{lemma:varbound1}, if we denote by $\Phi$ and $\Phi'$ the semidual functions related to $(P,Q)$ and $(P',Q)$, we have
 \begin{eqnarray*}
 \Var_{\infty}(f'-f)&\leq&\frac{2}{\kappa}\left(\Phi(f)-\Phi(f')\right)\\
 &\leq &\frac{2}{\kappa}\left(\Phi(f)-\Phi'(f)+\Phi'(f)-\Phi(f')\right)\\
 &\leq & \frac{2}{\kappa}\left(\Phi(f)-\Phi'(f)+\Phi'(f')-\Phi(f')\right),
 \end{eqnarray*}
 where in the last line we used the optimality of $f'$ for $\Phi'$. Now, note that most of the terms will cancel out as

 $$\Phi(f)=\int f(x) \ud P(x)- \int \Gamma_f(y) \ud Q(y),\quad \Phi'(f)=\int f(x) \ud P'(x)- \int \Gamma_{f}(y)\ud Q(y),$$
 and 
 $$\Phi'(f')=\int f'(x) \ud P'(x)- \int \Gamma_{f'}(y) \ud Q(y),\quad \Phi(f')=\int f'(x) \ud P(x)- \int \Gamma_{f'}(y)\ud Q(y),$$
 implying that 
 \begin{eqnarray*}
  \Var_{\infty}(f'-f)&\leq&\frac{2}{\kappa}\left(\int (f-f')(x)\ud (P-P')(x)\right)\\
  &\leq &\frac{2}{\kappa}\left(\frac{\theta}{2}\Var_{\alpha}\left(f'-f\right)+\frac{1}{2\theta}\chi^2(P'\|P)\right)\\
  &\leq &\frac{2}{\kappa}\left(\frac{\theta}{2}\Var_{\infty}\left(f'-f\right)+\frac{1}{2\theta}\chi^2(P'\|P)\right),
 \end{eqnarray*}
where the last line follows from \eqref{eq:varbound2} and the second to last from Young's inequality, Lemma H.1 in \cite{pooladian2023minimax}, that if $P'\ll P $, then for any function $f$ and $\theta>0$ (in the semidiscrete case we identify $\Var_P$ with $\Var_\alpha$)
 $$\int f \ud (P-P')(x)\leq \frac{\theta}{2}\Var_{P}(f)+\frac{1}{2\theta}\chi^2(P'\|P).$$
 We conclude the proof by taking $\theta=\kappa/2$
\end{proof}
\end{lemma}

\begin{proposition}[Dimension dependent bounds for potentials]\label{prop:potentialboundd}
Suppose that $P$ is supported on the ball $B(0,R)$ and that $Q$ is subGaussian with proxy $\varepsilon^2$. Then, under the gauge constraint \eqref{eq:gauge} the entropic optimal dual potential $f$ satisfies 
\[
|f(x)|
\;\lesssim\;
R^2 + d(\varepsilon^2 + R^2) + d^2(\varepsilon^2 + R^2)^2.
\]

where $\lesssim$ indicate constants independent of $R,\varepsilon,d$.
\end{proposition}

\begin{proof}
We will use Proposition 6 of \cite{mena2019statistical}, that the optimal entropic potentials $(f,g)$ between two subGaussian measures with same proxy $\lambda^2$ satisfy

\[- d\lambda^2\left(1+\frac{1}{2}\left(\|x\|+\sqrt{2d}\,\lambda\right)^2\right)-1
\le f(x)\le
\frac{1}{2}\left(\|x\|+\sqrt{2d}\,\lambda\right)^2,
\]
and 
\[- d\lambda^2\left(1+\frac{1}{2}\left(\|y\|+\sqrt{2d}\,\lambda\right)^2\right)-1
\le g(y)\le
\frac{1}{2}\left(\|y\|+\sqrt{2d}\,\lambda\right)^2,
\]
under the normalization constraint that $\E_P(f(X))=\E_Q(g(Y))=\frac{1}{2}S(P,Q).$

To apply this result we must identify a uniform subGaussianity proxy for $P$ and $Q$. Since $P$ is bounded, it is subGaussian with proxy $R^2$. Therefore, $P,Q$ are simultaneously subGaussian with proxy $\varepsilon^2+R^2$.  
Using the inequality $(a+b)^2 \le 2a^2 + 2b^2$, we obtain
\[
\left(\|x\| + \sqrt{2d}\,\lambda\right)^2
\le
2\|x\|^2 + 4d\lambda^2.
\]
Since $\|x\|\le R$, this yields
\[
\left(\|x\| + \sqrt{2d}\,\lambda\right)^2
\le
2R^2 + 4d\lambda^2.
\]

Substituting into the upper bound gives
\[
f(x)
\le
\frac{1}{2}(2R^2 + 4d\lambda^2)
\lesssim
R^2 + d\lambda^2.
\]

For the lower bound, we similarly obtain
\[
f(x)
\ge
- d\lambda^2\left(1 + \frac{1}{2}(2R^2 + 4d\lambda^2)\right) - 1
\lesssim
-\,d\lambda^2(1 + R^2 + d\lambda^2).
\]

Combining the two bounds yields
\[
|f(x)|
\lesssim
R^2 + d\lambda^2 + d^2\lambda^4.
\]

Finally, substituting $\lambda^2 = \varepsilon^2 + R^2$ gives the claim under the constraint$\E_P(f(X))=\frac{1}{2}S(P,Q)$. To achieve the final conclusion we define $f_{\text{new}}(x)=f(x)-\E_P(f(X))$. By definition, $f_{\text{new}}$ satisfies \eqref{eq:gauge} and
\begin{eqnarray*}
|f_{\text{new}}(x)|&\leq& |f(x)|+\sum_{k=1}^K \alpha_k f(x_k) \\ &\leq&  |f(x)|+\max_{k\in[K]} |f(x_k)|\sum_{k=1}^K\alpha_k\\ &\lesssim& 
2\lVert f\rVert_{\infty}\\
&\lesssim & R^2+d(\varepsilon^2+R^2)+d^2(\varepsilon^2+R^2)^2\end{eqnarray*}
by the bound on $f(x)$, so the proof is concluded
\end{proof}
\begin{lemma}\label{lemma:sigmaprob} Let $\tilde{\varepsilon}$ be the infimum of $\varepsilon_u$ such that $Q,Q_n$ are $\varepsilon_u^2$ subGaussian uniformly over $n$. Then, for any integer $m\geq 2$ there is a constant $C_m>0$ such that
 \begin{equation} \label{eq:sigmaprob} \Pb(\tilde{\varepsilon}^2\geq 3m\varepsilon^2)\leq C_m n^{- \frac{m}{2}}.\end{equation}
 \end{lemma}
 \begin{proof}
 The proof extends the one of Lemma A.3 in \cite{groppe2024lower}. Defining, for each $m$ integer
 
 $$\tau_{m,n}=\frac{1}{n} \sum_{i=1}^n \exp\left(\frac{\lVert Y_i\rVert^2}{2md\varepsilon^2}\right),\quad \tau_m=  \E\left(\tau_{m,n}\right)=\E\left(\exp\left(\frac{\lVert Y\rVert^2}{2md\varepsilon^2}\right)\right),$$
By Lemma \ref{lemma:MNW}(a), (with $\mu=Q_n$ and $\alpha =2md\varepsilon^2$) we conclude that $Q_n$ is $ 
 \tau_{m,n} m\varepsilon^2$-subGaussian and so $\tilde{\varepsilon}^2 \leq m\tau_{m,n}^2\varepsilon^2$. Therefore,
 $$\Pb(\tilde{\varepsilon}^2\geq 3m\varepsilon^2)\leq \Pb\left(\tau_{m,n}\geq 3\right).$$
 We now apply Markov's inequality to powers of centered $\tau_{m,n}$ (the other case is analogous). Since, By Jensen's inequality, $\tau_m,\leq 2^{1/m}\leq 2$, $(3-\tau_m)\geq 1$, and so we have
 $$\Pb\left(\tau_{m,n}\geq 3\right)\leq \Pb\left(\lvert\tau_{m,n}-\tau_m\rvert^ m\geq (3-\tau_m)^m\right)\leq \frac{\E\left(\lvert\tau_{m,n}-\tau_m\rvert^m\right)}{(3-\tau_m)^m}\leq \E\left(\lvert\tau_{m,n}-\tau_m\rvert^m\right).$$
 It only remains to bound the last term above. The case $m=2$ corresponds to Lemma A.3 in \cite{groppe2024lower}. Now, note that $\tau_{m,n}-\tau_m=\frac{1}{n} \sum_{i=1}^n X_i$, where $$X_i=\exp\left(\frac{\lVert Y_i\rVert^2}{2md\varepsilon^2}\right)-\E\left(\exp\left(\frac{\lVert Y\rVert^2}{2md\varepsilon^2}\right)\right)$$ is a zero mean variable. Note also that all moments up to order $m$ are bounded. Indeed, owing to that $|a-b|^k\leq 2^{k-1}(|a|^k+|b|^k)$ we get that for $m'\in\{2,m\}$, 
 \begin{eqnarray*} \E(|X_i|^{m'})&\leq& 2^{m'-1}\left(\E\left(\exp\left(\frac{m'\lVert Y_i\rVert^2}{2md\varepsilon^2}\right)\right)+\E\left(\exp\left(\frac{\lVert Y_i\rVert^2}{2md\varepsilon^2}\right)\right)^{m'}\right)\\
 &\leq & 2^{m'-1} \left(\E\left(\exp\left(\frac{\lVert Y_i\rVert^2}{2d\varepsilon^2}\right)\right)^{m'}+\E\left(\exp\left(\frac{\lVert Y_i\rVert^2}{2d\varepsilon^2}\right)\right)^{m'}\right)\\
 &\leq & 2^{m-1}\left(2^m+2^m\right)\\
 &\leq & 2^{2m}.
 \end{eqnarray*}
 by subGaussianity, the fact that $1/m\leq 1$, that $m'\geq 2$ and Jensen's inequality. Then, by Rosenthal's inequality (Lemma \ref{lemma:rosenthal}), for some constants $C_m$ (sometimes renamed from one to the next)
 \begin{eqnarray*} \E\left(\lvert\tau_{m,n}-\tau_m\rvert^m\right) &\leq& \frac{1}{n^m} C_m \max\Bigg\{n \E|X_i|^m,\left(  n\E|X_i|^2\right)^{m/2}\Bigg\} \\
 &\leq & C_m\left(n^{1-m}+n^{-m/2}\right)\\
  &\leq & C_mn^{-m/2}.
  \end{eqnarray*}
 where $C_{m,j}$ are polynomial expressions involving moments of $\tau_{m}$ up to order $m$. These moments are all bounded by 2, by subGaussianity.
 \end{proof}

The following the following $L^2$-type convergence of the empirical potentials at the rate $n^{-1}$:

\begin{lemma}\label{lemma:empirical}
For vectors $x_k\in B(0,R)\subseteq\mathbb{R}^d$, $k\in [K]$, and a vector $f=f(x_1,\ldots, x_k)\in\mathbb{R}^K$, define the function $\Gamma_f:\mathbb{R}^d\to \mathbb{R}$ as ($\sigma^2,\alpha$ are constants)
\begin{equation}\label{eq:empirical1}\Gamma_f(y):=\sigma^2 \log\left(\sum_{k=1}^K \alpha_k e^{(f_k-\frac{1}{2}\lVert y-x_k\rVert^2)/\sigma^2}\right).\end{equation}
Then, for any two vectors $f=f(x_1,\ldots, x_k), f'=f'(x_1,\ldots, x_k)\in\mathbb{R}^K$ we have
$$\E\left(\sup_{\Var_{\infty}(f'-f)\leq \tau^2} \Bigg \lvert\int \left(\Gamma_f(y)-\Gamma_{f'}(y)\right)\ud (Q-Q_n)(y)\Bigg \rvert \right)\leq C_{\underline{\alpha},R} \tau \sqrt{\frac{K}{n}},$$
for some constant $C_{\underline{\alpha},R}>0$.
Moreover, for $p\geq 2$ and a constant $C_{\underline{\alpha},R,p}>0$,
\begin{equation}\label{eq:empirical2}\E\left(\left[\sup_{\Var_{\infty}(f'-f)\leq \tau^2} \Bigg \lvert\int \left(\Gamma_f(y)-\Gamma_{f'}(y)\right)\ud (Q-Q_n)(y)\Bigg \rvert \right]^p\right)\leq C_{\underline{\alpha},R,p} \left(\tau \sqrt{\frac{K}{n}}\right)^p.\end{equation}

\end{lemma}
\begin{proof}
This is simply a re-statement of Lemma F.1 in \cite{pooladian2023minimax}, whose proof is based on controlling the covering numbers of the family of functions $\Gamma_f$ along with a generic bound on the empirical process (Lemma H.3 in \cite{pooladian2023minimax}), and we only comment on two minor differences: first, our $\Gamma_u$ includes dependence in $\alpha_k$. As we assume that $\underline{\alpha}>0$ throughout, nothing substantially changes. Second, our definition of $\Gamma_u(y)$ contains the term $\lVert x_k-y\rVert^2$ instead of an inner product. This is immaterial because i) the quadratic terms $\lVert y\rvert^2 $ cancel out in the difference $\Gamma_f(y)-\Gamma_{f'}(y)$, and the terms $\lVert x_k\rVert^2$ will at worst induce a dependency of the constant in $R$.
\end{proof}

 \begin{lemma}\label{lemma:multinomial} Let $P$ be a discrete measure supported on $K$ atoms with weights $\alpha$, and let $P_n$ be the corresponding empirical measure with weights $\alpha^n$. Then,  
$$\E\left(\chi^2(P_n\|P)\right)=\frac{K-1}{n}.$$
Consequently,
$$\E\left(\lVert \alpha^n-\alpha\lVert^2_\infty\right)\leq \E\left(\chi^2(P_n\|P)\right)\leq \frac{K-1}{n}.$$
Moreover, for each $q\geq 1$, $$\mathbb E\|\alpha^n-\alpha\|_\infty^{2q}
\lesssim n^{-q},$$
where the underlying constant only depends on $K$ and $q$.
Additionally, let $\Lambda_n$ be the event $\alpha^n_k\geq \alpha_k/2,\forall k\in[K]$.  Then, 
$$\Pb(\Lambda^c_n)\leq K \exp(-\underline{\alpha}nc),$$
for some $c>0$.
\end{lemma} 

\begin{proof}
We only show the statement for arbitrary $q$, the other ones are essentially Lemmas E.2 and H.2 in \cite{pooladian2023minimax}.
For each \(k\in[K]\), write
\[
\alpha_k^n-\alpha_k
=
\frac1n\sum_{i=1}^n \xi_{i,k},
\qquad
\xi_{i,k}:=\mathbf 1\{X_i=x_k\}-\alpha_k.
\]
The  \(\xi_{i,k}\)'s are i.i.d., centered, and satisfy $|\xi_{i,k}|\le 1$.  By Rosenthal's inequality (Lemma \ref{lemma:rosenthal}), and using that all moments $\E|\xi_{i,k}|^{m}\leq 1$, we obtain that for every fixed \(q\ge1\),
\[
\mathbb E\left|\sum_{i=1}^n \xi_{i,k}\right|^{2q}
\leq C_{2q}
\max\left\{
\sum_{i=1}^n \mathbb E|\xi_{i,k}|^{2q},
\left(\sum_{i=1}^n \mathbb E\xi_{i,k}^2\right)^q
\right\} \lesssim n^q.
\]
Dividing by \(n^{2q}\), we get
$
\mathbb E|\alpha_k^n-\alpha_k|^{2q}
\lesssim n^{-q}$ and
consequently,
\[
\mathbb E\|\alpha^n-\alpha\|_\infty^{2q} = \E \left[\max_{k\in[K]}|\alpha_k^n-\alpha_k|^{2q}\right]
\le 
\sum_{k=1}^K
\mathbb E|\alpha_k^n-\alpha_k|^{2q}
\lesssim n^{-q}.
\]
\end{proof}
\begin{lemma}[Rosenthal's inequality][Theorem 3 in \cite{rosenthal1970subspaces}] \label{lemma:rosenthal} Let $X_i\in\mathbb{R}^d$ be an i.i.d sequence of zero mean random variables with finite $m$-th moment, $m\geq 2$. Then,
$$\E\Big\lvert \sum_{i=1}^n X_i\Big\rvert^m \leq C_m\max\Bigg\{\sum_{i=1}^n \E|X_i|^m,\left(\sum_{i=1}^n \E|X_i|^2\right)^{m/2}\Bigg\},$$
where $C_m$ is a constant that depends only on $m$.
\end{lemma}

 \begin{lemma}[Lemmas 2 and 4 in \cite{mena2019statistical}]\label{lemma:MNW}

 The following statements hold:
\begin{itemize}
    \item[(a)] 
For each $\alpha>0$, if $t=\E_\mu\left(\exp\left(\frac{\lVert Y\rVert^2}{\alpha}\right)\right)$ is finite, then $\mu$ is $t\frac{\alpha}{2d}$-subGaussian. 
\item[(b)] Consequently, if $Q$ is $\varepsilon^2$ subGaussian then $Q_n$ is subGaussian with parameter 
$$\varepsilon^2_n := \frac{1}{n}\sum_{i=1}^n\exp\left(\frac{\lVert Y_i\rVert^2}{2d\varepsilon^2}\right).$$
\item[(c)] $Q,Q_n$ are uniformly subGaussian for some a.e. finite random subGaussianity proxy $\varepsilon^2_u$
\item[(d)] If $\tilde{\varepsilon}^2$ is the smallest such subGaussianity proxy, then for each integer $m$
$$\E(\tilde{\varepsilon}^{2m})\leq 2m^m\varepsilon^{2m}.$$
\end{itemize}
\end{lemma}

Essentially, the above variance plays the role of the usual Euclidean norm, but accounts for the fact that potentials are only defined up to additive constants. Finally, we are able to pass from the convergence of potentials to barycentric projections by relying on the following stability bound

\section{Proofs and auxiliary lemmata for Sinkhorn-EM}\label{app:sinkhorn}

In this appendix we present the proofs for the convergence sample-based Sinkhorn-EM, as established in \Cref{thm:convergence-sample-SEM}. We will denote $\mathcal{B} := B(\theta^*, \|\theta^*\|/4)$. 

We start with the proof of the main result, and then we will prove the intermediate results required for it. 

\subsection{Proof of \Cref{thm:convergence-sample-SEM}}

\begin{proof}
The proof uses the standard argument from finite-sample analyses of EM \citep{BalWaiYu17,Dwivedi2018}; we include it for completeness. The only difference is that we use the uniform deviation bound from \Cref{prop:sample-based-sem-iterates}. 

Step 1: Conditioning on the Statistical Concentration. \\
By \Cref{prop:sample-based-sem-iterates}, there exists an event holding with probability at least $1 - \delta - n^{-c_1d}-c_2n^{-2}$ upon which the empirical updates converge uniformly to the population updates over the entire basin $\mathcal{B}$. On this event, the maximum statistical error is
\begin{equation}
    \eta_n := \sup_{\theta \in  \mathcal{B}} \|F_n(\theta, \alpha_n(\theta)) - F(\theta, \alpha(\theta))\| \le C \sqrt{\frac{d \log(n) + \log(1/\delta)}{n}}.
\end{equation}
We condition on this high-probability event for all subsequent steps.

Step 2: Inductive Hypothesis \\
We propose that for any iteration $t \in \{0, \dots, T\}$, the following two conditions hold:
\begin{enumerate}
    \item \textbf{Basin Stability:} $\hat\theta^t_{SEM} \in \mathcal{B}$.
    \item \textbf{Error Bound:} $\|\hat\theta^t_{SEM} - \theta^\ast\| \le \kappa^t \|\hat\theta^0_{SEM} - \theta^\ast\| + \eta_n \sum_{i=0}^{t-1} \kappa^i$.
\end{enumerate}

\textit{Base Case ($t=0$):} By the initialization assumption, $\|\hat\theta^0_{SEM} - \theta^\ast\| \le \|\theta^\ast\|/4$, so the iterate is in the basin. The summation in the error bound is empty, so the condition holds trivially.

Step 3: Inductive Step \\
Assume the hypothesis holds for step $t$. We analyze the error at $t+1$ by decomposing it into statistical error and population-level contraction using the triangle inequality:
\begin{align}
    \|\hat\theta^{t+1}_{\mathrm{SEM}} - \theta^\ast\| &= \|F_n(\hat\theta^t_{\mathrm{SEM}}, \alpha_n(\hat\theta^t_{\mathrm{SEM}})) - \theta^\ast\| \nonumber \\
    &\le \underbrace{\|F_n(\hat\theta^t_{\mathrm{SEM}}, \alpha_n(\hat\theta^t_{\mathrm{SEM}})) - F(\hat\theta^t_{\mathrm{SEM}}, \alpha(\hat\theta^t_{\mathrm{SEM}}))\|}_{\text{Statistical Error}} + \underbrace{\|F(\hat\theta^t_{\mathrm{SEM}}, \alpha(\hat\theta^t_{\mathrm{SEM}})) - \theta^\ast\|}_{\text{Population Contraction}}. \label{eq:triangle_long}
\end{align}

We bound these terms individually:
\begin{enumerate}
    \item Statistical Error: Because $\hat\theta^t \in \mathcal{B}$ by the inductive hypothesis, this term is bounded by the uniform supremum $\eta_n$ from Step 1.
    \item Population Contraction: By \Cref{prop:SEMcoincideEM}, the population Sinkhorn-EM and EM operators coincide under our Gaussian mixture model. Furthermore, the population EM operator $F$ is a uniform $\kappa$-contraction for all $\theta \in \mathcal{B}$ under our assumptions as proved in \cite{BalWaiYu17}. Thus, $\|F(\hat\theta^t, \alpha(\hat\theta^t)) - \theta^\ast\| \le \kappa \|\hat\theta^t - \theta^\ast\|$.
\end{enumerate}

Substituting these back into \eqref{eq:triangle_long} and applying the inductive hypothesis for step $t$:
\begin{align}
    \|\hat\theta^{t+1}_{SEM} - \theta^\ast\| &\le \eta_n + \kappa \left( \kappa^t \|\hat\theta^0 - \theta^\ast\| + \eta_n \sum_{i=0}^{t-1} \kappa^i \right) \nonumber \\
    &= \kappa^{t+1} \|\hat\theta^0 - \theta^\ast\| + \eta_n \sum_{i=0}^{t} \kappa^i.
\end{align}
This verifies the recursive error bound for step $t+1$.

To complete the induction, we must show $\hat\theta^{t+1}_{SEM}$ does not leave $\mathcal{B}$. The maximum distance reached by the error bound is:
\begin{equation}
    \|\hat\theta^{t+1}_{SEM} - \theta^\ast\| \le \kappa^{t+1} \|\hat\theta^0 - \theta^\ast\| + \frac{\eta_n}{1-\kappa}.
\end{equation}
Since $\kappa < 1$ and $\|\hat\theta^0 - \theta^\ast\| \le \|\theta^\ast\|/4$, a sufficient condition for staying in the basin is $\eta_n + \kappa  \|\theta^\ast\|/4 \le \|\theta^\ast\|/4$, which simplifies to $\eta_n \le (1-\kappa) \|\theta^\ast\|/4$. Since $\eta_n$ decrease as $n^{-1/2}$,  for $n$ sufficiently large, the statistical noise $\eta_n$ is strictly bounded such that this inequality holds. Thus, $\hat\theta^{t+1}_{SEM} \in \mathcal{B}$, closing the induction.

The final result follows by substituting the sum of the geometric series and the definition of $\eta_n$.
\end{proof}

The main ingredients to prove the theorem above are \Cref{prop:sample-based-sem-iterates} (statistical error control) and \Cref{prop:SEMcoincideEM} (population contraction). We start proving \Cref{prop:sample-based-sem-iterates}.

\subsection{Proof of \Cref{prop:sample-based-sem-iterates}}

For this proof, we use many intermediate results which are deferred to the next subsection.   

\begin{proof}[Proof of \Cref{prop:sample-based-sem-iterates}.]
By the triangle inequality, we decompose the error as follows
\begin{align*}
    \|F_n(\theta,\alpha_n(\theta))-F(\theta,\alpha(\theta))\|
&\leq
\|F_n(\theta,\alpha_n(\theta))-F(\theta,\alpha_n(\theta))\|\\
&+ \|F(\theta,\alpha_n(\theta)) - \E[F(\theta,\alpha_n(\theta))]\|\\
&+ \|\E[F(\theta,\alpha_n(\theta))]-F(\theta,\alpha(\theta))\|.
\end{align*}

Now, we bound each term separately. 

\paragraph{First term.} We start bounding the first term in the decomposition with  \Cref{cor:data_dependent_alpha}. \Cref{lem:alpha_n_bound} gives a high probability bound for $\alpha_n(\theta)$ away from $0$ and $1$. Therefore, since $\mathcal{B} \subset B(0,2\|\theta^*\|)$, \Cref{cor:data_dependent_alpha} applies to the data-dependent map
\(\widehat\alpha_n=\alpha_n\) on $\mathcal T=\mathcal B$. It gives that, with probability at least
\(1-n^{-cd} - Cn^{-2}\)
\[
\sup_{\theta\in \mathcal{B}}
\left\|
F_n(\theta,\alpha_n(\theta))
-
F(\theta,\alpha_n(\theta))
\right\|
\le
C
\left(
\|\theta^*\|
+
\frac{\|\theta^*\|^2}{\sigma^2} + 1
\right)
\sqrt{\frac{d\log n}{n}}.
\]
Absorbing the displayed factor into the constant \(C\), this is with probability at least
\(1-n^{-cd} - Cn^{-2}\)
\[
\sup_{\theta\in \mathcal{B}}
\left\|
F_n(\theta,\alpha_n(\theta))
-
F(\theta,\alpha_n(\theta))
\right\|
\le
C\sqrt{\frac{d\log n}{n}}.
\]

\paragraph{Second term.} The bound for second term of the decomposition is a consequence of \Cref{prop:uniform-centered-sinkhorn-update}, which states that with probability at least $1-\delta-Cn^{-2}$,
\[\sup_{\theta\in \mathcal{B}}
\left\|
F(\theta,\alpha_n(\theta))
-
\E[F(\theta,\alpha_n(\theta))]
\right\|
\le
C
\max\left\{\sigma,\|\theta^\ast\|\right\}
\exp\left(C\frac{\|\theta^\ast\|^2}{\sigma^2}\right)
\sqrt{\frac{d\log(en)+\log(1/\delta)}{n}}.
\]

\paragraph{Third term.} This term was already handled in \Cref{subsec:finite_sample_SEM}. There, we showed that 
\[\left\|\E[F(\theta,\alpha_n(\theta))]-F(\theta,\alpha(\theta))\right\|
\le
\frac{1}{\|\theta\|}
\left(\E\|Y\|^2\right)^{1/2}
\left(
\E\left\|\back_n^\theta(Y)-\back^\theta(Y)\right\|^2
\right)^{1/2}.\]
Moreover, under the balanced Gaussian mixture,
\[
\E\|Y\|^2=d\sigma^2+\|\theta^\ast\|^2,
\]
and 
\[\frac{1}{\|\theta\|} \lesssim \frac{1}{\|\theta^*\|}\]
for all $\theta\in \mathcal{B}$. Thus, by the barycentric-map bound from \Cref{theo:barybound}, 
\[
\sup_{\theta\in \mathcal{B}}
\left\|
\E[F(\theta,\alpha_n(\theta))]
-
F(\theta,\alpha(\theta))
\right\|
\lesssim
\frac{(d\sigma^2+\|\theta^\ast\|^2)^{1/2}}
{\|\theta^\ast\|}
\,\sqrt{\frac{1}{n}}.
\]
After absorbing the factors depending on
\(\sigma,\|\theta^\ast\|\) into \(C\),
\[
\sup_{\theta\in \mathcal{B}}
\left\|
\E[F(\theta,\alpha_n(\theta))]
-
F(\theta,\alpha(\theta))
\right\|
\le
C
\sqrt{\frac{d}{n}}.
\]

Combining the three bounds and absorbing constants, we obtain that with probability at least \(1-n^{-c_1d}-c_2n^{-2}-\delta\),
\[
\sup_{\theta\in \mathcal{B}}
\left\|
F_n(\theta,\alpha_n(\theta))
-
F(\theta,\alpha(\theta))
\right\|
\le
C\sqrt{\frac{d\log n + \log(1/\delta)}{n}},
\]
as claimed.
\end{proof}

\subsection{Proof of \Cref{prop:SEMcoincideEM}}

\begin{proof}
Since \(\alpha=\frac12\), the population data distribution is
\[
Q^\ast
=
\frac12\mathcal N(\theta^\ast,\sigma^2 I_d)
+
\frac12\mathcal N(-\theta^\ast,\sigma^2 I_d),
\]
which is symmetric under \(y\mapsto -y\).

We first show that, for every \(\theta\in\mathbb R^d\),
\[
\alpha(\theta)=\frac12.
\]
Recall that \(\alpha(\theta)\) is defined from an optimal semidual
potential
\[
f(\theta)=(f_1(\theta),f_2(\theta))\in\mathbb R^2
\]
by
\[
\alpha(\theta)
=
\frac{\alpha e^{f_1(\theta)/\sigma^2}}
{\alpha e^{f_1(\theta)/\sigma^2}
 +(1-\alpha)e^{f_2(\theta)/\sigma^2}}.
\]
Since here \(\alpha=\frac12\), it is enough to prove that
\[
f_1(\theta)=f_2(\theta).
\]

Let
\[
P_\theta=\frac12\delta_\theta+\frac12\delta_{-\theta}.
\]
The semidual objective for \(S(P_\theta,Q^\ast)\) is
\[
\Phi_\theta(f_1,f_2)
=
\frac12 f_1+\frac12 f_2
-\sigma^2
\int
\log\!\left(
\frac12
e^{(f_1-\frac12\|y-\theta\|^2)/\sigma^2}
+
\frac12
e^{(f_2-\frac12\|y+\theta\|^2)/\sigma^2}
\right)
\,\ud Q^\ast(y).
\]
Using the symmetry of \(Q^\ast\) and the identities
\[
\|-y-\theta\|^2=\|y+\theta\|^2,
\qquad
\|-y+\theta\|^2=\|y-\theta\|^2,
\]
a change of variables \(y\mapsto -y\) yields
\[
\Phi_\theta(f_1,f_2)=\Phi_\theta(f_2,f_1)
\qquad
\text{for all }(f_1,f_2)\in\mathbb R^2.
\]

Now let
\[
f(\theta)=(f_1(\theta),f_2(\theta))
\]
be the optimal semidual potential satisfying the gauge condition
\eqref{eq:gauge},
\[
\mathbb E_{X\sim P_\theta}[f(X)]=0.
\]
Since
\[
P_\theta=\frac12\delta_\theta+\frac12\delta_{-\theta},
\]
this condition becomes
\[
\frac12 f_1(\theta)+\frac12 f_2(\theta)=0.
\]
Thus the gauge condition is invariant under swapping the two
coordinates.

Because
\[
\Phi_\theta(f_1,f_2)=\Phi_\theta(f_2,f_1),
\]
if \((f_1(\theta),f_2(\theta))\) is an optimal potential satisfying
the gauge condition, then
\((f_2(\theta),f_1(\theta))\) is also an optimal potential satisfying
the same gauge condition. Since optimal semidual potentials are
unique up to additive constants, and the gauge condition fixes this
constant uniquely, we must have
\[
(f_1(\theta),f_2(\theta))
=
(f_2(\theta),f_1(\theta)).
\]
Hence
\[
f_1(\theta)=f_2(\theta).
\]
In fact, together with the gauge condition, this also gives
\[
f_1(\theta)=f_2(\theta)=0.
\]
Therefore,
\[
\alpha(\theta)=\frac12
\qquad
\text{for every }\theta\in\mathbb R^d.
\]

It follows that the population Sinkhorn-EM update reduces to
\[
\theta^{t+1}_{\mathrm{SEM}}
=
F\left(
\theta^t_{\mathrm{SEM}},
\alpha(\theta^t_{\mathrm{SEM}})
\right)
=
F\left(
\theta^t_{\mathrm{SEM}},\frac12
\right).
\]
But \(F(\theta,\frac12)\) is exactly the population EM update map for
the balanced symmetric two-Gaussian model \cite{xu2016globalEM,BalWaiYu17}. Since the two algorithms have the
same initialization and update map, induction yields
\[
\theta^t_{\mathrm{SEM}}=\theta^t_{\mathrm{EM}}
\]
for every \(t\geq 0\).
\end{proof}

\subsection{Other technical lemmata for Proposition \ref{prop:sample-based-sem-iterates}}

The bound of the first term in the decomposition in the proof of \Cref{prop:sample-based-sem-iterates} relies on the a Corollary from a result in \cite{weinberger2022algorithm} along with a bound on $\alpha_n(\theta)$, which we prove below.

\begin{corollary}[Mean-update concentration with a data-dependent weight]\label{cor:data_dependent_alpha}
Assume $\|\theta^\ast\|/\sigma\le C_\theta$, and let
$\mathcal T\subseteq B(0,2\|\theta^\ast\|)$ be deterministic.
Let $\widehat\alpha_n:\mathcal T\to[0,1]$ be any possibly data-dependent
map such that, for some $\alpha_0\in(0,1/2)$, with probability at least
$1-\beta$,
$\widehat\alpha_n(\theta)\in[\alpha_0,1-\alpha_0]$ for every
$\theta\in\mathcal T$. Set $C_\rho:=1-2\alpha_0$.
There exist constants $C,c>0$, depending only on $C_\theta$ and $C_\rho$,
such that for $n\ge C d\log n$ and $n\ge2$, with probability at least
$1-n^{-cd}-\beta$,
\[
\sup_{\theta\in\mathcal T}
\|F_n(\theta,\widehat\alpha_n(\theta))
 -F(\theta,\widehat\alpha_n(\theta))\|
\le C\left(\|\theta^\ast\|
 +\frac{\sigma}{2}\log\frac{1-\alpha_0}{\alpha_0}\right)
\sqrt{\frac{d\log n}{n}}.
\]
\end{corollary}

\begin{proof}[Proof of \Cref{cor:data_dependent_alpha}]
Apply Theorem~3 of \cite{weinberger2022algorithm} after rescaling
$Y$ and $\theta$ by $\sigma$. Its uniform mean-update bound, together with
$|2\alpha-1|\le\frac12|\log(\alpha/(1-\alpha))|$, gives, with probability
at least $1-n^{-cd}$,
\[
\|F_n(\theta,\alpha)-F(\theta,\alpha)\|
\le C\left(\|\theta\|
 +\frac{\sigma}{2}\left|\log\frac{\alpha}{1-\alpha}\right|\right)
\sqrt{\frac{d\log n}{n}}
\]
simultaneously for $\theta\in B(0,2\|\theta^\ast\|)$ and
$\alpha\in[\alpha_0,1-\alpha_0]$.
Intersect this event with the assumed weight-control event and evaluate the
uniform bound at $\alpha=\widehat\alpha_n(\theta)$ for $\theta\in\mathcal T$.
Since $\|\theta\|\le2\|\theta^\ast\|$ and
$|\log(\alpha/(1-\alpha))|\le\log((1-\alpha_0)/\alpha_0)$ on this interval,
the claim follows by a union bound and an adjustment of $C$.
\end{proof}

\begin{lemma}\label{lem:alpha_n_bound}
Let \(Q^\ast\) be as in \eqref{eq:mixture_of_gaussians} with
\(\alpha=1/2\), and assume \(\|\theta^\ast\|\le C_\theta\). For each
\(\theta\in\mathcal B\), let \(f_n(\theta)\in\mathbb R^2\) be an
optimal semidual potential for \(S(P_\theta,Q_n)\), where
\(P_\theta=\frac12\delta_\theta+\frac12\delta_{-\theta}\), and define
\[
\alpha_n(\theta)
=
\frac{
\frac12\exp(f_{n,1}(\theta)/\sigma^2)
}{
\frac12\exp(f_{n,1}(\theta)/\sigma^2)
+
\frac12\exp(f_{n,2}(\theta)/\sigma^2)
}.
\]
Then, with probability at least \(1-Cn^{-2}\), for every
\(\theta\in\mathcal B\),
\[
\alpha_n(\theta)\in[\alpha_0,1-\alpha_0],
\]
where
\[
\alpha_0
:=
\frac{1}{
1+\exp\left(
2C_0\left[\|\theta^\ast\|^2/\sigma^2+1\right]
\right)
},
\]
and \(C_0>0\) is a constant depending only on the fixed problem
parameters.
\end{lemma}

\begin{proof}[Proof of \Cref{lem:alpha_n_bound}]
Let \(\tilde\varepsilon\) be the smallest constant such that \(Q_n\) and
\(Q^\ast\) are uniformly subGaussian. By Lemma~\ref{lemma:MNW}, this is finite
almost surely. Applying Proposition~\ref{prop:potentialbound} to the problem
\((P_\theta,Q_n)\), and using that \(P_\theta\) has two atoms with weights
\(1/2\), gives
\[
\|f_n(\theta)\|_\infty
\lesssim
\|\theta\|^2
+
\sigma^2\log 2
+
\frac{\tilde\varepsilon^2}{\sigma^2}\|\theta\|^2 .
\]
Define the event
\[
E_n:=\{\tilde\varepsilon^2<12\varepsilon^2\}.
\]
By Lemma~\ref{lemma:sigmaprob} with \(m=4\),
\[
\mathbb P(E_n^c)\lesssim n^{-2}.
\]
On \(E_n\), uniformly over \(\theta\in\mathcal B\),
\[
\|f_n(\theta)\|_\infty
\le
C_0\bigl(\|\theta^\ast\|^2+\sigma^2\bigr)
=:L,
\]
for some constant \(C_0>0\) depending only on the fixed problem parameters,
where we used
\[
\|\theta\|
\le
\|\theta^\ast\|+\|\theta-\theta^\ast\|
\le
\frac54\|\theta^\ast\|.
\]

Since \(\alpha=1/2\), the Sinkhorn weight satisfies
\[
\frac{\alpha_n(\theta)}{1-\alpha_n(\theta)}
=
\exp\left(
\frac{f_{n,1}(\theta)-f_{n,2}(\theta)}{\sigma^2}
\right).
\]
Therefore, on \(E_n\),
\[
e^{-2L/\sigma^2}
\le
\frac{\alpha_n(\theta)}{1-\alpha_n(\theta)}
\le
e^{2L/\sigma^2}
\qquad
\text{for every } \theta\in\mathcal B.
\]
It follows that
\[
\alpha_n(\theta)\in[\alpha_0,1-\alpha_0]
\qquad
\text{for every } \theta\in\mathcal B,
\]
where
\[
\alpha_0
:=
\frac{1}{1+\exp(2L/\sigma^2)}
=
\frac{1}{
1+\exp\left(
2C_0\left[\|\theta^\ast\|^2/\sigma^2+1\right]
\right)
}.
\]
\end{proof}

The bound of the second term in the decomposition in the proof of \Cref{prop:sample-based-sem-iterates} is given below. 

\begin{proposition}[Uniform concentration of the Sinkhorn-corrected population update]
\label{prop:uniform-centered-sinkhorn-update}
Assume \eqref{eq:mixture_of_gaussians} holds with $\alpha=1/2$ and
$\theta^\ast\neq 0$. For each $\theta\in\B$, let $\alpha_n(\theta)$ be
the empirical Sinkhorn-corrected weight, and let $\alpha_0$ be as in
\Cref{lem:alpha_n_bound}. There exists a constant $C>0$, depending only
on $\alpha_0$, such that for every $\delta\in(0,1)$, with probability at least
$1-\delta-Cn^{-2}$,
\[
\sup_{\theta\in\B}
\left\|
F(\theta,\alpha_n(\theta))
-
\E_{Y_1,\ldots,Y_n}
\bigl[F(\theta,\alpha_n(\theta))\bigr]
\right\|
\le
C
\max\left\{\sigma,\|\theta^\ast\|\right\}
\exp\left(C\frac{\|\theta^\ast\|^2}{\sigma^2}\right)
\sqrt{
\frac{d\log(en)+\log(1/\delta)}{n}
}.
\]
\end{proposition}

\begin{proof}
The proof begins with a ghost-sample symmetrization. To control the
deviation of $F(\theta,\alpha_n(\theta))$ from its expectation, we introduce
a weight $\alpha_n'(\theta)$ estimated from an independent sample.
Averaging $F(\theta,\alpha_n'(\theta))$ over that sample recovers the
expectation we want to subtract, so Jensen's inequality reduces the
problem to comparing $F(\theta,\alpha_n(\theta))$ and
$F(\theta,\alpha_n'(\theta))$.
Both weights solve empirical versions of the same population calibration
equation. We show that this equation is stable: small sampling errors
lead to small changes in its solution, following the basic principle
of $Z$-estimation
\cite[Chapter~5, Sections~5.2--5.3]{van2000asymptotic}.
We then show that these small changes in the weight produce small changes
in $F$, giving the desired bound after averaging over the independent
sample. All bounds below are uniform over $\theta\in\mathcal B$.

\paragraph{Step 1: The Sinkhorn calibration equation.}
Recall that
\[
\Psi(y,\theta,\alpha)
:=
\frac{
\alpha \exp\!\left(-\frac{\|y-\theta\|^2}{2\sigma^2}\right)
}{
\alpha \exp\!\left(-\frac{\|y-\theta\|^2}{2\sigma^2}\right)
+
(1-\alpha)\exp\!\left(-\frac{\|y+\theta\|^2}{2\sigma^2}\right)
}\]
Here $\alpha$ is a variable responsibility weight; the true mixture weight remains fixed at $1/2$. Define
\[
G(\theta,\alpha)
:=
\E_{Y\sim Q^\ast}\Psi(Y,\theta,\alpha),
\qquad
G_n(\theta,\alpha)
:=
\frac1n\sum_{i=1}^n\Psi(Y_i,\theta,\alpha).
\]
\[
Z_n:=\sup_{\theta\in\B}\sup_{\alpha\in I}
|G_n(\theta,\alpha)-G(\theta,\alpha)|,
\]
where $I:=[\alpha_0,1-\alpha_0]$. Let $g_n^\theta$ be the continuous-side potential paired with $f_n(\theta)$.
The first marginal equation \eqref{eqn:fstar_from_gstar}, applied to
$(P_\theta,Q_n)$ at $x=\theta$, reads
\[
1=\frac1n\sum_{i=1}^n
\exp\left(\frac{f_{n,1}(\theta)+g_n^\theta(Y_i)-c(\theta,Y_i)}{\sigma^2}\right).
\]
Substituting the empirical version of \eqref{eqn:gstar_from_fstar} into
this identity and using the definition of $\alpha_n(\theta)$ gives
\begin{equation}\label{eq:sem-empirical-calibration}
G_n(\theta,\alpha_n(\theta))=\frac12.
\end{equation}
Indeed, each exponential in the preceding sum equals
$2\Psi(Y_i,\theta,\alpha_n(\theta))$ because the discrete marginal has
weights $1/2,1/2$. The identical calculation applies to an independent
sample. Thus $\alpha_n(\theta)$ is an exact root of
$G_n(\theta,\alpha)-1/2=0$.

\paragraph{Step 2: Uniform control of the calibration map.} Let $\mathcal G
:=
\left\{
y\mapsto\Psi(y,\theta,\alpha):
\theta\in\B,\ \alpha\in I
\right\}$. Since
\[
\Psi(y,\theta,\alpha)
=
\ell\left(
\log\frac{\alpha}{1-\alpha}
+
\frac{2\langle y,\theta\rangle}{\sigma^2}
\right),
\qquad
\ell(t):=\frac{1}{1+e^{-t}},
\]
the inner class is an affine class in $y$ with parameters
$(\theta,\alpha)$ of dimension $O(d)$. More explicitly, for
$u\in(0,1)$, the subgraph condition
$u\le \Psi(y,\theta,\alpha)$ is equivalent to
\[
\log\frac{u}{1-u}
\le
\log\frac{\alpha}{1-\alpha}
+
\frac{2\langle y,\theta\rangle}{\sigma^2},
\]
which is a halfspace condition after the fixed transformation
\[
(y,u)
\longmapsto
\left(y,\log\frac{u}{1-u}\right).
\]
Hence $\mathcal G$ is a bounded VC-subgraph class with VC-subgraph
dimension $O(d)$; see, e.g.,
\cite[Lemmata~2.6.15 and~2.6.18]{wellner2013weak}.
The VC entropy bound
\citep[Theorem~2.6.7]{wellner2013weak}, together with symmetrization
and the entropy-integral bound for empirical processes, yields
\[
\E Z_n
\le
C\sqrt{\frac{d\log(en)}{n}}.
\]
Moreover, replacing one observation changes $Z_n$ by at most $1/n$,
since every function in $\mathcal G$ takes values in $[0,1]$.
Therefore, McDiarmid's inequality implies that, with probability at
least $1-\delta$,
\[
Z_n
\le
\E Z_n+\sqrt{\frac{\log(1/\delta)}{2n}}
\le
C\sqrt{\frac{d\log(en)+\log(1/\delta)}{n}}.
\]

\paragraph{Step 3: Population slope and update bounds.} Define
\[
\lambda_\B:=\inf_{\theta\in\B}\inf_{\alpha\in I}\partial_\alpha G(\theta,\alpha),
\qquad
L_\B:=\sup_{\theta\in\B}\sup_{\alpha\in I}\|\partial_\alpha F(\theta,\alpha)\|,
\]
and
\[
M_\B:=\sup_{\theta\in\B}\sup_{\alpha,\alpha'\in(0,1)}
\|F(\theta,\alpha)-F(\theta,\alpha')\|.
\]
We show that $\lambda_\B>0$ and bound $L_\B/\lambda_\B$ and $M_\B$
independently of $d$. A direct calculation gives
\[
\partial_\alpha\Psi(y,\theta,\alpha)
=
\frac{
\Psi(y,\theta,\alpha)
\bigl(1-\Psi(y,\theta,\alpha)\bigr)
}{
\alpha(1-\alpha)
},
\]
and hence
\[
\partial_\alpha G(\theta,\alpha)
=
\E\left[
\frac{
\Psi(Y,\theta,\alpha)
\bigl(1-\Psi(Y,\theta,\alpha)\bigr)
}{
\alpha(1-\alpha)
}
\right].
\]
For
\[
s(y,\theta,\alpha)
:=
\log\frac{\alpha}{1-\alpha}
+
\frac{2\langle y,\theta\rangle}{\sigma^2},
\]
we have
\[
\Psi(y,\theta,\alpha)\bigl(1-\Psi(y,\theta,\alpha)\bigr)
=
\frac{1}{2+e^{s(y,\theta,\alpha)}+e^{-s(y,\theta,\alpha)}}.
\]
For $\alpha\in I$,
\[
\left|\log\frac{\alpha}{1-\alpha}\right|
\le
\log\frac{1-\alpha_0}{\alpha_0}.
\]
Consequently, on the event
$|\langle Y,\theta\rangle|\le 2\|\theta^\ast\|^2$,
\begin{align*}
\Psi(Y,\theta,\alpha)\bigl(1-\Psi(Y,\theta,\alpha)\bigr)
&=
\frac{1}{2+\exp(s(Y,\theta,\alpha))+\exp(-s(Y,\theta,\alpha))} \\
&\ge
\frac14\exp\bigl(-|s(Y,\theta,\alpha)|\bigr) \\
&\ge
\frac14
\exp\left(
-\left|\log\frac{\alpha}{1-\alpha}\right|
-\frac{2|\langle Y,\theta\rangle|}{\sigma^2}
\right) \\
&\ge
\frac14
\exp\left(
-\log\frac{1-\alpha_0}{\alpha_0}
-4\frac{\|\theta^\ast\|^2}{\sigma^2}
\right) \\
&=
\frac{\alpha_0}{4(1-\alpha_0)}
\exp\left(
-4\frac{\|\theta^\ast\|^2}{\sigma^2}
\right) \\
&\ge
c\exp\left(
-C\frac{\|\theta^\ast\|^2}{\sigma^2}
\right),
\end{align*}
where $c,C>0$ depend only on $\alpha_0$. It follows that
\[
\partial_\alpha G(\theta,\alpha)
\ge
c\exp\left(
-C\frac{\|\theta^\ast\|^2}{\sigma^2}
\right)
\Pr\left(
|\langle Y,\theta\rangle|
\le 2\|\theta^\ast\|^2
\right).
\]

Write $Y=S\theta^\ast+\sigma Z$, where $S$ is Rademacher and $Z\sim N(0,I_d)$ is independent of $S$.
For fixed $\theta$,
\[
\langle Y,\theta\rangle
\sim
\frac12
N\left(
\langle\theta^\ast,\theta\rangle,
\sigma^2\|\theta\|^2
\right)
+
\frac12
N\left(
-\langle\theta^\ast,\theta\rangle,
\sigma^2\|\theta\|^2
\right).
\]
By symmetry, if $X_\theta
\sim
N\left(
\langle\theta^\ast,\theta\rangle,
\sigma^2\|\theta\|^2
\right)$, then
\[
\Pr\left(
|\langle Y,\theta\rangle|
\le 2\|\theta^\ast\|^2
\right)
=
\Pr\left(
|X_\theta|
\le 2\|\theta^\ast\|^2
\right).
\]
Moreover,
\[
\Pr\left(
|X_\theta|
\le 2\|\theta^\ast\|^2
\right)
\ge
\Pr\left(
0\le X_\theta\le 2\|\theta^\ast\|^2
\right).
\]
Since $\theta\in\B$,
\[
\|\theta\|
\le \frac54\|\theta^\ast\|,
\qquad
\frac34\|\theta^\ast\|^2
\le
\langle\theta^\ast,\theta\rangle
\le
\frac54\|\theta^\ast\|^2.
\]
Therefore,
\[
\Pr\left(
|\langle Y,\theta\rangle|
\le 2\|\theta^\ast\|^2
\right)
\ge
c\min\left\{
\frac{\|\theta^\ast\|}{\sigma},1
\right\},
\]
which follows from the following inequality
\[
\Phi_{\rm std}(x)-\Phi_{\rm std}(-x)
\ge c\min\{x,1\},
\qquad x\ge 0,
\]
where $\Phi_{\rm std}(\cdot)$ denotes the standard normal cumulative distribution function. Thus, we obtain
\[
\lambda_\B
\ge
c
\min\left\{
\frac{\|\theta^\ast\|}{\sigma},1
\right\}
\exp\left(
-C\frac{\|\theta^\ast\|^2}{\sigma^2}
\right).
\]

Next, we upper bound $L_\B$. For
$\theta\in\B$ and $\alpha\in I$, let
\[
h_\alpha(t)
:=
\partial_\alpha
\ell\left(
\log\frac{\alpha}{1-\alpha}
+
\frac{2t}{\sigma^2}
\right).
\]
Because $\alpha$ is bounded away from zero and one, and the first two derivatives of the logistic function are uniformly bounded,
\[
\|h_\alpha\|_\infty\le C,
\qquad
\|h_\alpha'\|_\infty\le \frac{C}{\sigma^2}.
\]
Consequently, dominated convergence implies
\[
\partial_\alpha F(\theta,\alpha)
=
2\E\left[
Yh_\alpha(\langle Y,\theta\rangle)
\right].
\]
Using $Y=S\theta^\ast+\sigma Z$, the signal contribution satisfies
\[
\left\|
\E\left[
S\theta^\ast
h_\alpha\left(
S\langle\theta^\ast,\theta\rangle
+
\sigma\langle Z,\theta\rangle
\right)
\right]
\right\|
\le
C\|\theta^\ast\|.
\]
For the Gaussian contribution, conditioning on $S$ and applying
Stein's lemma gives
\[
\E_Z\left[
Z
h_\alpha\left(
S\langle\theta^\ast,\theta\rangle
+
\sigma\langle Z,\theta\rangle
\right)
\right]
=
\sigma\theta\,
\E_Z\left[
h_\alpha'\left(
S\langle\theta^\ast,\theta\rangle
+
\sigma\langle Z,\theta\rangle
\right)
\right].
\]
Therefore,
\[
\left\|
\sigma\E\left[
Z
h_\alpha\left(
S\langle\theta^\ast,\theta\rangle
+
\sigma\langle Z,\theta\rangle
\right)
\right]
\right\|
\le
C\|\theta\|.
\]
Since $\theta\in\B$ implies
$\|\theta\|\le 5\|\theta^\ast\|/4$, we conclude that
\[
L_\B
\le
C\|\theta^\ast\|.
\]
Combining this bound with
\[
\lambda_\B
\ge
c
\min\left\{
\frac{\|\theta^\ast\|}{\sigma},1
\right\}
\exp\left(
-C\frac{\|\theta^\ast\|^2}{\sigma^2}
\right),
\]
we obtain
\[
\frac{L_\B}{\lambda_\B}
\le
C
\max\left\{
\sigma,\|\theta^\ast\|
\right\}
\exp\left(
C\frac{\|\theta^\ast\|^2}{\sigma^2}
\right).
\]

Finally, write
\[
F(\theta,\alpha)
=
\E\left[
Yq_\alpha(\langle Y,\theta\rangle)
\right],
\]
where
\[
q_\alpha(t)
:=
2\ell\left(
\log\frac{\alpha}{1-\alpha}
+
\frac{2t}{\sigma^2}
\right)-1.
\]
Uniformly over $\alpha\in(0,1)$,
\[
\|q_\alpha\|_\infty\le 1,
\qquad
\|q_\alpha'\|_\infty\le \frac{C}{\sigma^2}.
\]
The signal component of $F(\theta,\alpha)$ is therefore bounded by
$\|\theta^\ast\|$, while Stein's lemma bounds the Gaussian component
by $C\|\theta\|$. Since $\theta\in\B$, this gives
\[
\sup_{\theta\in\B}
\sup_{\alpha\in(0,1)}
\|F(\theta,\alpha)\|
\le
C\|\theta^\ast\|,
\]
which implies that
\[M_\B \leq 2C\|\theta^\ast\|.\]

\paragraph{Step 4: Stability of the empirical roots.}
Introduce a ghost sample $Y_1',\ldots,Y_n'\stackrel{\mathrm{i.i.d.}}{\sim}Q^\ast$,
independent of $Y_1,\ldots,Y_n$. Let $G_n'$, $\alpha_n'$, and $Z_n'$
denote the corresponding empirical calibration map, Sinkhorn-corrected
weight, and empirical-process deviation. Let $\E'$ denote expectation
only with respect to the ghost sample. Define
\[
H_n(\theta):=F(\theta,\alpha_n(\theta)),
\qquad
H_n'(\theta):=F(\theta,\alpha_n'(\theta)).
\]
Let $\mathcal E_\alpha$ and $\mathcal E_\alpha'$ be the events from
\Cref{lem:alpha_n_bound} for the original and ghost samples,
respectively. On
$\mathcal E_\alpha\cap\mathcal E_\alpha'$, both
$\alpha_n(\theta)$ and $\alpha_n'(\theta)$ belong to $I$ for every
$\theta\in\B$. By the mean value theorem,
\[
\lambda_\B
|\alpha_n(\theta)-\alpha_n'(\theta)|
\le
|G(\theta,\alpha_n(\theta))
-G(\theta,\alpha_n'(\theta))|,
\]
by the definition of $\lambda_\B$.
Using \eqref{eq:sem-empirical-calibration} for the two samples, we obtain
\[
\begin{aligned}
|G(\theta,\alpha_n(\theta))
-G(\theta,\alpha_n'(\theta))| &\leq |G(\theta,\alpha_n(\theta))
-G_n(\theta,\alpha_n(\theta))| + |G_n'(\theta,\alpha_n'(\theta))
-G(\theta,\alpha_n'(\theta))|\\ 
&\leq
Z_n+Z_n'.
\end{aligned}
\]
Consequently,
\[
\sup_{\theta\in\B}
|\alpha_n(\theta)-\alpha_n'(\theta)|
\le
\frac{Z_n+Z_n'}{\lambda_\B},
\]
and therefore
\[
\sup_{\theta\in\B}
\|H_n(\theta)-H_n'(\theta)\|
\le
\frac{L_\B}{\lambda_\B}
\left(Z_n+Z_n'\right)
\]
on $\mathcal E_\alpha\cap\mathcal E_\alpha'$. 
\paragraph{Step 5: Ghost-sample centering.}
Since $H_n'$ has the same distribution as $H_n$, $\E' H_n'(\theta)=\E H_n(\theta)$. Furthermore, for each fixed realization of the original sample, $H_n(\theta)$ is constant with respect to $\E'$. Hence $H_n(\theta)-\E H_n(\theta)=\E'[H_n(\theta)-H_n'(\theta)]$, and Jensen's inequality gives
\[
\begin{aligned}
\sup_{\theta\in\B}
\|H_n(\theta)-\E H_n(\theta)\|
&=
\sup_{\theta\in\B}
\left\|
\E'\left[
H_n(\theta)-H_n'(\theta)
\right]
\right\| \\
&\le
\E'
\sup_{\theta\in\B}
\|H_n(\theta)-H_n'(\theta)\|.
\end{aligned}
\]

Now work on the event $\mathcal E_\alpha$ for the original sample.
Splitting according to the ghost event $\mathcal E_\alpha'$ yields
\[
\begin{aligned}
\E'
\sup_{\theta\in\B}
\|H_n(\theta)-H_n'(\theta)\|
&\le
\frac{L_\B}{\lambda_\B}
\E'\left[
\left(Z_n+Z_n'\right)
\mathbf 1_{\mathcal E_\alpha'}
\right]
+
M_\B\Pr'((\mathcal E_\alpha')^c) \\
&\le
\frac{L_\B}{\lambda_\B}
\left(Z_n+\E Z_n\right)
+
M_\B\Pr(\mathcal E_\alpha^c),
\end{aligned}
\]
By \Cref{lem:alpha_n_bound},
\[
\Pr(\mathcal E_\alpha^c)\le Cn^{-2}.
\]
It follows that, on $\mathcal E_\alpha$,
\[
\sup_{\theta\in\B}
\|H_n(\theta)-\E H_n(\theta)\|
\le
\frac{L_\B}{\lambda_\B}
\left(Z_n+\E Z_n\right)
+
CM_\B n^{-2}.
\]

By Step~2, the event
$Z_n
\le
C\sqrt{\frac{d\log(en)+\log(1/\delta)}{n}}$ has
probability at least $1-\delta$. On its intersection with $\mathcal E_\alpha$ we have
\[
\sup_{\theta\in\B}
\|H_n(\theta)-\E H_n(\theta)\|
\le
\frac{L_\B}{\lambda_\B}
\left(
C\sqrt{\frac{d\log(en)+\log(1/\delta)}{n}}
+
C\sqrt{\frac{d\log(en)}{n}}
\right)
+
CM_\B n^{-2}.
\]
The bounds from Step~3 absorb the $n^{-2}$ remainder into the displayed $n^{-1/2}$ scale.

We conclude that with probability at least
$1-\delta-Cn^{-2}$,
\[
\sup_{\theta\in\B}
\left\|
F(\theta,\alpha_n(\theta))
-
\E F(\theta,\alpha_n(\theta))
\right\|
\le
C
\max\left\{
\sigma,\|\theta^\ast\|
\right\}
\exp\left(
C\frac{\|\theta^\ast\|^2}{\sigma^2}
\right)
\sqrt{
\frac{d\log(en)+\log(1/\delta)}{n}
}.
\]
\end{proof}
\end{document}